\documentclass[11pt]{amsart}
\usepackage{amsmath,amssymb,mathtools}
\usepackage{stmaryrd} 
\usepackage[T1]{fontenc}
\usepackage{lmodern,textcomp}
\usepackage{color}
\usepackage{graphicx}
\usepackage{amsthm}
\usepackage[all]{xy}
\usepackage{caption}
\usepackage{enumitem}
\usepackage{hyperref}

\usepackage{soul}

\usepackage{tikz}
\usepackage{tikz-cd}

\usepackage{fullpage}
 \usepackage{setspace}
\usepackage{cleveref}

\def\cH{{\mathcal{H}}}
\def\cM{{\mathcal{M}}}

\def\cL{{\mathcal{L}}}

\def\cO{\mathcal{O}}

\def\oC{\overline{\mathcal{C}}}
\def\oM{\overline{\mathcal{M}}}

\def\oH{\overline{\mathcal{H}}}
\def\oSQ{\overline{\mathcal{SQ}}}
\def\SQ{{\mathcal{SQ}}}
\def\oGamma{{\overline{\Gamma}}}

\def\CC{\mathbb{C}}
\def\ZZ{\mathbb{Z}}
\def\NN{\mathbb{N}}
\def\QQ{\mathbb{Q}}
\def\PP{\mathbb{P}}

\def\res{{\rm res}}
\def\Aut{{\rm Aut}}

\theoremstyle{plain}
\newtheorem{theorem}{Theorem}[section]
\newtheorem{lemma}[theorem]{Lemma}
\newtheorem{proposition}[theorem]{Proposition}
\newtheorem{corollary}[theorem]{Corollary}
\newtheorem{conjecture}[theorem]{Conjecture}

\theoremstyle{definition}
\newtheorem{definition}[theorem]{Definition}

\newtheorem{example}[theorem]{Example}
\newtheorem{remark}[theorem]{Remark}

\begin{document}
\title{Spin volumes of minimal strata \\ and Chiodo integrals}

\author{Andrei Bud}
\address{Institut de Mathématiques de Toulouse, Université de Toulouse,
118 Route de Narbonne, France}
\email{andreibud95@protonmail.com}

\author{Georgios Politopoulos}
\address{HUN-REN Alfréd Rényi Institute of Mathematics, 1053 Budapest, Reáltanoda u. 13-15}
\email{geopolito@renyi.hu}

\author{Stijn Velstra}

\begin{abstract}
We derive a closed formula for the Masur-Veech volumes of spin-parity components of the stratum of abelian differentials with a single zero of maximal order. Our approach is based on the intersection theory of the virtual subcone of spin-parity squares inside the Hodge bundle, recently developed in \cite{HolPolSau}. This is a spin refinement of the result of \cite{Sau} and agrees with the lattice-point counting technique of \cite{CheMolSauZag}. 
We further apply this theory to compute spin counterparts of virtual volumes considered in \cite{Sauvaget-virtual}. These virtual volumes are related to area Siegel-Veech constants, cf. \cite{vanSau}, and we are able to compute these invariants component-wise.   
In addition, by comparing our method for non-parity squares with the classical result of Sauvaget, we compute specific Chiodo integrals.
\end{abstract}
\maketitle

\setcounter{tocdepth}{1}
\tableofcontents

\section*{Introduction}

\subsection*{Setting and motivation} 

Starting with the influential paper \cite{Kon-Zor} of Kontsevich and Zorich, strata of differentials have been a central object of study in Algebraic Geometry and Teichm\"uller Dynamics. In recent years, substantial progress have been made in connecting these two viewpoints: 

\begin{itemize}
    \item intersection-theoretic methods have been used to compute dynamical invariants of strata (such as  Masur-Veech volumes, area Siegel-Veech constants, and Lyapunov exponents, cf. \cite{Sau}, \cite{CheMolSauZag}, \cite{vanSau} and \cite{ChenMol}, among many others)
    \item the flat-surface interpretation of differentials is used to determine the number of connected components, and to construct compactifications with mild singularities, cf. \cite{Kon-Zor} and \cite{BaiCheGenGruMol}. 
\end{itemize} 

In fact many of these invariants can be defined component-wise, while the invariant of the entire stratum is additive. One key observation is that, if we weight them by the spin parity of the component, this new difference of invariants has a meaningful numeric and intersection-theoretic interpretation. This perspective already appears in the literature in several recent papers, studying spin DR-cycles \cite{HolPolSau}, spin orbifold Euler characteristics of strata \cite{CosSauSch}, the spin Masur-Veech volume \cite{CheMolSauZag}, cylinder counts \cite{vanSau} and the BKP hierarchy \cite{KloVel}. \\

\noindent The goal of this paper is to compute the Masur-Veech volume of the minimal stratum, with its components weighted by the spin parity. We will use intersection-theoretical methods to find recurrence formulas that let us compute volumes via coefficient extraction, see Theorem \ref{thm:maintheorem}. We will apply the same method in order to compute the area Siegel-Veech constant for all its components.

Let $g$ and $n$ be two non-negative integers satisfying $2g-2 + n > 0$ and denote by $\mathbb{P}\cH_{g,n}$ the projectivized Hodge bundle over the moduli space $\cM_{g,n}$ of smooth genus $g$ curves with $n$ marked points. This is the projective bundle whose fiber over a point $[C,x_1,\ldots, x_n]$ is $\mathbb{P}H^0(C,\omega_C)$, i.e. $\mathbb{P}\cH_{g,n}$ parametrizes abelian differentials up to scaling. Given a nonnegative partition $\mu = (\mu_1, \ldots, \mu_n)$ of $2g-2 +n$, one defines the stratum  $\mathbb{P}\cH_{g,n}(\mu) \subseteq \mathbb{P}\cH_{g,n}$ parametrizing (classes of) Abelian differentials whose zero divisor is supported at the marked points $x_1, \ldots, x_n$ with vanishing orders $\mu_1 -1, \ldots, \mu_n-1$, respectively; that is, a point $[C,x_1,\ldots, x_n, \varphi] \in \mathbb{P}\cH_{g,n}(\mu)$ satisfies 
\[ \textrm{div}(\varphi) = \sum_{i=1}^n(\mu_i-1)x_i.\]

 From a geometric perspective, a key feature is that, for any partition $\mu$, the stratum $\mathbb{P}\cH_{g,n}(\mu)$ is a smooth orbifold of dimension $2g-2 + n$, see \cite{Polis}, and has at most three connected components, distinguished by hyperellipticity and spin parity, see \cite{Kon-Zor}. For many geometric and enumerative questions, however, it is essential to work with a suitable compactification. In our case, there exists a natural extension $\mathbb{P}\oH_{g,n}$ of the Hodge bundle over $\oM_{g,n}$ and we want to understand the closure of $\mathbb{P}\cH_{g,n}(\mu)$ inside this space. This question was fully answered in \cite{BaiCheGenGruMol-IVC}, where the Incidence Variety Compactification $\mathbb{P}\oH_{g,n}(\mu)$ is explicitly described. This was further improved in \cite{BaiCheGenGruMol}, where the authors constructed a smooth orbifold compactification, whose boundary is a normal crossing divisor. These compactifications provide a framework to obtain recursive formulas for intersection-theoretic computations on strata of differentials. 

\subsection*{Spin parity} Let $[C,x_1,\ldots, x_n, \varphi]$ be an element of a stratum $\mathbb{P}\cH_{g,n}(\mu)$ and assume all entries of $\mu$ are odd. Then, the line bundle
\[
\theta \coloneqq 
\cO_C\!\left(\sum_{i=1}^n \frac{\mu_i-1}{2}x_i\right)
\]
is a theta characteristic on $C$ (i.e. satisfies $\theta^{\otimes 2} \cong \omega_C)$. It is a classical result of \cite{Mum1, Ati}, that the parity $h^0(C,\theta)\pmod 2$ is deformation invariant in families of smooth curves. Consequently, the stratum $\mathbb{P}\cH_{g,n}(\mu)$ has at least two connected components, distinguished by the parity of the theta characteristic. We denote by $\mathbb{P}\cH_{g,n}(\mu)^+$ and $\mathbb{P}\cH_{g,n}(\mu)^-$ the components of $\mathbb{P}\cH_{g,n}(\mu)$ of even and respectively odd spin parity.

While spin parity provides a clear geometric distinction between the two components, separating them from an intersection-theoretic perspective proved to be more subtle. However, the problem of computing the class of the individual components was recently solved in \cite{HolPolSau}, where the authors extended the method of \cite{Sau2} to the cone of spin sections. Their work provides the framework for spin-parity intersection theory, and plays a central role in the present paper. 

Moreover, when the partition $\mu$ is either $(2g-1)$ or $(g,g)$, the corresponding stratum of Abelian differentials admits a component parametrizing only hyperelliptic curves. As before, we can associate to an element $[C,x_1,\varphi]\in \mathbb{P}\cH_{g,1}(2g-1)^{\rm hyp} $ a theta characteristic. Consequently, the hyperelliptic component contributes to the difference $[\mathbb{P}\cH_{g,n}(\mu)^{+}]-[\mathbb{P}\cH_{g,n}(\mu)^{-}]$, and therefore will yield a non-trivial contribution to the corresponding Masur-Veech volume. On the other hand, the geometry of the hyperelliptic component is significantly simpler, and its Masur-Veech volume has already been computed in  \cite[Section 6]{CheMolSauZag} and indirectly in \cite{AthEskZor}.  

\subsection*{Volumes}
Let $g,n$ satisfy $2g-2+n >0$, and let $\mu$ be a length $n$ partition of $2g-2 + n$. We consider the subcone $\cH_{g,n}(\mu)$ of $\cH_{g,n}$, which parametrizes Abelian differentials with order profile $\mu-1$ at the $n$ marked points, or equivalently, flat surfaces with $n$ marked conic singularities of angles $2\pi \mu$, see \cite{Chen} and \cite{Zor}.

Let $[C,x_1,\ldots, x_n,\varphi]$ be an element of $\cH_{g,n}(\mu)$ and consider the relative cohomology group 
\[ H \coloneqq H^1(C, \{x_1,\ldots, x_n\}; \mathbb{C}).  \]

By choosing a basis of the relative homology group, and then integrating the differential $\varphi$ along these $2g+n-1$ paths, we realize $H$ as a system of local coordinates of $\cH_{g,n}(\mu)$, called the period coordinates. The upshot is that the relative cohomology group $H$ contains a natural integer lattice $H^1(C, \{x_1,\ldots, x_n\}; \mathbb{Z}\oplus i\mathbb{Z})$, allowing us to consider a linear volume form $d\nu_{MV}$, normalized in such a way that the fundamental (hypercubic) domain has volume equal to $1$. 

We consider two subloci $\cH^{\leq 1}_{g,n}(\mu)$ and  $\cH^{1}_{g,n}(\mu)$ parametrizing flat surfaces satisfying the area condition 
\[ {\rm Area}(C,\varphi) \coloneqq \frac{i}{2}\int_C\varphi\wedge\overline{\varphi} \leq 1, \ \textrm{and} \ {\rm Area}(C,\varphi) = 1, \ \textrm{respectively.} \]
We can obtain a volume form on $\cH^{1}_{g,n}(\mu)$ via desingularisation, see \cite{EskKonZor2014}. The volume with respect to this form, called the Masur-Veech volume of $\cH_{g,n}^{1}(\mu)$, is finite, see \cite{Mas1982} and \cite{Vee1982}. It was first computed in \cite{EskOko} as limits of Hurwitz-type numbers by relating the lattice points of the strata to ramified torus covers $C \ni x \mapsto \int_{x_1}^x \omega \in \frac{\CC}{\ZZ\oplus i\ZZ}$. For the relation of these Masur-Veech volumes to the statistics of saddle connections, see for example \cite{EskMasZor}. 

Finally, a volume form over $\mathbb{P}\cH_{g,n}(\mu)$ can be obtained via disintegration of the form $d\nu_{MV}$, (see \cite[Section 2.2]{Sau}). The total volume of $ \mathbb{P}\cH_{g,n}(\mu)$ is thus obtained using the projection map $p\colon \cH_{g,n}(\mu) \rightarrow \mathbb{P}\cH_{g,n}(\mu)$. We consider the volume form of $\mathbb{P}\cH_{g,n}(\mu)$ respecting the equality
\[ \textrm{Vol}(U) = \textrm{Vol}(p^{-1}U\cap \cH^{\leq 1}_{g,n}(\mu))\]
for every open domain $U$ of $\mathbb{P}\cH_{g,n}(\mu)$. It is a result of \cite{CheMolSauZag} that the volume obtained in this way can be computed via intersection theory on (a compactification of) the stratum $\mathbb{P}\cH_{g,n}(\mu)$. 

\subsection*{Intersection Theory}
For $\mu = (2g-1)$ this Masur-Veech volume was compared to an integral of tautological bundles on the incidence variety compactification $\PP\oH_{g,1}(2g-1)$ of the projectivized stratum; see \cite{Sau}. This was based on the existence of a desingularisation with a good (in the sense of \cite{Mum2}) metric extending the metric on the tautological bundle over $\PP\cH_{g,n}(\mu)$ induced by the area. In \cite{CosMolZac} a good enough canonical metric extension was obtained using the moduli spaces of multi-scale differentials of \cite{BaiCheGenGruMol}. This leads in \cite{Sau} to the volume formula for the minimal partition $\mu = (2g-1)$:  

\[ \textrm{Vol}(\cH^{1}_{g,1}(2g-1)) = \frac{2(2i\pi)^{2g}}{(2g-2)!}\int_{\PP\oH_{g,1}(2g-1)}\xi^{2g-2} \cdot \psi_1 = \frac{2(2i\pi)^{2g}}{(2g-2)!}\int_{\PP\oH_{g,1}(2g-1)}\xi^{2g-1},\]
where $\xi$ denotes the first Chern class of the universal line bundle $\cO(1)$ on $\PP\oH_{g,n}$ and the $\psi$-class is the pullback of the one on $\overline{\cM}_{g,1}$. This intersection-theoretic formula was computed recursively in \cite{Sau} using an explicit formula for the class of $\PP\oH_{g,n}(\mu)$ in $\PP\oH_{g,n}$. This produces the relation~\ref{eq:ag-formula} below, which was analysed in \textit{op.cit.} to prove conjectures of \cite{EskZor} on volume asymptotics in case $\mu = (2g-1)$. 

 We denote $a_g \coloneqq \int_{\PP\oH_{g,n}(2g-1)}\xi^{2g-1}$ and consider the formal series $\mathcal{S}(t) = \frac{t/2}{\sinh{t/2}}$ and
\[ \mathcal{F}(t) = 1 + \sum_{g>0}(2g-1)a_gt^{2g}.\]
The main result of \cite{Sau} is the coefficient extraction formula
\begin{align*} [t^{2g}]\ \mathcal{S}(t) = \frac{1}{(2g)!}[t^{2g}]\left(\mathcal{F}(t)^{2g}\right),
\end{align*}
where the square brackets indicate the coefficient of $t^{2g}$. We can invert this result and write it in terms of Bernoulli numbers; we define 
$$\varpi(x) \coloneqq \sum_{n\geq2} \frac{B_n}{n} x^n$$ 
and we have
\begin{align*}
[t^{2g}]\ \mathcal{F}(t) 
&= \frac{1}{1-2g}[t^{2g}]\bigg(e^{ -\varpi(t) + 2\ \varpi(t/2)}\bigg)^{(1-2g)}. \\
\end{align*}

\noindent This follows from a recursion that is solved in terms of linear Hodge integrals $b_g \coloneqq \int_{\oM_{g,1}} \psi_1^{2g-2}\lambda_g$, namely
\begin{equation}\label{eq:ag-formula}(2g-1)a_g = \sum_{\substack{n\geq1,\, g_1,\dots,g_n\geq 1\\ g_1+\dots+ g_n = g}}\frac{(1-2g)^{n-1}}{n!}\prod_{i=1}^n(2g_i-1)!(-1)^{g_i}b_{g_i}.\end{equation} 
The power series result of \cite{Sau} then uses that the numbers $(-1)^gb_g$ are the non-trivial coefficients of $\frac{t/2}{\sinh{t/2}}$, see \cite{FabPan}.\\

\noindent Based on this initial case, the volumes for general length $n$ partition are expressed as integrals in \cite{CheMolSauZag}, i.e. 
\[ \textrm{Vol}(\cH^{1}_{g,n}(\mu)) = \frac{2(2i\pi)^{2g}}{(2g-3+n)!}\int_{\PP\oH_{g,n}(\mu)}\xi^{2g-2} \cdot \prod_{i=1}^n \psi_i = \frac{2(2i\pi)^{2g}}{(2g-3+n)!}\int_{\PP\oH_{g,n}(\mu)}\xi^{2g-1} \cdot \prod_{i\neq j} \psi_i.\]

\noindent This is done by showing that the volumes and the integrals satisfy the same recursion on $n$, and, as a consequence, they obtain large genus asymptotics for Masur-Veech volumes, proving conjectures of Eskin-Zorich.\\

\noindent For other occurrences of volumes of moduli spaces expressed via enumerative geometry, see e.g. \cite{Mir}, \cite{Wol} for moduli of bordered hyperbolic surfaces \cite{CheMolSau}, \cite{Gou16}, \cite{DurGouYak25} for quadratic differentials, \cite{Nor}, \cite{StaWit} for bordered supercurves, or \cite{AnaNor},\cite{Sau3} for cone angles.

\subsection*{Statement of the main result.}

Our goal is to compute via intersection-theory, the volume difference of the components ${\rm Vol}^\pm(2g-1):= {\rm Vol}\left(\cH_{g,1}^1(2g-1)^+\right) - {\rm Vol}\left(\cH_{g,1}^1(2g-1)^-\right)$. The intersection-theoretic formula of \cite{Sau} holds componentwise and hence we have 
\[ \textrm{Vol}^{\pm}(2g-1) = \frac{2(2i\pi)^{2g}}{(2g-2)!}\Bigl(\int_{\PP\oH_{g,1}(2g-1)^+}\xi^{2g-1} -\int_{\PP\oH_{g,1}(2g-1)^-}\xi^{2g-1} \Bigr).\]
We define a scaling of the volume
\[ a_g^{\pm} \coloneqq \int_{\PP\cH_{g,1}(2g-1)^+}\xi^{2g-1} - \int_{\PP\cH_{g,1}(2g-1)^-}\xi^{2g-1}.\]
 Similarly to \cite{Sau}, these spin volumes can then be computed via coefficient extraction. We define the formal series
\begin{align*}
    \mathcal{S}^{\pm}(t) \coloneqq \frac{t/2 \cdot \cosh{t/2}}{\sinh{t/2}} \ \ \textrm{and} \ \
    \mathcal{F}^{\pm}(t)\coloneqq 1 + \sum_{g > 0}(2g-1)a^{\pm}_gt^{2g}.
\end{align*}

 Our main result, the proof of which is concluded in Section \ref{sec: solve recursion}, states the following.
\begin{theorem}[Main Theorem]\label{thm:maintheorem}
For $g\geq 1$, 
\[(2g-1)a_g^\pm = \sum_{n=1}^g\sum_{\substack{\, g_1,\dots,g_n\geq 1\\ g_1+\dots+ g_n = g}}\frac{(1-2g)^{n-1}}{n!}\prod_{i=1}^n(2g_i-1)!(-1)^{g_i}b^\pm_{g_i},\]
where $b_g^\pm \coloneqq-\frac{2^{g-1}}{2^{2g-1}-1} b_g.$ Equivalently, 
\begin{align*}
[t^{2g}]\ \mathcal{F}^\pm(t) 
&= \frac{1}{1-2g}[t^{2g}]\bigg(e^{(1-2g)\ \varpi(t/\sqrt2)}\bigg),\\
 2^{-g} [t^{2g}]\ \mathcal{S}^\pm(t) &= \frac{1}{(2g)!}[t^{2g}]\Bigl( \mathcal{F}^{\pm}(t)\Bigr)^{2g},
\end{align*}
where the latter two expressions are inverse to each other. 
\end{theorem}

 This formula appeared in the shape presented here in \cite[Section 6.3]{CosSauSch} as an ingredient to compute the Euler characteristic of the spin-parity components of the minimal strata. Another consequence of Theorem \ref{thm:maintheorem} and \cite{Sau} is that we can compute the Masur-Veech volume of each of the three components of $\mathbb{P}\cH_{g,1}(2g-1)$. Indeed, the contribution coming from the hyperelliptic component is explicitly computed in \cite[Section 6]{CheMolSauZag} as 
\[ \textrm{Vol}(\cH^1_{g,1}(2g-1)^{\textrm{hyp}}) = \frac{2}{(2g+1)!}\frac{(2g-3)!!}{(2g-2)!!}\pi^{2g}. \]

\subsection*{Outline of the proof.}
Our proof adapts the methods of \cite{Sau} and \cite{Sauvaget-virtual} to the spin-parity case. In order to do this we use the intersection theory of spin-parity squares (instead of the Hodge bundle) developed in \cite{HolPolSau}, to compute $a_g^{\pm}$. We view the locus $\PP\cH_{g,1}(2g-1)^\pm$ as a sublocus of the cone of squares, and we compute its class via the inductive procedure of \textit{op.cit.} which decreases the order of the zero by $2$ at each step. This process will give us the spin Masur-Veech volume in terms of integrals against weighted Segre classes, with some extra boundary contribution that we need to account for. While in \cite{Sau} the boundary contribution is immediately computed, the case of spin-parity squares proves to be more cumbersome. This contribution is a sum of Hurwitz integrals on residueless strata in genus $0$, and we deduce and solve recurrence formulas to account for these terms, see Theorem \ref{thm:Hurformula}. The Segre classes appearing in \textit{loc.cit.} are expressed in terms of linear Hodge integrals via Mumford's formula. For the cone of squares, the weighted Segre classes are related to spin-Chiodo classes as in \cite{HolPolSau}, and thus will be expressed in terms of double Hodge integrals. The proof concludes with adding together the principal and boundary contribution to $a_g^\pm$ and analysing the combinatorics of the formula.

\subsection*{Cylinder counts and area Siegel-Veech constants} 
Let $[C,x_1,\ldots, x_n, \varphi]$ be an element of $\mathcal{H}_g(\mu)$, viewed as a flat surface. Then the union of (parallel) closed geodesic of a given homotopy type form a cylinder $Z$, whose width $w(Z)$ is the length of any of these geodesics. One defines the area weighted counting function as
\[
N_0(C,\varphi,L)
\coloneqq
\sum_{w(Z)<L}
\frac{\mathrm{area}(Z)}{\mathrm{area}(C)}.
\]
For any connected component $X \subseteq \mathcal{H}_g(\mu)$, one has the quadratic asymptotic growth
\[
N_0(C,\varphi,L) \sim c_{\textrm{area}}(X)\,\pi L^2,
\]
for almost all abelian differentials $[C,x_1,\ldots, x_n, \varphi]  \in X$, where $c_{\textrm{area}}(X)$ is a constant depending on $X$, called the area Siegel--Veech constant, cf. \cite{Vee98,EskMasZor}. By \cite[Theorem 1.4]{CheMolSauZag}, these invariants can be computed via intersection theory as
\[ c_{\textrm{area}}(X) = -\frac{1}{4\pi^2}\frac{\int_X \delta_0\cdot \xi^{2g-2}\prod_{i=2}^n\psi_i}{\int_X \xi^{2g-1}\prod_{i=2}^n\psi_i},\]
where $\delta_0 $ is the divisor class of the locus of curves with a non-separable node.

 In the case $n=1$, these constants $c_{\textrm{area}}(\mathcal{H}_g(2g-1)) = -\frac{1}{4\pi^2}\frac{{\int_{\mathbb{P}\overline{\mathcal{H}}_g(2g-1)} \delta_0\cdot \xi^{2g-2}}}{a_g}$ for the whole stratum can be computed via \cite[Theorem 1.6]{Sau}. The constants for the hyperelliptic components are implicitly computed in \cite[Lemma 1.1 and Theorem 3]{EskKonZor2014}, giving: 
 \[
c_{\mathrm{area}}\big(\mathcal H_g^{\mathrm{hyp}}(2g-1)\big)
=
\frac{(2g+1)g}{(2g-1)\pi^2} \ \ \textrm{and} \ \ c_{\mathrm{area}}\big(\mathcal H_g^{\mathrm{hyp}}(g,g)\big)
=
\frac{(2g+1)(g+1)}{2g\,\pi^2}.
\]
 When all entries of $\mu$ are odd the area Siegel-Veech constants admit parity refinements, as studied in \cite{vanSau}.
In order to compute these refined invariants, we use the same strategy as for our volume computation and obtain a recursion for 
\[ d_g^\pm \coloneqq 
\int_{\mathbb{P}\overline{\mathcal{H}}_g(2g-1)^+} \delta_0\cdot \xi^{2g-2} - \int_{\mathbb{P}\overline{\mathcal{H}}_g(2g-1)^-} \delta_0\cdot \xi^{2g-2}.\] 

We mention in passing that by flat geometric results of \cite[Section 5.3]{Boi}, these invariants are integrals on genuine (spin-parity induced) components of minimal strata with a pair of simple poles.
The recursion formula is made explicit in Lemma \ref{lem:dgpm-recursion}, which is the spin analogue of \cite[Theorem 3.15]{Sau}.

 As in the case of volumes, the case of the minimal stratum is an initial step in a recursion on genera and partitions of 
\[ d_g^\pm(\mu) \coloneqq 
\int_{\mathbb{P}\overline{\mathcal{H}}_g(\mu)^+} \delta_0\cdot \xi^{2g-2}\prod_{i=2}^n\psi_i - \int_{\mathbb{P}\overline{\mathcal{H}}_g(\mu)^-} \delta_0\cdot \xi^{2g-2}\prod_{i=2}^n\psi_i.\]
As such, this recursion allows for the computations of area Siegel-Veech constants for every component of Abelian strata. This was used in \cite{CheMolSauZag} to prove the large genus limit conjectures of Eskin and Zorich in \cite{EskZor}.

\subsection*{Generalization to virtual volumes}

Because of the intersection theoretical computation of Masur-Veech volumes, their definition can be extended even for strata where they do not make geometric sense. When we work with meromorphic strata, the same method of defining volumes will fail since we are working with translation surfaces of infinite area. However, in this case we can just associate a purely algebraic definition of volume, based on intersection theory. We consider a meromorphic stratum 
\[ \mathbb{P}\mathcal{H}^R_{g,1+m}(2g-1+m, \underbrace{0, \ldots, 0}_{m \ \textrm{times}}) \]
with a unique zero, $m$ simple poles and residue conditions given by a subspace $R$ of 
$$\{(r_1,\ldots, r_m) \ \mid \ r_1+\cdots+r_m=0\}.$$ 

The virtual volume of this space is defined to be the integral of the top power of $\xi$ we can consider, scaled by a constant, i.e. if $N \coloneqq \dim \mathbb{P}\mathcal{H}^R_{g,1+m}(2g-1+m, \underbrace{0, \ldots, 0}_{m \ \textrm{times}})$ then 
\[ \textrm{Vol}(2g-1+m, R) \coloneqq \frac{2(2\pi i)^{2g}}{(2g-1+m)!}\int_{\mathbb{P}\overline{\mathcal{H}}^R_{g,1+m}(2g-1+m, \underbrace{0, \ldots, 0}_{m \ \textrm{times}})} \xi^N. \]
This intersection theoretical definition of volumes appears in \cite{Sauvaget-virtual}, where many such virtual volumes are computed. The same methods that compute the spin version of the Masur-Veech volume, and of the area Siegel-Veech constant, can be employed to study (spin) virtual volumes of strata of meromorphic differentials with only simple poles. \\

\noindent We will restrict ourself to the case $m=2p$ and space of residue conditions $$R = \{ r_1 + r_2 = 0, r_3 + r_4 = 0, \ldots, r_{2p-1}+r_{2p} = 0\}.$$ 
In this setting, we have a gluing map
 $$\zeta\colon \mathbb{P}\overline{\mathcal{H}}_{g, 1+ 2p}^R(2g-1+2p, \underbrace{0, \dots, 0}_{\text{2p times}}) \rightarrow \mathbb{P}\overline{\mathcal{H}}_{g+p, 1}(2g-1 + 2p)$$
 identifying the markings in the same pair. This allows us to view the virtual volume as an intersection on a minimal stratum, generalizing the Masur-Veech volume and area Siegel-Veech constant. 

 We define 
 \[ [D_p] \coloneqq \zeta_* [\mathbb{P}\overline{\mathcal{H}}_{g, 1+ 2p}^R(2g-1+2p, \underbrace{0, \dots, 0}_{\text{2p times}})].  \]
 Up to scaling, the virtual volumes can be computed as integrals of the form
\[ d_{g,p} \coloneqq \int_{\mathbb{P}\overline{\mathcal{H}}_g(2g-1)} \xi^{2g-1-p}[D_p], \]
and we also have a spin version 
\[ d_{g,p}^\pm \coloneqq \int_{\mathbb{P}\overline{\mathcal{H}}_g(2g-1)^\pm} \xi^{2g-1-p}[D_p], \]
whose computation will be the focus of Section \ref{sec: areaSV}. We note that this definition recovers $d_g^\pm$ up to a factor of $2$; more precisely, $d_{g,1}^\pm=2d_g^\pm$. 

\noindent In \cite[Theorem 1.5]{Sauvaget-virtual}, Sauvaget considered (amongst other directions) the deformation of the volumes by pairs of simple poles with zero residue sum \[\mathcal{F}_\epsilon(t):= 1 + \sum_{g>0}(2g-1)\ \left(\sum_p d_{g,p}\ \frac{\epsilon^p}{p!}\right)\ t^{2g},\] and proved the formula \[[t^{2g}]\exp(t^2\epsilon)\ \mathcal{S}(t) = \frac{1}{(2g)!}[t^{2g}]\left(\mathcal{F}_\epsilon(t)\right)^{2g} \in\mathbb{Q}[\epsilon].\]

Projecting this formula to $\mathbb{Q}[\epsilon]/(\epsilon^2)$ one re-obtains the inverse of equation \ref{eq:ag-formula} and formula (4) of \cite[Theorem 1.6]{Sau} used to compute the area Siegel-Veech constants for the sum of the components of minimal strata.  \\

\noindent In the spin-parity case, we obtain a recursion formula for the spin virtual volumes $d_{g,p}^\pm$. As in the classical case, our formula resembles the Leibniz rule for derivation. Seeing Lemma \ref{lem:bouquetrecursion} as a $p$-iterated derivation of Lemma \ref{lem:agpm-recursion}, we obtain the following seemingly simpler corolla corollary to Theorem \ref{thm:maintheorem}, proved in Subsection \ref{sec:derivation}. 
\begin{theorem}[Spin flowerseries formula] \label{thm:bouquetformula} 
Consider (for $a_0^\pm = -1$) the power series $$\mathcal{F}_p^\pm(t) \coloneqq \sum_{g\geq0} (2g-1)d_{g,p}^\pm t^{2g} \ \textrm{and}$$  
$$\mathcal{F}_\epsilon^{\pm} \coloneqq \sum_{p\geq 0} \mathcal{F}_p^\pm \frac{\epsilon^p}{p!} \in \mathbb{Q}\llbracket t\rrbracket[\epsilon],$$ then for positive genera
\begin{align*} 
 [t^{2g}]\big( 2^{-g}\mathcal{S}^\pm - \epsilon t^2\big) &= \frac{1}{(2g)!}[t^{2g}]\Bigl( \mathcal{F}_\epsilon^{\pm}(t)\Bigr)^{2g},\\
[t^{2g}]\ \mathcal{F}_p^\pm(t) 
&= -(2g-1)^{p-1}\ [t^{2g-2p}]\bigg(e^{(1-2g)\  \varpi({t/\sqrt2})}\bigg).
\end{align*}
\end{theorem}

\noindent In particular,  $d^\pm_{g,g} = -(2g-1)^{g-1}$ and
opposed to the classical case, extracting $O(\epsilon^2)$ gives formulas among the $d_{g,p}^\pm$ without reference to Hodge integrals a priori.

\subsection*{Consequence for Chiodo integrals}
The intersection theory of spin-parity squares of \cite{HolPolSau} holds for the cone of squares without parity weighting, and we still obtain a correspondence between Segre and Chiodo classes, see Theorem \ref{thm:Chiodo=Segre}. This correspondence, together with the vanishing related to powers of $\xi$, gives the following vanishing result for Chiodo classes in genus $g>0$ (see Corollary \ref{cor:Chiodo-vanishing}): \[\Omega_{g,1}^{1^n}(t) \text{ is  polynomial of degree } 2g-1.\]

\noindent Furthermore, analogous to our proof of the Main Theorem \ref{thm:maintheorem}, the volumes $a_g$ can then be expressed in terms of integrals against Chiodo classes. While in the proof of the Main Theorem the spin-Chiodo classes are known, in the non-parity case the classical volumes are known. We leverage this to deduce the ratio of the occurring spin-Chiodo integrals to Chiodo integrals to be the ratio of $b_g^\pm$ to $b_g$, cf. Theorem \ref{thm:chiodo-bg}. Phrased compactly, if the genus $0$ integral on the left-hand sides are set to $1$, then the following are equivalent:

\begin{align*}\sum_{g\geq 0}\left(\int_{\oM_{g,1}}\frac{\psi_1^{g-1}}{(g-1)!}\ \Omega_{g,1}^{1}(t)\right)
 =& \frac{\mathcal{S}(t)}{2t},
 \\\sum_{g\geq 0}\left(\int_{\oM_{g,1}}\frac{\psi_1^{g-1}}{(g-1)!}\ \Omega_{g,1}^{\pm,1}(t)\right) =& \frac{2^{-g}\mathcal{S}^\pm(t)}{2t}, \end{align*} the latter of which is readily deduced from \cite[Remark 9.11]{GiaKraLew} in Equation \ref{eq:chiodopm-bgpm}. \\

\begin{conjecture} 
From this numerical evidence we conjecture \[\frac{[t^{2g-1}]\Omega_{g,0}^{\pm}(t)}{[t^{2g-1}]\Omega_{g,0}^{}(t)} = \frac{b_g^\pm}{b_g},\] or equivalently $[t^{2g-1}]\Omega_{g,1}^{1}(t) =(-1)^g\frac{2^{2g-1}-1}{2^g}\lambda_g\lambda_{g-1}$. 
\end{conjecture}
\noindent For $g=3$, by the established Gorenstein property in low genus \cite[Corollary 1.8]{CanLar}, this is equivalent to the minimality of the Chiodo class.
By the CohFT axioms this is equivalent to $[t^{5}]\Omega^{(0,0)}_{2,2}(t)=0$, which can be computed via \textit{admcycles}~\cite{admcycles}.

\subsection*{Outline of the paper} \begin{itemize}
    \item In Section \ref{sec:squares}, we discuss the intersection theory of moduli spaces of squares. This section contains the main theoretical results of \cite{HolPolSau} necessary for our computations: \begin{enumerate}
        \item an analogue of \cite[Proposition 3.11]{Sau}, which allows for the recursive computation of the class $[\mathbb{P}\oH_{g,n}(2g-1)^{+}]-[\mathbb{P}\oH_{g,n}(2g-1)^{-}]$ inside the moduli of squares, cf. Proposition \ref{prop:strata-recursion}; 
      \item a way to relate spin-Segre classes of the cone of squares $\oSQ_{g,n}$ over $\oM_{g,n}$ to spin-Chiodo classes, followed by a formula for Chiodo classes in terms of  $\lambda$-classes, cf. Lemma \ref{eq:spinchiodo=2Lambda} and Theorem \ref{thm:Chiodo=Segre}. 
    \end{enumerate}
   \item In Section \ref{sec: volume}, we compute the class $[\mathbb{P}\oH_{g,n}(2g-1)^{+}]-[\mathbb{P}\oH_{g,n}(2g-1)^{-}]$ and we multiply it with $\xi^{2g-1}$ in order to obtain the spin Masur-Veech volume (up to scaling with a known constant). We will use the formula $\xi^{2g} = 0$ in $\mathbb{P}\cH_{g,n}$ in order to simplify many of the terms and write volumes in terms of Hurwitz integrals and descendants of Segre classes, cf. Lemma \ref{lem:agpm-recursion}.   
   \item In Section \ref{sec: hurwitz}, we look at the Hurwitz integrals, defined as sums of integrals of residueless strata in genus $0$. Seeing these strata either as determinantal loci of strata of differentials, or as Hurwitz spaces, we obtain recurrence formulas for these integrals. We use these formulas to obtain a combinatorial description of Hurwitz integrals, cf. Theorem \ref{thm:Hurformula}.  
   \item In Section \ref{sec: solve recursion}, we bridge together the spin-volume formula of Section \ref{sec: volume}, the computation of the Segre class descendant, and the computation of the Hurwitz integrals. This will conclude the proof of our main result, Theorem \ref{thm:maintheorem}.
   \item In Section \ref{sec: areaSV}, we adapt our methods to compute $d_{g,p}^\pm$. Many of the steps follow similarly, with the remark that along the vanishing $\xi^{2g} = 0$ we will also use $\xi^{2g-1}\delta_0 = 0$ and its analogues for $p>1$. This argument will lead to a proof of Theorem \ref{thm:bouquetformula} and the computability of area Siegel-Veech constants.
   \item Lastly in Section \ref{sec: non-weighted}, we discuss the non-weighted volume studied in \cite{Sau} through the perspective of the intersection theory of the cone of squares. As a consequence, we compute the descendant of a Chiodo class in $\oM_{g,1}$. 
\end{itemize}
\subsubsection*{Acknowledgments.}
We would like to thank Adrien Sauvaget for introducing us to this problem and for his guidance. We furthermore thank Danilo Lewański for useful discussions about Chiodo classes on early stages of this project, as well as Martin M\"oller and Miguel Prado for discussions about volumes and intersection theory of strata. \textbf{A. B.} and \textbf{S.V.} 
were supported by the ERC Starting Grant SpiCE, no.101164820. \textbf{G. P.} was supported by a 
postdoctoral fellowship at the Erd\H{o}s Center of 
the HUN-REN Alfréd Rényi Institute of Mathematics, in 
connection with the Simons Semester in Algebraic
Geometry. Finally, the second named author would also 
like to thank Leiden University for its hospitality 
during part of the preparation of this paper.

\section{Squares} \label{sec:squares}

\subsection{Moduli spaces of differentials}
We denote by $\pi:\oC_{g,n}\to \oM_{g,n}$ 
the universal curve, and by $\oH_{g,n}\to\oM_{g,n}$
the Hodge bundle.
Recall the standard notation for the following tautological classes.
\begin{itemize}
\item The $\psi$-classes $\psi_i := c_1(\sigma_i^*\omega_{\oC_{g,n}/\oM_{g,n}}) \in A^{1}(\oM_{g,n})$,
\item The $\lambda$-classes $\lambda_j := c_j(\oH_{g,n})\in A^{j}(\oM_{g,n})$, and the total Chern polynomial
$$ \Lambda_{g}(t):= \sum_i \lambda_it^i\in A^{*}(\oM_{g,n})[t].$$
\end{itemize}
We denote by $p\colon\PP\oH_{g,n}\to \oM_{g,n}$
the projectivised Hodge bundle and by $\xi:=c_{1}(\cO(1))$
the tautological line bundle class. We also recall
the Leray-Hirsch type identification of the Chow ring of
$\PP\oH_{g,n}$
\begin{equation}\label{eq:Chow-Hodge-bundle}
    A^*(\PP\oH_{g,n})\cong A^*(\oM_{g,n})[\xi]/(\sum_{i=0}^g \xi^{g-i}\lambda_i).
\end{equation} 

\begin{definition}\label{def:mero-hodge-bundle}Let $P=(p_{1},\dots,p_{n})$ be a vector of non-negative integers. We denote by $\oH_{g}[P]$ the vector bundle
with fiber over a point $(C,x_{1},\dots,x_{n})$ in $\oM_{g,n}$ given by 
\begin{equation*}
    H^0\left(C,\omega_{C}\left(\sum_{i=1}^{n}p_{i}x_{i}\right)\right).
\end{equation*}
By Riemann-Roch this vector bundle has rank $\sum_{i=1}^{n}p_{i}+g-1$. 
    
\end{definition}
Let $\ell\in\ZZ_{>0}$, and let $\textbf{g}=(g^{1},\dots,g^{\ell})$ and $\textbf{n}=(n^{1},\dots,n^{\ell})$
be vectors of non-negative integers satisfying 
\begin{equation}\label{eq:cond-on-g-n}
    2g^{i}-2+n^{i}>0, \ \ i=1,\dots,\ell.
\end{equation}
We will always assume this condition whenever we choose a tuple $(\textbf{g},\textbf{n})$. Moreover, 
let $\textbf{P}=(P^{1},\dots,P^{\ell})$ be a vector of 
vectors of non-negative integers $P^{i}=(p^{i}_{1},\dots,p^{i}_{n_{i}})$. Using these data we define
\begin{equation*}
    \oM_{\textbf{g},\textbf{n}}:=\prod_{i=1}^{\ell}\oM_{g^{i},n^{i}},\ \text{and}\ \ \oH_{\textbf{g}}[\textbf{P}]:=\prod_{i=1}^{\ell}\oH_{g^{i}}[P^{i}].
\end{equation*}
Furthermore, we consider vectors of integers 
$a^{i}:=(a_{1}^{i},\dots,a_{n^{i}}^{i})\in\ZZ^{n^{i}}$  
for all $i = 1,\dots,\ell$
such that $-p_{j}^{i}\leq a^{i}_{j}-1$ for all
pairs $(i,j)$.  We denote by  
$\textbf{a}=(a^{i})_{i=1}^{\ell}$
the vector with entries the vectors of integers $a^{i}$.
Then we define the sub-cone  
\begin{equation*}
    \cH_{\textbf{g}}(\textbf{a})\subseteq\oH_{\textbf{g}}[\textbf{P}]    
\end{equation*}
parametrizing differentials of
smooth curves whose order of vanishing at $x^{i}_{j}$ is
given by $a^{i}_{j}-1$, and we define the \emph{incidence variety} $\oH_{\textbf{g}}(\textbf{a})$ to be the Zariski closure of $\cH_{\textbf{g}}(\textbf{a})$ in $\oH_{\textbf{g}}[\textbf{P}]$.

\subsubsection{Residue conditions} Let $N$ denote
the set of pairs $(i,j)$ such that $a^{i}_{j}\leq0$, 
where $1\leq i \leq \ell$ and $1\leq j \leq n_{i}$. We
denote by $\mathcal{R}(\textbf{a})$, or simply by $\mathcal{R}$ when $\textbf{a}$ is clear from the context,
the subspace
of $\CC^{|N|}$ consisting  of vectors $(r^{i}_{j})_{(i,j)\in N}$ satisfying: 
for all $i\in \{1,\dots ,\ell\}$
\begin{equation*}
    \sum_{\substack{j\in \{1,\dots,n^{i}\} \\ (i,j)\in N}}r_{j}^{i} = 0.
\end{equation*}
A \emph{space of residue conditions} is a subspace 
$R\subseteq\mathcal{R}(\textbf{a})$ defined by linear conditions of the form:
\begin{equation*}
    \sum_{ (i,j)\in E}r_{j}^{i} = 0,
\end{equation*}
where $E\subseteq N$. We denote by 
$\oH_{\textbf{g}}(\textbf{a},R)\subseteq\oH_{\textbf{g}}(\textbf{a})$ the sub-cone of differentials whose residues
are prescribed by $R$.

\begin{definition}Let $g,n\in \mathbb{Z}_{\geq0}$ such that $2g-2+n>0$. A \emph{stable graph of type} $(g,n)$ is the data of
\begin{equation*}
    \Gamma=\left(V,H,L, g\colon V\to \mathbb{Z}_{\geq0}, i\colon H\to H, \phi\colon H\to V, \ell\colon L\to \{1,\dots,n\}\right)
\end{equation*}
where $V$, $H$, and $L$ are finite sets and:

\begin{enumerate}
    \item [(i)] The function $i$ is an involution of $H$.
    \item [(ii)] Elements of $H$ are called \emph{half-edges} and the 
    cycles of length 2 for $i$ are called \emph{edges}. We denote the set 
    of cycles of length $2$ for $i$ by $E$. 
    \item [(iii)] 
    The involution $i$ has $n$ fixed points, and $L$ is the subset of
    such fixed points; elements of $L$ are called \emph{legs}. 
    The function $\ell$ identifies $L$ with the set 
    $\{1,\dots,n\}$ of integers from $1$ to $n$.
    \item [(iv)] An element $v\in V$ is called a \emph{vertex}, and
    the value $g(v)$ is called the \emph{genus}
    of the vertex $v$. A half-edge $h$ is \emph{incident} to $v$ if
    $\phi(h) = v$.
    We denote by $H(v)$ the set of half-edges
    incident to $v$, and by $n(v)$ the valency of the vertex $v$, i.e.,
    the cardinality of $H(v)$. 
    \item [(v)] The genus of the graph, defined as
    \begin{equation*}
        g(\Gamma) := h^1(\Gamma) + \sum_{v\in V} g(v), \ \text{where}\ 
        h^1(\Gamma) = |E|-|V|+1,
    \end{equation*}
    is equal to $g$.
    \item [(vi)] The graph is connected, and for all vertices $v$ we have $2g(v) - 2 + n(v) > 0$.
\end{enumerate}
For the sets $H,E$, and $V$, we will 
also use the notation $H(\Gamma)$, $E(\Gamma)$, 
and $V(\Gamma)$. Given a stable graph of type $(g,n)$ we 
use the standard notation $$\zeta_{\Gamma}\colon\prod_{v\in V(\Gamma)}\oM_{g(v),n(v)}\to \oM_{g,n}$$
for the familiar gluing maps.
\end{definition}

\begin{definition} \label{def:twist} A \emph{twist} on a stable graph $\Gamma$
is a function $I\colon H(\Gamma)\to \ZZ$ such that:
\begin{enumerate}
    \item [(a)] if $e=(h,h')$ is an edge, then 
    $$I(h)+I(h')=0.$$
    \item [(b)] For all $v,v'\in V(\Gamma)$, 
    if two edges $(h_{1},h_{1}'), (h_{2},h_{2}')$
    connect the vertices $v,v'$, then
    \begin{equation*}
        I(h_{1})\geq0\Leftrightarrow I(h_{2})\geq0
    \end{equation*}
    in which case we write $v\geq v'$.
    \item [(c)] The relation on the set of vertices
    described in the previous point
    is transitive.
\end{enumerate}
A pair $(\Gamma,I)$ is called a \emph{twisted graph}.
\end{definition}

\subsubsection{Boundary components}
We fix $(\textbf{g},\textbf{n})$ to be vectors of vectors
of non-negative integers satisfying the inequalities \ref{eq:cond-on-g-n}. 
Then a stable graph $\Gamma$ of type 
$(\textbf{g},\textbf{n})$ is the data of stable graphs 
$\Gamma^{i}$ of type $(g^{i},n^{i})$ for $i=1,\dots,\ell$. 
The markings on $\Gamma^{i}$ are labelled by $(i,j)$ for 
$j=1,\dots,n^{i}$. Finally, given a stable graph
$\Gamma$ of type $(\textbf{g},\textbf{n})$ we have the 
familiar gluing morphisms
\begin{equation*}
    \zeta_{\Gamma}\colon \prod_{v\in V(\Gamma)}\oM_{g(v),n(v)}\to \oM_{\textbf{g},\textbf{n}}
\end{equation*}
Moreover, a \emph{twist} on a stable graph of type $(\textbf{g},\textbf{n})$ is a twist on every component of the graph.

\begin{definition} A \emph{bi-colored} graph
is a twisted graph 
$(\Gamma,I)$ together with a non-trivial partition 
$V(\Gamma)=V_{0}\sqcup V_{-1}$ such that:
\begin{enumerate}
    \item [(i)] For each vertex $v\in V(\Gamma)$ we have 
    \begin{equation*}
        \sum_{h\to v}I(h)\leq 2g(v)-2+n(v).
    \end{equation*}
    \item [(ii)] If $e=(h, h')$ is an edge, then $I(h)\neq 0$. If we assume that $I(h) > 0$, then $h$ is incident to a 
    vertex in $V_{0}$ and $h'$ to a vertex in $V_{-1}$. 
\end{enumerate}
Moreover, 
\begin{enumerate}
    \item [(a)] for $v\in V_{i}$, we say that $v$ is \emph{in level} $i$, for $i=0,-1$.
    \item [(b)] If the values of $I$ are odd on every half-edge, we say that the bi-colored graph $\oGamma$ is \emph{odd}.
    \item [(c)] The \emph{multiplicity} of a bi-colored graph $\overline{\Gamma}=(\Gamma,I,V_{0})$ is 
given by $m(\overline{\Gamma})=\prod_{e=(h,h')}\sqrt{-I(h)I(h')}$.
\end{enumerate}
Furthermore, we denote by ${\rm Bic}_{\textbf{g},\textbf{n}}$
the set of odd bi-colored graphs of type $(\textbf{g},\textbf{n})$, and by ${\rm Bic}_{\textbf{g},\textbf{n}}(i,j)\subset{\rm Bic}_{\textbf{g},\textbf{n}}$ the subset
of bi-colored graphs $\oGamma$ where the leg corresponding
to $(i,j)$ lies on a vertex in level $-1$. Finally, given a
vector of vectors of integers $\textbf{a}$, we say that 
$\oGamma$ is compatible with $\textbf{a}$ if 
the values of the twist $I$ at legs corresponding are given by the vector $\textbf{a}$. We denote the set of such bi-colored graphs as ${\rm Bic}_{\textbf{g},\textbf{n}}(\textbf{a})$. 
\end{definition}

Given a bi-colored graph $\overline{\Gamma}=(\Gamma,I,V_{0})$, 
it determines the vectors of genera and twist values at each level: 
that is, 
\begin{align*}
    \textbf{g}[i]&:=(g(v))_{v\in V_{i}} \\
    \textbf{a}[i]&:= (I(h))_{h\in H(\Gamma),\ \phi(h)\in V_i}
\end{align*}
for $i=0,-1$. Moreover, given a space of residue conditions $R$, we denote by $R[i]$ the subspace
of $\mathcal{R}(\textbf{a}[i])$ defined by the linear conditions of $R$ restricted to level $i$, together with those forced by the \emph{global residue condition} (GRC)
(see~\cite[Definition 1.2 (4)]{BaiCheGenGruMol}).
\begin{remark} Throughout the paper we will not deal
with residue conditions on the legs since we will mostly treat the holomorphic case. That is, $R$ will not force any non-trivial residue conditions on the legs, but $R[i]$ may imply non-trivial residue conditions on the half-edges crossing levels (see for example Lemma \ref{lem:BB-graph-residue}).
\end{remark}

\begin{definition} A bi-colored graph $\Gamma$ is called 
\emph{back-bone} if there is a unique vertex $v\in V_{-1}$
such that for all vertices $u\in V_{0}$ there is a unique 
edge connecting $u$ to $v$. Equivalently, $\oGamma$ is
back-bone if it is connected, of compact type, 
and there is a unique vertex $v\in V_{-1}$. 
\end{definition}

Given a bi-colored graph $\oGamma$, 
we denote by $\oH(\oGamma,R)$ the space
\begin{equation*}
    \oH_{\textbf{g}[0]}(\textbf{a}[0],R[0])\times p_{-1}\left(\PP\oH_{\textbf{g}[-1]}(\textbf{a}[-1],R[-1])\right)\subseteq \oH_{\textbf{g}[0]}(\textbf{a}[0])
    \times \oM_{\textbf{g}[-1],\textbf{n}[-1]},
\end{equation*}
where $p_{-1}\colon \PP\oH_{\textbf{g}[-1]}(\textbf{a}[-1])\to  \oM_{\textbf{g}[-1],\textbf{n}[-1]}$ is the 
morphism forgetting the differential. Furthermore, we denote by   
\begin{equation*}
    \zeta_{\oGamma}\colon \oH(\oGamma,R)\to \oH_{\textbf{g}}(\textbf{a},R)
\end{equation*}
the gluing map assigning to each point a differential
vanishing on vertices in level $-1$. The authors of ~\cite{BaiCheGenGruMol} showed that, given a geometric point
in the incidence variety compactification of a stratum $\PP\cH_{g}(a)$, i.e. the closure of $\PP\cH_{g}(a)$ inside the projectivized (twisted) Hodge bundle,
there is an associated level graph $\oGamma$ satisfying certain conditions. In particular, the GRC implies the following lemma when the associated level graph $\oGamma$  is back-bone.
We refer the reader
to~\cite{Sau,BaiCheGenGruMol} for details on the definition
of $\zeta_{\oGamma}$. 

\begin{lemma}\label{lem:BB-graph-residue} Let $(C,x_{1},\dots,x_{n},\omega)$ be an element in the incidence variety compactification of a stratum $\PP\cH_{g}(a)$ whose associated level graph $\oGamma$ is a connected back-bone graph. Then the residues of the restriction of $\omega$ at the branches of nodes vanish.
\end{lemma}

\begin{proof}  First, we note that since $\oGamma$ is
back-bone graph, we have $V_{0}=V\setminus V_{-1}$. Furthermore, for all $u\in V_{0}$, we denote by $h_{u}$ the unique non-leg half-edge incident to $u$. 
Let $(C_{v},(x_{h})_{h\to v},\omega_{v})$ denote
the irreducible component of $C$ corresponding 
to $v\in V_{-1}$. 
Then the global residue condition implies that 
\begin{equation*}
    {\rm res}_{x_{h_{u}}}\omega=0,
\end{equation*}
where $x_{h_{u}}$ is the branch of the node corresponding
to the unique edge joining $u$ to $v$. Finally, since the residues on branches of nodes are opposite we obtain the desired result.
\end{proof}

\subsection{Differentials with spin signs} In this subsection we introduce the \emph{cone of squares}, which was firstly introduced and studied in~\cite{HolPolSau}. Below we recall the basic facts related to our work.

\begin{definition}\label{def:squares}

We denote by $\SQ_{g,n}\subset\cH_{g,n}$ the sub-cone  containing all tuples  $(C,x_1,\dots,x_n,\omega)$ where the order of vanishing of $\omega$ is even
at all points of $C$. Throughout, we will call $\SQ_{g,n}$ and its closure $\oSQ_{g,n}$ in $\oH_{g,n}$ the \emph{sub-cone of squares}
\end{definition}

\noindent Every geometric point $(C,x_{1},\dots x_{n},\omega)\in\SQ_{g,n}$ naturally carries
a spin structure\footnote{As the name suggests, squares are related to tensor squares of sections of the spin bundle. In, particular they are \textit{not} squares of differentials.}. Indeed, the order of vanishing of $\omega$ at all points is even, so the divisor $D\coloneqq \textrm{div}(\omega)/2$ is well defined. In particular, the line bundle $\cO_{C}(D)$ squares to $\omega_{C}$, that is, $\cO_{C}(D)$ is a spin structure on $(C,x_{1},\dots,x_{n})$. Therefore, since the parity
of a spin structure is deformation invariant~\cite{Mum1,Ati}, we obtain the decomposition 
\begin{equation*}
    \SQ_{g,n}=\SQ_{g,n}^{+}\sqcup\SQ_{g,n}^{-}
\end{equation*}
in terms of the parity of the spin structure associated at each point. Now, the space $\PP\oSQ_{g,n}$ can be equivalently interpreted as the image of the stratum $\PP\oH_{g,n+g-1}(3,\dots,3)\subset \PP\oH_{g,n+g-1}$ via the forgetful map
\begin{equation*}
    \pi_{g-1}\colon\PP\oH_{g,n+g-1}\to \PP\oH_{g,n}.
\end{equation*} 
Indeed, this is true because generically a point in 
$\SQ_{g,n}$ has
$g-1$ zeros of order $2$, that is, $\pi_{g-1}$ restricts to 
a finite map
$$\PP\cH_{g,n+g-1}(3,\dots,3)\to \PP\SQ_{g,n}$$
of degree $(g-1)!$ over the open and dense subset of $\SQ_{g,n}$ of differentials with $g-1$ zeros of order $2$. 
In particular, we have $$\dim_{\mathbb{C}}\PP\oSQ_{g,n}=3g-3+n=\dim\oM_{g,n}$$ as $\PP\oH_{g,n+g-1}(3,\dots,3)$ is of codimension $2g-2$ in $\PP\oH_{g,n+g-1}$, which in turn is of dimension $5g-5+n$.
In the following definition we generalize the reasoning above to the case of prescribed orders of vanishing.

\begin{definition} Let $a\in \ZZ^{n}_{>0}$ be a vector of odd integers such that $|a|\leq 2g-2+n$. The sub-cone 
\begin{equation*}
    \SQ_{g}(a):=\Big\{(C,x_1,\dots,x_n,\omega)\in\SQ_{g,n}\ |\ {\rm ord}_{x_{i}}\omega=a_{i}-1\Big\}
\end{equation*}
is called \emph{the sub-cone of squares with zeros prescribed by $a$}.
\end{definition}
Similar arguments as above show that the space 
$\PP\oSQ_{g}(a)$ can be identified as the image of  
the forgetful morphism $\pi_N:\PP\oH_{g,n+N}\to \PP\oH_{g,n}$
restricted to $\PP\oH_{g}(\widetilde{a})$,
where
\begin{equation*}
    N=g-1-\frac{|a|-n}{2},
\end{equation*}
and $\widetilde{a}$ is the vector obtained from 
$a$ by appending the number $3$ exactly $N$ times, 
that is, 
\begin{equation*}
    \widetilde{a}=(a_{1},\dots,a_{n},\underbrace{3,\dots,3}_{\times N}).
\end{equation*}
Essentially, $N$ is the number of zeros of order $2$ required to specify all the orders that the differentials attain. Note that if $a=(1^{n})$, 
we have that $\oSQ_{g}(a)= \oSQ_{g,n}$, in which case we have $N=g-1$. On the other hand, if $N=0$, we have
$\cH_{g}({a})=\SQ_{g}({a})$.
Moreover, a straightforward dimension count gives that
\begin{equation*}
    {\rm codim}\PP\oSQ_{g}(a)=g-1+\frac{|a|-n}{2}
\end{equation*}
in $\PP\oH_{g,n}$, which recovers the dimension count we mention above for $\PP\oSQ_{g,n}$. By the same arguments as above, $\SQ_{g}(a)$ splits according to the parity of the spin structure associated to each point.
\begin{definition}\label{def:pm-squares-classes} The \emph{fundamental class of $\PP\oSQ_{g}(a)$ (resp. weighted by the spin sign)} is given by 
\begin{equation*}
    \left.\begin{matrix}
 [\PP\oSQ_{g}(a)]\hspace{-4mm}&:=[\overline{\PP\SQ_{g}(a)^{+}}]+[\overline{\PP\SQ_{g}(a)^{-}}]\ \ \  \\
[\PP\oSQ_{g}(a)]^{\pm}\hspace{-3mm} &  :=[\overline{\PP\SQ_{g}(a)^{+}}]-[\overline{\PP\SQ_{g}(a)^{-}}]\ \ \ \\ \end{matrix}\right\}\in A^{\bullet}(\PP\oH_{g,n}).
\end{equation*}
\end{definition}

\begin{remark}\label{rem:N!-factorial} Since $\pi_N$ has degree $N!$,  
\begin{equation*}
    (\pi_N)_{*}[\PP\oH_{g}(\widetilde{a})]=N![\PP\oSQ_{g}(a)],
\end{equation*}
and similarly for the spin-parity components. Moreover, the following relation holds for the
anti-tautological line bundle under forgetting markings: $\pi_{N}^{*}\cO(1)\cong \cO(1)$. 
\end{remark}

\subsubsection{Boundary components} 
We can now generalize the definitions above to the (possibly disconnected) case of negative entries, as we did for the Hodge bundle. 
We fix a tuple $(\textbf{g},\textbf{n})$. 
Furthermore, choosing suitable  $\textbf{P}$ and space of residue conditions $R$, we fix a vector of vectors of odd integers $\textbf{a}$ as in the start of the section and we consider the space $\SQ_{\textbf{g}}(\textbf{a},R)\subset\cH_{\textbf{g}}(\textbf{a},R)$ 
of differentials with only even zeros and poles. The analysis and properties carried out in the previous subsection for $\SQ_{g,n}$ can be adapted directly in the case of $\SQ_{\textbf{g}}(\textbf{a},R)$. 
In particular, for these cones we also a have a parity decomposition, as well as a similar dimension count, and a pullback invariance of the line bundle $\cO(1)$ of their projectivizations.  
Given an odd bi-colored graph, we define the following classes:
\begin{align*}
    \oSQ(\oGamma,R)&:=\oSQ_{\textbf{g}[0]}(\textbf{a}[0],R[0])\times p_{-1}\left(\PP\oSQ_{\textbf{g}[-1]}(\textbf{a}[-1],R[-1])\right)\subset\oH(\oGamma,R), \ \text{and} \\
    [\PP\oSQ(\oGamma,R)]^{\pm}&:=[\PP\oSQ_{\textbf{g}[0]}(\textbf{a}[0],R[0])]^{\pm}\otimes p_{-1*}[\PP\oSQ_{\textbf{g}[-1]}(\textbf{a}[-1],R[-1])]^{\pm}, \\
    [\PP\oSQ(\oGamma,R)]&:=[\PP\oSQ_{\textbf{g}[0]}(\textbf{a}[0],R[0])]\otimes p_{-1*}[\PP\oSQ_{\textbf{g}[-1]}(\textbf{a}[-1],R[-1])]
\end{align*} 
in the Chow ring $A^\bullet\left(\PP\oH(\oGamma,R)\right)$,
where the classes $[\PP\oSQ_{\textbf{g}}(\textbf{a},R)]^{(\pm)}$
are defined analogously to Definition \ref{def:pm-squares-classes}. As in~\cite[Section 5]{HolPolSau} the image of 
$\PP\oSQ(\oGamma,R)$ for odd bicolored graphs $\oGamma$ under the gluing map $\zeta_{\oGamma}$ are boundary components of $\PP\oSQ_{\textbf{g}}(\textbf{a},R)$.
For ease of notation we denote the pullback of $\psi_i$ to the projectivised twisted Hodge bundle again by $\psi_i$. We will require the following intersection theoretic result, which is \cite[Proposition 5.4]{HolPolSau}: 
\begin{proposition} \label{prop:strata-recursion} 
Let $a\in\ZZ^{n}$ be a vector of odd integers and let $i\in \{1,\ldots,n\}$. If we assume that $a_i\geq -1,$ and $a_j>0$ for all $j\neq i,$ then
    \begin{equation}\label{for:induction3}
(\xi + a_{i}  \psi_{i})  [\PP\overline{\mathcal{SQ}}_g(a)]^\pm = 2\,[\PP\overline{\mathcal{SQ}}_g(a')]^\pm + \frac{1}{N!}\sum_{\oGamma \in {\rm Bic}_{g,n+N}^*(\widetilde{a};\ i)} \frac{m(\oGamma)}{|\Aut(\oGamma)|} \pi_{N*} \zeta_{\oGamma\, *} [\PP\overline{\mathcal{SQ}}(\oGamma)]^\pm.
\end{equation}
where $a'=(a_1,\ldots,a_i+2,\ldots)$ if $a_i>0$ and $(a_1,\ldots,3,\ldots)$ otherwise, and $N = g - 1 - \frac{|a| - n}{2}$. Moreover, in this formula ${\rm Bic}_{g,n+N}^*(\widetilde{a};\ i)\subset {\rm Bic}_{g,n+N}(\widetilde{a};\ i)$ stands for the set of bi-colored graphs (with leg $i$ on level $-1$ and twists compatible with $\widetilde{a}$) for which the stabilization of the graph obtained by removing the last $N$ legs is not trivial. After removing the $\pm$-symbol, the statement holds true as well. 
\end{proposition}

\begin{remark} The aforementioned proposition is
proven in~\cite{HolPolSau} for the fundamental classes weighted by the spin sign. However, their arguments hold verbatim also when 
we replace the weighted classes with the unweighted classes.
\end{remark}

\subsection{Segre\texorpdfstring{$^\pm$}\ \ and Chiodo\texorpdfstring{$^\pm$}\ \ classes} 

Define the $j$-th spin-parity Segre class of
squares $s_{g,j}^\pm \in A^c(\oM_{g,n})[t]$ 
as the $t^j$-coefficient of 
\begin{align*}
    s^{\pm}_{g}(t) &:=  \sum_{j\geq0}p_{*}(\xi^{j}[\PP\oSQ_{g}(1^n)]^{\pm})t^{j} \in A^\bullet(\oM_{g,n})[t].
\end{align*}
Note that, since $\oSQ_{g,n}$ is of rank $1$ over $\oM_{g,n}$, the class $s_{g}^{\pm}(t)$ is the difference of the Segre polynomials of the spin-parity components of the cone $\oSQ_{g,n}$ over $\oM_{g,n}$. 
Moreover, write $s_{g}^{\pm}=s_{g}^{\pm}(1)$. Note
that we did not include the number of markings in our notation as these classes are invariant under pulling back via the morphism that forgets markings.\\

\noindent Their computation in terms of $\lambda$-classes was performed in ~\cite{HolPolSau} using the following.

\begin{definition} Let 
$a = (1-2a_1\ , \dots , 1-2a_n)\in\ZZ^{n}$ 
be a vector of odd integers, 
$\cL\to \oC_{g,n}^{1/2}\to \oM_{g,n}^{1/2}$ 
the universal spin bundle and 
$\epsilon\colon \oM_{g,n}^{1/2}\to \oM_{g,n}$ 
the forgetful morphism. Then $\cL( \text{\tiny$\sum_i$} a_i x_i)$ 
is the universal square root of 
$\omega_{\rm log}\left(-a\cdot (x_1,\dots,x_n)\right)$.
We denote by 
\[\Omega_{g,n}^{\pm,a}(t):=\epsilon_{*}\bigg(c_t\big(-R^{\bullet}\pi_{*}\cL( \text{\tiny$\sum_i$} a_i x_i)\big)\ [\pm]\bigg)\]
the \emph{spin Chiodo class}, 
where $[\pm]$ is the class of the even minus the class of the odd parity component. 
\end{definition}

\begin{theorem} \label{thm:Segre-Chiodo}~\cite[Theorem 1.10]{HolPolSau} For all $(g,n)$ in the stable range the following identities hold 
    \begin{align}
     \Omega_{g,n}^{\pm,1^n}(2t)-2^{2g-1}&=\sum_{\Gamma\in{{\rm Tree}_{g,n}}}\frac{t^{E(\Gamma)}}{|\Aut(\Gamma)|}\zeta_{\Gamma*}\left(\bigotimes_{v\in V(\Gamma)}s^{\pm}_{g(v)}(t)\right), \ \ {\rm and} \\
     s^{\pm}_{g}(t)&=\sum_{\Gamma\in{{\rm Tree}_{g,n}}}\frac{(-t)^{E(\Gamma)}}{|\Aut(\Gamma)|}\zeta_{\Gamma*}\left(\bigotimes_{v\in V(\Gamma)}\Omega_{g(v),n(v)}^{\pm,1^{n(v)}}(2t)-2^{2g(v)-1}\right)\label{eq:weighted-segre}
\end{align} 
in $A^\bullet(\oM_{g,n})[t]$, where ${\rm Tree}_{g,n}$ denotes the set of compact type stable graphs of type $(g,n)$.
\end{theorem}
\noindent In Section \ref{sec:non-parity} we discuss and prove its modification for non-parity Segre and Chiodo classes.
In the spin-parity case, the following explicit formula was then used in \textit{op.cit.} to express the spin Segre classes in terms of $\lambda$-classes.

\begin{lemma}~\cite[Remark 9.11]{GiaKraLew}\label{eq:spinchiodo=2Lambda}
The spin Chiodo classes with all insertions set to $1$ can be computed as
    $$\Omega_{g,n}^{\pm,1^{n}}(2t)=2^{g-1}\Lambda_{g}(-t)\Lambda_{g}(2t).$$
\end{lemma}

\section{Volume\texorpdfstring{$^\pm$}\ \  recursion}  \label{sec: volume}

As explained in the introduction, in~\cite[Theorem 1.1]{CheMolSauZag} the authors
prove that the difference ${\rm Vol}^{\pm}(2g-1)$ of the volumes of the even and odd
components of the minimal stratum can be computed via the following proposition.
\begin{proposition}[\cite{CheMolSauZag}]\label{prop:segreMV} The difference of volumes ${\rm Vol}^\pm(2g-1)$ can be expressed as
\begin{equation*}
    {\rm Vol}^\pm(2g-1)=\frac{2(2i\pi)^{2g}}{(2g-1)!}a_{g}^{\pm},
\end{equation*}
where 
\begin{equation*}
    a_{g}^{\pm}:=\int_{[\PP\oSQ_g(2g-1)]^\pm} \xi^{2g-1}.
\end{equation*}

\end{proposition}

\noindent In order to state the recursion for the latter integrals, we introduce some further notation. Firstly, we will write  \[\beta^{\pm}_{g}:= \int_{\oM_{g,1}}\psi^{g-1}_{1}s_{g,2g-1}^{\pm}.\] Writing moduli spaces with residue condition $R = \{0\}$ with the superscript $\res$, let furthermore be
\begin{align*}
    h(k-1,-2\underline{g})&:=\int_{\oM_{0,m+1}}\psi_{1}^{N_{k,g}-1}p_{*}[\PP\oSQ^{\rm res}_{0}(k,\ (-2g_i+1)_{i = 1}^m)]\ ,\\
    Hur(\underline{g}) &:= \sum_{k=1,\ {\rm odd}}^{2g-3} \frac{(2g-3)!!}{k!!\,\,2^{N_{k,g}}}\,\,h(k-1,-2\underline{g}).
\end{align*} 
Here we take $N_{k,g}:=g-1 - \frac{k-1}{2}$. Note that, even though our computations involve classes weighted by the spin sign, we omit the $\pm$-symbol in the definition of $h(k-1,2\underline{g})$ as the parity is always even in genus $0$. Moreover, since residueless differentials in genus $0$ are exact, the above integral could be interpreted as integrals over certain Hurwitz spaces, motivating our notation. Furthermore, we define $$D(g) \coloneqq (-1)^{g}\ 2^{2g-2}\ (g-1)!.$$

 The goal of this section is to prove the following 
recursion property for the integrals $a_{g}^{\pm}$.
\begin{lemma}[Volume$^\pm$ recursion]\label{lem:agpm-recursion}
Let $g\in \ZZ_{\geq2}$ and let ${\rm Part}(g)_{m}$ be the set of length $m$ unordered partitions $\underline{g}=(g_1,\dots,g_m)$ of $g$, with entries $g_i \geq 1$. Then the following holds:
\begin{equation*}
    a_{g}^{\pm}=\frac{(2g-3)!!}{2^{g-1}}\,\beta^{\pm}_{g}-\sum_{\substack{2\leq m\leq g,\\\underline{g}\in{\rm Part}(g)_{m}}}\frac{1}{m!} \left(\prod_{i=1}^{m}(2g_{i}-1)\,a_{g_{i}}^{\pm}\right)\,Hur(\underline{g}).
\end{equation*}

Moreover, \[\beta_g^\pm =\sum_{m=0}^{g-1}\frac{ (b^\pm_1)^m}{m!}\ {D}(g-m)\ b^\pm_{g-m}. \]
\end{lemma}

Our tools for proving such a recursion property are
the recursive formula in Proposition ~\ref{prop:strata-recursion} and
the geometry of the spaces $\PP\oSQ_{g}(a)$ introduced in the previous section. In particular, the formula
of Proposition ~\ref{prop:strata-recursion} applied $g-1$ times to the case $n=1$
and $a=(2g-1)$ takes the following form, where the number $N_{k,g}$ is the appended number of zeros of order $2$ that we need to impose in each step $k$.
\begin{proposition}\label{prop:2g-1-recursion}
Let $g\in\ZZ_{\geq1}$. Then the following holds in $A^\bullet(\PP\oH_{g,1})$.
\begin{equation*}
    [\PP\oSQ_{g}(2g-1)]^{\pm}=\frac{1}{2^{g-1}}[\PP\oSQ_{g}(1)]^{\pm}\prod_{\substack{k=1, \\ k\ {\rm odd}}}^{2g-3}(\xi+k\psi_{1}) - \sum_{\substack{k=1, \\ k\ {\rm odd}}}^{2g-3}\frac{1}{N_{k,g}!}\frac{\Delta_{k}}{2}\prod_{\substack{\ell>k, \\ \ell\,{\rm odd}}}^{2g-3}\left(\frac{\xi+\ell\psi_{1}}{2}\right),
\end{equation*} 
where $N_{k,g}=g-1-\frac{k-1}{2}$ and
\begin{equation}\label{def:Delta-k}
    \Delta_{k}:=\sum_{\oGamma \in {\rm Bic}_{g,1+N_{k,g}}^*(\widetilde{2g-1-2k};\  1)} \frac{m(\oGamma)}{|\Aut(\oGamma)|} \pi_{N_{k,g}*} \zeta_{\oGamma\, *} [\PP\overline{\mathcal{SQ}}(\oGamma)]^\pm. 
\end{equation}
Finally, the same equation holds when we omit the $\pm$-symbol. 
\end{proposition}

Multiplying by $\xi^{2g-1}$ and pushing forward via $p\colon\PP\oH_{g,1}\to \oM_{g,1}$, the LHS of the 
equation in the aforementioned proposition 
is $a_{g}^{\pm}$. On the other hand, the RHS consists 
of multiple terms which we treat individually. 
The forthcoming subsections will thus be devoted to computing the \emph{principal term} (in terms of $\beta_g^\pm$, which are computed in Lemma \ref{lem:btilde-computation})
\begin{equation*}
    \int_{\oM_{g,1}} p_{*}\left(\xi^{2g-1}\frac{1}{2^{g-1}}[\PP\oSQ_{g}(1)]^{\pm}\prod_{\substack{k=1, \\ k\ {\rm odd}}}^{2g-3}(\xi+k\psi_{1})\right),
\end{equation*}
and the \emph{boundary term} (see Corollary \ref{cor:boundaryterm}). For the latter,
it will turn out that we want to compute (in terms of $Hur$ and $a_{g}^\pm$ of smaller $g$, see Proposition \ref{prop:boundary-contribution}) for each odd bi-colored graph $\oGamma$
appearing in $\Delta_{k}$ the contribution
\begin{equation*}
   \int_{\oM_{g,1}} p_{*}\left(\psi^{N_{k,g}-1}_{1}\xi^{2g-1}\pi_{N_{k,g}*}\zeta_{\oGamma*}[\PP\oSQ(\oGamma)]^{\pm}\right),
\end{equation*}
as well, and show that only the bi-colored graphs $\oGamma$
that are back-bone with rational central vertex contribute 
non-trivially. Finally, by assembling these terms and the computation of the principal part we will deduce the formula
in Lemma \ref{lem:agpm-recursion}.\\

We remark that the computations in the following two subsections hold verbatim if we replace the classes weighted by the spin sign with the classical ones, up to the evaluation of Lemma \ref{eq:spinchiodo=2Lambda}. 
\subsection{Principal term} 
The following vanishing result, which follows directly from Mumford's formula, will determine the non-trivial contributions to the principal part.
\begin{lemma}~\cite{Sauvaget-virtual}\label{lem:xi-vanishing} For all 
$k>0$ we have 
\begin{equation*} 
    \xi^{2g-1+k}=0 \in A^\bullet(\PP\oH_{g,n}).
\end{equation*}
\end{lemma} 

After multiplying the RHS in Proposition~\ref{prop:2g-1-recursion} with $\xi^{2g-1}$, by Lemma \ref{lem:xi-vanishing}, only the constant term in $\xi$ of the polynomial $\prod_{\substack{k=1, \\ k\ {\rm odd}}}^{2g-3}(\xi+k\psi_{1})$  will contribute non-trivially. 
In particular, we have 
\begin{equation*}
    \xi^{2g-1}\frac{1}{2^{g-1}}[\PP\oSQ_{g}(1)]^{\pm}\prod_{\substack{k=1, \\ k\ {\rm odd}}}^{2g-3}(\xi+k\psi_{1})=\frac{(2g-3)!!}{2^{g-1}}\xi^{2g-1}[\PP\oSQ_{g}(1)]^{\pm}\psi^{g-1}_{1},
\end{equation*}
and so, using the projection formula, we only have to compute the integral
\begin{equation*}
    \int_{\oM_{g,1}} p_{*}\left(\xi^{2g-1}[\PP\oSQ_{g}(1)]^{\pm}\psi^{g-1}_{1}\right)=\int_{\oM_{g,1}}\psi_{1}^{g-1}s_{g,2g-1}^{\pm} = \beta_g^{\pm}.
\end{equation*}

\begin{lemma}[Segre$^\pm$ descendant computation]\label{lem:btilde-computation}\label{lem:bg-pm-computation} For all $g\geq1$ we have
\begin{equation*}
    \beta^\pm_g = (-1)^{g-1}\sum_{m=0}^{g-1}\frac{2^{2(g-m-1)}}{m!}
    \left(b_1\right)^{m} b_{g-m}^{\rm quad},
\end{equation*}
where $b_{g}^{\rm quad} \coloneqq \int_{\oM_{g,1}}\psi_1^{g-1}
\lambda_{g}\lambda_{g-1}$ are quadratic Hodge integrals, which in turn is equivalent to 
\begin{equation*}
    \beta^\pm_g = (-1)^{g}\sum_{m=0}^{g-1}\frac{(g-m-1)!2^{2(g-m-1)}}{m!}
    \left(b_1\right)^{m} b_{g-m}^{\pm}.
\end{equation*}
\end{lemma}
\begin{proof} 
We use equality \ref{eq:weighted-segre} to identify the $2g-1$-th degree 
of $s_{g}^{\pm}$ with
\begin{align}
    s_{g,2g-1}^{\pm}&=\sum_{\Gamma\in{\rm Tree}_{g,1}}\frac{(-1)^{|E(\Gamma)|}}{|\Aut(\Gamma)|}\zeta_{\Gamma*}\left(\bigotimes_{v\in V(\Gamma)}       [\Omega_{g(v),n(v)}^{\pm,1^{n(v)}}(2t)]_{2g(v)-1}\right).\label{eq:top-weighted-Segre}
\end{align}
Indeed, for each graph appearing in said formula, $h^{1}(\Gamma)=0$, so 
the RHS after pushing forward via $\zeta_{\Gamma}$
lies in degree
\begin{equation*}
    E(\Gamma)+ \sum_{v\in V(\Gamma)}(2g(v)-1)=E(\Gamma)+2g-V(\Gamma)=2g-1.
\end{equation*}
No other terms contribute, since $2g(v)-1$ is the top degree of these spin Chiodo classes by Lemma \ref{eq:spinchiodo=2Lambda}.
Then, multiplying both sides of 
equation \ref{eq:top-weighted-Segre} with $\psi_{1}^{g-1}$ and integrating, we obtain after a straightforward computation
\begin{align}\label{eq:segre-psi}
\int_{\oM_{g,1}}\psi^{g-1}_{1}s^{\pm}_{g,2g-1}=\sum_{\Gamma\in{\rm Tree}_{g,1}}\frac{(-1)^{|E(\Gamma)|}}{|\Aut(\Gamma)|}
    \int_{\oM_{g(v_{1}),n(v_{1})}}&\psi_{1}^{g-1}[\Omega_{g(v_{1}),n(v_{1})}^{\pm,1^{n(v_{1})}}(2t)]_{2g(v_{1})-1} \times \\
    &\prod_{v\neq v_{1}}\int_{\oM_{g(v),n(v)}}
          [\Omega_{g(v),n(v)}^{\pm,1^{n(v)}}(2t)]_{2g(v)-1}, \nonumber
\end{align}
where $v_{1}\in V(\Gamma)$ 
is the vertex carrying the unique marking. 
Thus, on the vertices $v\in V(\Gamma)\setminus \{v_{1}\}$, the 
graph $\Gamma$ is decorated with a class 
in codimension $2g(v)-1$.
This class is \emph{not} in top degree unless $g(v)=1$, $n(v)=1$.
Therefore, the only graphs that may contribute non-trivially
are back-bone graphs with central vertex $v_{1}$ of genus
$g-m$ carrying the unique leg and $0\leq m \leq g-1$  outlying vertices of 
genus $1$ with $1$ marking, that is $\Gamma$ is a back-bone graph with elliptic tails decorated with $\lambda$-classes as depicted in the figure below\footnote{The particular $\lambda$-classes inserted at each vertex follows by the equation above and Lemma~\ref{eq:spinchiodo=2Lambda}, as it is also explained below the figure.}. 

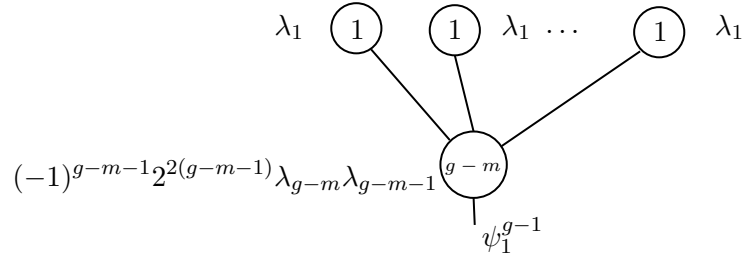
\tikzset{every picture/.style={line width=0.75pt}}
\[\begin{tikzpicture}[x=0.75pt,y=0.75pt,yscale=-1,xscale=1]

\draw   (248,75.42) .. controls (248,68.47) and (253.63,62.83) .. (260.58,62.83) .. controls (267.53,62.83) and (273.17,68.47) .. (273.17,75.42) .. controls (273.17,82.37) and (267.53,88) .. (260.58,88) .. controls (253.63,88) and (248,82.37) .. (248,75.42) -- cycle ;
\draw    (319.43,160.55) -- (320,174.12) ;
\draw    (268,86) -- (307.5,131.55) ;
\draw    (310,89) -- (319.43,127.23) ;
\draw    (403,87) -- (333.5,134.55) ;
\draw   (297.42,76.42) .. controls (297.42,69.47) and (303.05,63.83) .. (310,63.83) .. controls (316.95,63.83) and (322.58,69.47) .. (322.58,76.42) .. controls (322.58,83.37) and (316.95,89) .. (310,89) .. controls (303.05,89) and (297.42,83.37) .. (297.42,76.42) -- cycle ;
\draw   (302.78,143.89) .. controls (302.78,134.69) and (310.23,127.23) .. (319.43,127.23) .. controls (328.63,127.23) and (336.09,134.69) .. (336.09,143.89) .. controls (336.09,153.09) and (328.63,160.55) .. (319.43,160.55) .. controls (310.23,160.55) and (302.78,153.09) .. (302.78,143.89) -- cycle ;
\draw   (400,77.42) .. controls (400,70.47) and (405.63,64.83) .. (412.58,64.83) .. controls (419.53,64.83) and (425.17,70.47) .. (425.17,77.42) .. controls (425.17,84.37) and (419.53,90) .. (412.58,90) .. controls (405.63,90) and (400,84.37) .. (400,77.42) -- cycle ;

\draw (354,75) node [anchor=north west][inner sep=0.75pt]   [align=left] {$\displaystyle \dotsc $};
\draw (256,70) node [anchor=north west][inner sep=0.75pt]   [align=left] {$\displaystyle 1$};
\draw (322,168.97) node [anchor=north west][inner sep=0.75pt]   [align=left] {$\displaystyle \psi _{1}^{g-1}$};
\draw (439,67) node [anchor=north west][inner sep=0.75pt]   [align=left] {$\displaystyle \lambda _{1}$};
\draw (332,67) node [anchor=north west][inner sep=0.75pt]   [align=left] {$\displaystyle \lambda _{1}$};
\draw (218,67) node [anchor=north west][inner sep=0.75pt]   [align=left] {$\displaystyle \lambda _{1}$};
\draw (72,139) node [anchor=north west][inner sep=0.75pt]   [align=left] {$\displaystyle \ \ \ (-1)^{g-m-1}2^{2( g-m-1)} \lambda _{g-m}$$\displaystyle \lambda _{g-m-1}$};
\draw (305,70) node [anchor=north west][inner sep=0.75pt]   [align=left] {$\displaystyle 1$};
\draw (408,71) node [anchor=north west][inner sep=0.75pt]   [align=left] {$\displaystyle 1$};
\draw (303.57,139.57) node [anchor=north west][inner sep=0.75pt]  [font=\tiny] [align=left] {$\displaystyle {\textstyle g-m}$};

\end{tikzpicture}
\]
\noindent For such graphs $\Gamma$ we have $E(\Gamma)=m$,
$V(\Gamma)=m+1$, and $|\Aut(\Gamma)|=m!$. In particular, 
$|V(\Gamma)\setminus\{v_{1}\}|=m$. Therefore the equation~\ref{eq:segre-psi}
becomes

\begin{align}\label{eq:top-segre-integral-w/Chiodo}
          \int_{\oM_{g,1}}\psi^{g-1}_{1}s_{g,2g-1}^{\pm}=\sum_{m=0}^{g-1}\frac{(-1)^{m}}{m!}
    \int_{\oM_{g-m,m+1}}&\psi_{1}^{g-1}[\Omega_{g-m,m+1}^{\pm,1^{m+1}}(2t)]_{2(g-m)-1}\ \times  \\
    &\prod_{i=1}^{m}\int_{\oM_{1,1}}
           [\Omega_{1,1}^{\pm,1}(2t)]_{1}.\nonumber
\end{align}

\noindent To compute the right-hand side we substitute formula \ref{eq:spinchiodo=2Lambda},
\[[\Omega_{g,n}^{\pm,1^{n}}(2t)]_{2g-1}=2^{2(g-1)}(-1)^{g-1}\lambda_{g}\lambda_{g-1}.\]
First, note that the projection formula along forgetting the last $m$-marked points then gives
\begin{equation*}
    \int_{\oM_{g-m,m+1}}\psi_{1}^{g-1}\Omega^{\pm,1^{m+1}}_{g-m,m+1}(2t)= \int_{\oM_{g-m,1}}\psi_{1}^{g-m-1}\Omega_{g-m,1}^{\pm,1}(2t).
\end{equation*}
Then, these integrals are the quadratic Hodge integrals explicitly computed in \cite{Fab}. Comparing with the explicit computation of the linear Hodge integral in \cite{FabPan} we get
\begin{align} 
    \int_{\oM_{g,1}}\psi_{1}^{g-1}\Omega_{g,1}^{\pm,1}(2t)
    &= 2^{2(g-1)}(-1)^{g-1}\cdot \int_{\oM_{g,1}}\psi_{1}^{g-1}\lambda_{g}\lambda_{g-1} \nonumber \\
    &= 2^{2(g-1)}(-1)^{g-1} \cdot \frac{2^{g-1}(g-1)!}{2^{2g-1}-1}b_{g} \nonumber \\
    &= 2^{2(g-1)}(-1)^{g} \cdot (g-1)!\ b_{g}^\pm. \label{eq:chiodopm-bgpm}
\end{align} 
Combined with Equality \ref{eq:top-segre-integral-w/Chiodo} this concludes 
\[\beta^\pm_g=(-1)^{g}
    \sum_{m=0}^{g-1}
    \frac{(g-m-1)!\ 2^{2(g-m-1)}}{m!}\left(b_1\right)^{m} b^{\pm}_{g-m}.\] 
\end{proof} 
\noindent This proves the second statement in Lemma \ref{lem:agpm-recursion}.

\subsection{Boundary term}
As stated before, this subsection is devoted to computing the boundary contribution in Proposition \ref{prop:2g-1-recursion} after multiplying by $\xi^{2g-1}$ and pushing forward to $\oM_{g,1}$. We start with a vanishing result which we will use along the way.
\begin{lemma}\label{lem:non-trivial-terms} Let $a\in\ZZ^{n}_{\geq0}$ be a vector of odd integers such 
that $|a|:=\sum a_{i}\leq2g-2+n$ and $t \in \NN$. Then 
\begin{equation*}
    \int_{\PP\oH_{g,n}}\xi^{ t}[\PP\oSQ_{g}(a)]=0
\end{equation*}
unless $t= 2g-1$, $n=1$, and $a=(2g-1)$. Moreover, this is 
true also for the class weighted by the spin sign.
\end{lemma}
\begin{proof}
This is essentially a degree argument similar to Lemma \ref{lem:xi-vanishing}.
Denote $\ell_{a}:=g-1+\frac{|a|-n}{2}$.
The Leray-Hirsch Theorem states that there exists $\alpha_{k}\in A^{\bullet}(\oM_{g,n})$ such that
$$[\PP\oSQ_{g}(a)]=\sum_{k=0}^{\ell_{a}}\xi^{k}p^*\alpha_{k} \in A^{\ell_{a}}(\PP\oH_{g,n}).$$
The integral is zero unless $t=\dim\PP\oSQ_{g}(a)=3g-3+n-\frac{|a|-n}{2}.$
Combining these we obtain $\xi^{t}[\PP\oSQ_{g}(a)]=\sum_{k=0}^{\ell_{a}}\xi^{3g-3+n-\frac{|a|-n}{2}+k}\,p^*\alpha_{k}$.
The assumption $|a|\leq 2g-2+n$ implies $
        3g-3+n-\frac{|a|-n}{2}+k\geq 2g-2+n+k$.
    Using Lemma ~\ref{lem:xi-vanishing}, the integrals are zero unless $k=0$ and $n=1$, in which case $a=2g-1$.   
\end{proof}
\begin{lemma}\label{lem:xi-int-gluing-vanishing} 
Let $V_{0}$ be a finite set, and let $\textbf{g}=(g_{v})_{v\in V_{0}}$ and $\textbf{n}=(n(v))_{v\in V_{0}}$ be vectors of positive integers such that $2g_v-2+n_{v}>0$ for all $v\in V_{0}$. Let $I=(I(v))_{v\in V_{0}}$ be a vector of vectors of odd positive integers such that for all $v\in V_{0}$ we have  ${\rm length}(I(v))=n(v)$. Consider the projectivization of product of cones
\begin{equation*}
    \PP\left(\prod_{v\in V_{0}}\oSQ_{g(v)}(I(v))
\right).
\end{equation*}
Then 
\begin{equation*}
    \int_{[\PP\left(\prod_{v\in V_{0}}\oSQ_{{g(v)}}(I(v))
\right)]}\xi^{2g-1}=0
\end{equation*}
unless for all vertices $v\in V_{0}$ we have $n(v)=1, I(v)=2g(v)-1$, and $\sum_{v\in V_{0}}g(v)=g$. The same holds if we consider the class weighted by the spin sign.
\end{lemma}

\begin{proof} 
Using the Segre class decomposition on product of cones we obtain 
\begin{equation*}
    \xi^{2g-1}[\PP\left(\prod_{v\in V_{0}}\oSQ_{g(v)}(I(v))\right)]=\sum_{(t_{v})\ \vdash\ d}\ \bigotimes_{v\in V_{0}}\xi^{t_{v}}[\PP\oSQ_{g(v)}(I(v))],
\end{equation*}
where $d=2g-|V_{0}|$ and $(t_{v})_{v\in V(\Gamma)}$ 
is a partition of $d$. Therefore, we obtain
\begin{align}\label{eq:integrals-of-xi}
    \int_{[\PP\left(\prod_{v\in V_{0}}\oSQ_{g_{v}}(I(v))
\right)]}\xi^{2g-1}
    &= \sum_{(t_{v})\ \vdash\ d}\ \prod_{v\in V_{0}}\int_{\PP\oH_{g(v),n(v)}}\xi^{t_{v}}[
    \PP\oSQ_{g(v)}(I(v))].
\end{align}
Finally, by Lemma~\ref{lem:non-trivial-terms}, the RHS of the equation above vanish unless $n(v)=1$ and $I(v)=t_{v}=2g(v)-1$ for all vertices $v\in V_{0}$. 
Furthermore the dimension of 
$\oSQ_{g(v)}(2g(v)-1)$ is $2g(v)$,
thus we obtain
\begin{align*}
    \dim\PP\left(\prod_{v\in V_{0}}\oSQ_{g(v)}(2g(v)-1)\right)=-1+\sum_{v\in V_{0}}2g(v).
\end{align*} 
Therefore,
the LHS of equation ~\ref{eq:integrals-of-xi} will vanish unless 
\begin{equation*}
    -1+\sum_{v\in V_{0}}2g(v)=2g-1\Leftrightarrow\sum_{v\in V_{0}}g(v)=g.
\end{equation*}
Finally, we note that all arguments used hold both in weighted and non-weighted case.
\end{proof}

Recall the numbers $$h(k-1,-2\underline{g})=\int_{\oM_{0,m+1}}\psi_{1}^{N_{k,g}-1}p_{*}[\PP\oSQ^{\rm res}_{0}(k,\ (-2g_i+1)_{i = 1}^m)]$$ defined just above Lemma \ref{lem:agpm-recursion}. These can be used to compute contributions coming from the boundary. In the notation of Proposition \ref{prop:strata-recursion} we have:
\begin{proposition}\label{prop:boundary-contribution} Let $\oGamma$ be an odd bi-colored graph in ${\rm Bic}_{g,1+N_{k,g}}^*(\widetilde{2g-1-2k};\  1)$, and let $k\in\{1,\dots,2g-3\}$ be an odd number.
Then
\begin{equation*}
    \int_{\oM_{g,1}}p_{*}\left(\psi^{N_{k,g}-1}_{1}\xi^{2g-1}\pi_{N_{k,g}*}\zeta_{\oGamma*}[\PP\oSQ(\oGamma)]^{\pm}\right) = \begin{cases}  N_{k,g}!\ h\big(k-1,(-2g_v)_{v\in V_0}\big)\prod_{v\in V_{0}}a_{g_v}^{\pm} & \text{ if } \oGamma \text{ satisfies }(*), \\ 0 & \text{ else} . \end{cases}
\end{equation*}
Here $(*)$ means that $\oGamma$ is a back-bone graph such that the
unique vertex in $V_{-1}$ is of genus $0$, carries all $N_{k,g}$ markings, 
and for all $v\in V_{0}$ the twist $I$ satisfies $I(h)=2g(v)-1$ on the unique half-edge $h$ adjacent to $v$.
\end{proposition}

\begin{proof}
First we will decompose the integrand into the level $0$ and level $-1$ contributions.
Note that, $\pi_{N_{k,g}}^*(\cO(1))=\cO(1)$ and that the differential vanishes identically on level $-1$ (implying that $\xi^{2g-1}$ is
supported only in level $0$). Then using repeatedly the projection formula we obtain
\begin{align*}
    p_{*}\bigg(\psi^{N_{k,g}-1}_{1}\xi^{2g-1}\pi_{N_{k,g}*}\zeta_{\oGamma*}[\PP\oSQ(\oGamma)]^{\pm}\bigg)&=\psi^{N_{k,g}-1}_{1}p_{*}\bigg(\xi^{2g-1}\pi_{N_{k,g}*}\zeta_{\oGamma*}[\PP\oSQ(\oGamma)]^{\pm}\bigg) \\
    &=\psi^{N_{k,g}-1}_{1}p_{*}\pi_{N_{k,g}*}\bigg(\xi^{2g-1}\zeta_{\oGamma*}[\PP\oSQ(\oGamma)]^{\pm}\bigg) \\
    &=\psi^{N_{k,g}-1}_{1}p_{*}\pi_{N_{k,g}*}\zeta_{\oGamma*}\bigg( \xi^{2g-1}\otimes 1\cdot[\PP\oSQ(\oGamma)]^{\pm}\bigg).
\end{align*}
Denote by superscript $\Gamma'$ the stabilised graph after forgetting $N_{k,g}$ additional markings. Then following commutative diagrams including forgetting markings, differentials and gluing morphisms we obtain
\begin{align*}&\psi^{N_{k,g}-1}_{1}\zeta_{\Gamma'*}\ (p_{0*}\otimes p_{-1*})\ \pi_{N_{k,g}*}\bigg(\xi^{2g-1}\otimes 1\cdot[\PP\oSQ(\oGamma)]^{\pm}\bigg).\\
\end{align*}
Which in turn, by the fact that the additional markings $N_{g,k}$ are distributed as $N_0$ resp. $N_{-1}$ over level $0$ and $-1$, equals
\begin{align*}
&\zeta_{\Gamma'*}\Bigg(1\otimes \psi^{N_{k,g}-1}_{1}  \cdot   (p_{0*}\otimes p_{-1*})\ \pi_{N_{k,g}*}\bigg(\xi^{2g-1}\otimes 1\cdot[\PP\oSQ(\oGamma)]^{\pm}\bigg)\Bigg)= \\
&\zeta_{\Gamma'*}\Bigg(1\otimes \psi^{N_{k,g}-1}_{1}  \cdot   (p_{0*}\otimes p_{-1*})\ (\pi_{N_{0}*}\otimes \pi_{N_{-1}*})\bigg(\xi^{2g-1}\otimes 1\cdot[\PP\oSQ(\oGamma)]^{\pm}\bigg)\Bigg)= \\ 
&\zeta_{\Gamma'*}\Bigg(p_{0*}\pi_{N_{0}*}\big(\xi^{2g-1}\cdot[\PP\oSQ_{\textbf{g}[0]}(\textbf{a}[0],R[0])]^{\pm}\big)\otimes\psi_1^{N_{g,k}-1}p_{-1*}\pi_{N_{-1}*}[\PP\oSQ_{\textbf{g}[-1]}(\textbf{a}[-1],R[-1])]^{\pm} \bigg)\Bigg).\\ 
\end{align*}

Therefore, our integral splits into a product of integrals in level $0$ and $-1$.
We will treat each level individually.  \\

\textbf{Level 0:}\\
By Lemma~\ref{lem:xi-int-gluing-vanishing}, the integral in level $0$ before forgetting markings and the differentials vanishes unless $n(v)=1$ and $I(v)=2g(v)-1$ for all $v\in V_{0}$ and $\sum_{v\in V_{0}}g(v)=g$. In particular $N_{0} = 0$
and all additional $N_{k,g}$ markings are concentrated in level $-1$. Altogether, we obtain that level $0$ contributes as 
\begin{equation*}
    \prod_{v\in V_{0}}\int_{[\PP\oSQ_{g(v)}(2g(v)-1)]^{\pm}}\xi^{2g(v)-1}=\prod_{v\in V_{0}}a_{g(v)}^{\pm}.\\
\end{equation*}

\textbf{Level -1:}\\
The fact that $\sum_{v\in V_{0}}g(v)=g$ implies
that the underlying graph ($\Gamma'$ and hence) $\Gamma$ is of compact type, i.e. $h^{1}(\Gamma)=0$.
Moreover, by a dimension argument, there
is a unique vertex $v$ in level $-1$ which necessarily satisfies $g(v)=0$. Indeed, in order for \[ \psi_{1}^{N_{k,g}-1}\left(\pi_{N_{k,g}}\circ p_{-1}\right)_{*}\left(\PP\oSQ_{\textbf{g}[-1]}(\textbf{a}[-1],R[-1])\right)\]
to not vanish, the push-forward must have the same dimension as $\PP\oSQ_{\textbf{g}[-1]}(\textbf{a}[-1],R[-1])$. Since the forgetful map to $\oM_{g,1}$ forgets all the differentials, we have 
\begin{align*}
    \dim \left(\pi_{N_{k,g}}\circ p_{-1}\right)\left(\PP\oSQ_{\textbf{g}[-1]}(\textbf{a}[-1],R[-1])\right) &\leq  \dim\left(\prod_{v\in V_{-1}}\PP\oSQ_{0}(a(v), R(v))\right) \\
   &\leq \dim\PP\left(\prod_{v\in V_{-1}}\oSQ_{0}(a(v), R(v))\right) - \left(|V_{-1}| -1\right).
\end{align*} 
In particular, $|V_{-1}| = 1$ and $\oGamma$ is a back-bone graph. Furthermore,
the contribution of the unique vertex in level $-1$
is given by 
\begin{equation*}
    \int_{\oM_{0,m+1}}\psi_{1}^{N_{k,g}-1}\pi_{N_{k,g}*}\left(p_{-1*}[\PP\oSQ_{0}(k,(-2g(v)+1)_{v\in V_{0}},\underbrace{3,\dots,3}_{\times N_{k,g}}, R[-1])]^{\pm}\right).
\end{equation*}
Since $\Gamma$ is a back-bone graph, the global residue condition implies   
the residue vanishes on all poles parametrized by the vertices in level $0$ (see Lemma ~\ref{lem:BB-graph-residue}), that is $R[-1]=\{0\}$. Finally, as discussed in Remark ~\ref{rem:N!-factorial} and by the commutativity of forgetting the differential and the markings we have 
\begin{equation*}
    \pi_{N_{k,g}*}\left(p_{*}[\PP\oSQ_{0}(k,(-2g(v)+1)_{v\in V_{0}},\underbrace{3,\dots,3}_{\times N_{k,g}}, R[-1])]^{\pm}\right)=N_{k,g}!p_{*}[\PP\oSQ_{0}(k,(-2g(v)+1)_{v\in V_{0}},\{0\})]^{\pm},
\end{equation*}
which proves the desired result.
\end{proof}

In particular, the boundary contributions always come from back-bone graphs, with contribution the one schematically presented in the following figure.

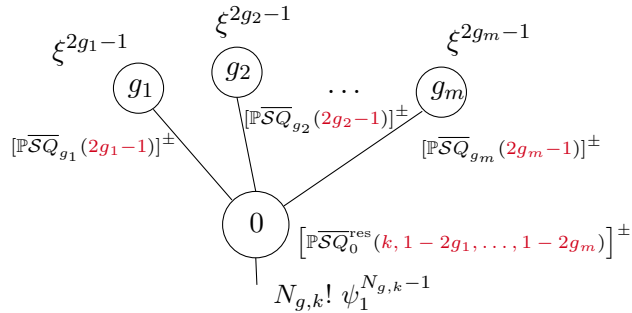
\[\begin{tikzpicture}[x=0.75pt,y=0.75pt,yscale=-1,xscale=1]
\draw   (238,98.22) .. controls (238,91.27) and (243.63,85.63) .. (250.58,85.63) .. controls (257.53,85.63) and (263.17,91.27) .. (263.17,98.22) .. controls (263.17,105.17) and (257.53,110.8) .. (250.58,110.8) .. controls (243.63,110.8) and (238,105.17) .. (238,98.22) -- cycle ;
\draw    (309.43,183.35) -- (310,196.92) ;
\draw    (258,108.8) -- (297.5,154.35) ;
\draw    (300,102.8) -- (309.43,150.03) ;
\draw    (393,109.8) -- (323.5,157.35) ;
\draw   (287.42,90.22) .. controls (287.42,83.27) and (293.05,77.63) .. (300,77.63) .. controls (306.95,77.63) and (312.58,83.27) .. (312.58,90.22) .. controls (312.58,97.17) and (306.95,102.8) .. (300,102.8) .. controls (293.05,102.8) and (287.42,97.17) .. (287.42,90.22) -- cycle ;
\draw   (292.78,166.69) .. controls (292.78,157.49) and (300.23,150.03) .. (309.43,150.03) .. controls (318.63,150.03) and (326.09,157.49) .. (326.09,166.69) .. controls (326.09,175.89) and (318.63,183.35) .. (309.43,183.35) .. controls (300.23,183.35) and (292.78,175.89) .. (292.78,166.69) -- cycle ;
\draw   (390,100.22) .. controls (390,93.27) and (395.63,87.63) .. (402.58,87.63) .. controls (409.53,87.63) and (415.17,93.27) .. (415.17,100.22) .. controls (415.17,107.17) and (409.53,112.8) .. (402.58,112.8) .. controls (395.63,112.8) and (390,107.17) .. (390,100.22) -- cycle ;

\draw (344,97.8) node [anchor=north west][inner sep=0.75pt]   [align=left] {$\displaystyle \dotsc $};
\draw (244,92.8) node [anchor=north west][inner sep=0.75pt]   [align=left] {$\displaystyle g_{1}$};
\draw (293,83.8) node [anchor=north west][inner sep=0.75pt]   [align=left] {$\displaystyle g_{2}$};
\draw (396,93.8) node [anchor=north west][inner sep=0.75pt]   [align=left] {$\displaystyle g_{m}$};
\draw (317.07,191) node [anchor=north west][inner sep=0.75pt]  [font=\small] [align=left] {$\displaystyle N_{g,k} !\ \psi _{1}^{N_{g,k} -1}$};
\draw (328.09,165.69) node [anchor=north west][inner sep=0.75pt]  [font=\tiny] [align=left] {$\displaystyle \left[\mathbb{P}\overline{\mathcal{S} Q}_{0}^{\text{res}}(\textcolor[rgb]{0.82,0.01,0.11}{k,1-2g}\textcolor[rgb]{0.82,0.01,0.11}{_{1}}\textcolor[rgb]{0.82,0.01,0.11}{,\dotsc ,1-2g}\textcolor[rgb]{0.82,0.01,0.11}{_{m}})\right]^{\pm }$};
\draw (185,119.8) node [anchor=north west][inner sep=0.75pt]  [font=\tiny] [align=left] {$\displaystyle [\mathbb{P}\overline{\mathcal{S} Q}_{g_{1}}(\textcolor[rgb]{0.82,0.01,0.11}{2g}\textcolor[rgb]{0.82,0.01,0.11}{_{1}}\textcolor[rgb]{0.82,0.01,0.11}{-1})]^{\pm }$};
\draw (206,69.8) node [anchor=north west][inner sep=0.75pt]   [align=left] {$\displaystyle \xi ^{2g_{1} -1}$};
\draw (288,55.8) node [anchor=north west][inner sep=0.75pt]   [align=left] {$\displaystyle \xi ^{2g_{2} -1}$};
\draw (404,64.8) node [anchor=north west][inner sep=0.75pt]   [align=left] {$\displaystyle \xi ^{2g_{m} -1}$};
\draw (302,105.8) node [anchor=north west][inner sep=0.75pt]  [font=\tiny] [align=left] {$\displaystyle [\mathbb{P}\overline{\mathcal{S} Q}_{g_{2}}(\textcolor[rgb]{0.82,0.01,0.11}{2g}\textcolor[rgb]{0.82,0.01,0.11}{_{2}}\textcolor[rgb]{0.82,0.01,0.11}{-1})]^{\pm }$};
\draw (391.07,120.8) node [anchor=north west][inner sep=0.75pt]  [font=\tiny] [align=left] {$\displaystyle [\mathbb{P}\overline{\mathcal{S} Q}_{g_{m}}(\textcolor[rgb]{0.82,0.01,0.11}{2g}\textcolor[rgb]{0.82,0.01,0.11}{_{m}}\textcolor[rgb]{0.82,0.01,0.11}{-1})]^{\pm }$};
\draw (305,160) node [anchor=north west][inner sep=0.75pt]   [align=left] {$\displaystyle 0$};

\end{tikzpicture}\]

As a corollary, we obtain the following computation of the boundary term. 

\begin{corollary}\label{cor:boundaryterm} In the notation of Lemma ~\ref{lem:agpm-recursion} and Proposition \ref{prop:2g-1-recursion}, we have 
\begin{equation*}
    p_{*}\left(\xi^{2g-1}\sum_{\substack{k=1, \\ k\ {\rm odd}}}^{2g-3}\frac{1}{N_{k,g}!}\frac{\Delta_{k}}{2}\prod_{\substack{\ell>k, \\ \ell\,{\rm odd}}}^{2g-3}\left(\frac{\xi+\ell\psi_{1}}{2}\right)\right)=\sum_{\substack{2\leq m\leq g,\\\underline{g}\in{\rm Part}(g)_{m}}}\frac{1}{m!} \left(\prod_{i=1}^{m}(2g_{i}-1)\,a_{g_{i}}^{\pm}\right)\,Hur(\underline{g}).
\end{equation*}
\end{corollary}

\begin{proof} First, we observe that for all $k$,
by Lemma ~\ref{lem:xi-vanishing} we have 
\begin{equation*}
    \xi^{2g-1}\sum_{\substack{k=1, \\ k\ {\rm odd}}}^{2g-3}\frac{1}{N_{k,g}!}\prod_{\substack{\ell>k \\ \ell \,{\rm odd}}}^{2g-3}\frac{(\xi+\ell\psi_{1})}{2}\frac{1}{2}\Delta_{k}=\sum_{\substack{k=1, \\ k\ {\rm odd}}}^{2g-3}\frac{1}{N_{k,g}!}\frac{1}{2}\left(\prod_{\substack{\ell>k \\ \ell \,{\rm odd}}}^{2g-3}\frac{\ell}{2}\right)\psi_{1}^{N_{k,g}-1}\xi^{2g-1}\Delta_{k}.
\end{equation*}
Moreover, we simplify 
\begin{equation}\label{eq:coef-Hur(g)}
    \frac{1}{2}\left(\prod_{\substack{\ell>k \\ \ell \,{\rm odd}}}^{2g-3}\frac{\ell}{2}\right)=\frac{(2g-3)!!}{k!!2^{N_{k,g}}}
\end{equation}
Furthermore, for each odd  $k\in\{1,\dots,2g-3\}$, the previous proposition describes the non-zero contribution of 
each term in $p_{*}(\psi_{1}^{N_{k,g}-1}\xi^{2g-1}\Delta_{k})$. In particular, the bi-colored graphs $\oGamma$ appearing in the definition of 
$\Delta_{k}$ (see ~\ref{def:Delta-k}) that contribute non-trivially have the following combinatorial structure:
\begin{enumerate}
    \item [(i)] $\oGamma$ is back-bone with a rational central vertex $v_{0}$ in level $-1$ carrying the unique leg, 
    and $m$ outlying vertices in level $0$ whose genera form a length $m$ partition of $g$.
    \item [(ii)] The twist $I$ at the unique 
    half-edge $h$ incident to each vertex $v\in V_{0}$
    is given by $I(h)=2g(v)-1$.
\end{enumerate} 
On the other hand, given a length $m\in\{1,\dots,g\}$ partition $\underline{g}$ of $g$ one can determine 
uniquely a back-bone graph $\oGamma$ 
with rational vertex satisfying the conditions $(i)$
and $(ii)$ above. Indeed, this is because the twist depends on the partition of $g$ given by the genera of the outlying vertices. Let, for each odd $k\in\{1,\dots,2g-3\}$, 
$\oGamma_{k,\underline{g}}$ be the back-bone graph 
determined by a length $m$ partition $\underline{g}$ of $g$
as discussed above.
Then we compute that\footnote{Note in the formula below we use $m\geq2$, since otherwise the stabilization of $\oGamma$ after forgetting markings would be trivial.}

\begin{align*}
    \frac{1}{N_{k,g}!}p_{*}\left(\psi_{1}^{N_{k,g}-1}\xi^{2g-1}\Delta_{k}\right)&=\sum_{\substack{2\leq m\leq g,\\\underline{g}\in{\rm Part}(g)_{m}}}\frac{\prod_{i=1}^{m}(2g_{i}-1)}{m!N_{k,g}!}p_{*}\left(\psi^{N_{k,g}-1}_{1}\xi^{2g-1}\pi_{N_{k,g}*}\zeta_{\oGamma}[\PP\oSQ(\oGamma_{k,\underline{g}})]^{\pm}\right) \\
        &=\sum_{\substack{2\leq m\leq g,\\\underline{g}\in{\rm Part}(g)_{m}}}\frac{\prod_{i=1}^{m}(2g_{i}-1)}{m!N_{k,g}!}N_{k,g}!\ \prod_{i=1}^{m}a_{g_{i}}^{\pm}\cdot h(k-1,-2\underline{g})
\end{align*}
Here, to go from the first to the second line, we used Proposition ~\ref{prop:boundary-contribution}, we substituted the multiplicity $m(\oGamma_{k,\underline{g}})$ by $\prod_{i=1}^{m}(2g_{i}-1)$, and rearranged the sum over graphs weighted by $1/|\Aut(\Gamma_{k,\underline{g}})|$ to a sum over unordered partitions of $g$ of length $m$ weighted by $1/m!$.
Finally, considering the sum over all odd $k$, the coefficient simplification in~\ref{eq:coef-Hur(g)} and changing the order of the summation between $k$ and partitions of $g$ we obtain 
\begin{equation*}
    p_{*}\left(\xi^{2g-1}\sum_{\substack{k=1, \\ k\ {\rm odd}}}^{2g-3}\frac{1}{N_{k,g}!}\prod_{\substack{\ell>k \\ \ell \,{\rm odd}}}^{2g-3}\frac{(\xi+\ell\psi_{1})}{2}\frac{1}{2}\Delta_{k}\right)=\sum_{\substack{2\leq m\leq g,\\\underline{g}\in{\rm Part}(g)_{m}}}\frac{1}{m!}\left(\prod_{i=1}^{m}(2g_{i}-1)a_{g_{i}}^{\pm}\right)Hur(\underline{g})
\end{equation*}
as desired.
\end{proof}

 Gathering the conclusions of the various propositions in this section, we obtain a complete proof of Lemma \ref{lem:agpm-recursion}. $\qed$

\section{Hurwitz integrals} \label{sec: hurwitz}

In this section we provide a closed formula for the factor $Hur$ appearing in the boundary contribution of Lemma \ref{lem:agpm-recursion}. We recall that, given a partition $\underline{g}$ of $g$, this factor was defined as 
 \[Hur(\underline{g}) \coloneqq \sum_{k=1,\ {\rm odd}}^{2g-3} \frac{(2g-3)!!}{k!!\,\,2^{N_{k,g}}}\,\,h(k-1,-2\underline{g}),\]
 where $N_{k,g} \coloneqq g-1 - \frac{k-1}{2}$ and
\begin{equation*}
    h(k-1,-2\underline{g})=\int_{\oM_{0,n+1}}\psi_{1}^{N_{k,g}-1}p_{*}[\PP\oSQ^{\rm res}_{0}(k,\ (-2g_i+1)_{i = 1}^n)].
\end{equation*}

In order to compute this formula, we will use that the integrals are in genus $0$. We will write $\psi_1$ in terms of boundary components of $\oM_{0,n+1}$ and use the boundary description of residueless strata (cf. \cite[Section 4]{CosMolZac}) in order to obtain a recursion formula. As we will see, $Hur(\underline{g})$ does not depend on the partition of $g$, but only on its length and number of entries equal to $1$. This will further simplify our computations. 

When we go to the boundary of $\PP\oSQ^{\rm res}_{0}(k,\ (-2g_i+1)_{i = 1}^n)$, we will see that  $0$-dimensional residueless strata will naturally appear. The cardinality count for this strata was already established in \cite{CheMolSauZag}:

\begin{proposition}[\cite{CheMolSauZag}]\label{prop:hurwitzCMSZ}
Let $n\geq 1$ and consider some integers $k_1,k_2\geq 0,\, d_1,\dots,d_n \geq 2$ satisfying $k_1 + k_2 -\sum_{i=1}^n d_i = -2$. Then the cardinality of $\PP\mathcal{H}^{\rm res}_{0}(k_1+1, k_2+1,\ (-d_i+1)_{i = 1}^n)$ is computed by the coefficient extraction
\[h^{\rm full}(k_1,k_2,-d_1,\dots, -d_n) \coloneqq (n-1)! [t^{k_1+1}]\left(\prod_{i=1}^n \frac{t-t^{d_i}}{1-t}\right).\]
\end{proposition}

\begin{remark} \label{remark: simplified h^1}
	We will rely on a simplified version of this where the differential has only two or three poles. In this situation we have 
	\[ h^{\rm full}(k_1,k_2,-d_1,-d_2) = \min\{k_1,k_2, d_1-1, d_2-1\}\]
    and 
    \[ h^{\rm full}(d_1+ d_2 +d_3-4,2,-d_1,-d_2, -d_3) = 2.\]
\end{remark}

Moreover, we will use the result of Gendron and Tahar, see \cite{GenTah} in order to describe which residueless strata are empty (and hence do not contribute in our sum of integrals). 

\begin{lemma}[\cite{GenTah}]\label{lem:pruning}
Let $k_1,\dots,k_p \geq 0$ and $d_1,\dots,d_n \geq 2$ be non-negative integers such that 
$$\sum _{i=1}^pk_i - \sum_{j=1}^n d_j = -2.$$ 
Assume there exists an index $1\leq i \leq p$ such that 
$k_i > \sum_{j=1}^n d_j - (n+1)$. Then we have
\[\mathcal{H}_0^\res(k_1\!+\!1,\dots,k_p\!+\!1, -d_1\!+\!1,\dots,-d_n\!+\!1) = \emptyset.\]
\end{lemma}

We will need the following corollary, which is just a numerical consequence of the lemma. 

\begin{corollary} \label{cor:pruning-v2}
Let $n\geq 4$ and consider some non-negative integers $k_1, k_2 \geq 0$ and $d_1,\dots,d_n \geq 2$ satisfying $\sum k_i - \sum d_j = -2$. Assume that $k_2 = 2$. Then we have
\[\oSQ_0^\res(k_1\!+\!1 ,k_2\!+\!1, -d_1\!+\!1,\dots,-d_n\!+\!1) = \emptyset.\]
\end{corollary}

\subsection{Hur formula}
 
We will next find a combinatorial formula for the Hurwitz numbers $Hur(\underline{g})$, where $\underline{g}$ is a length $n$ partition of $g$, with exactly $m$ entries equal to $1$. We state the result in Lemma \ref{lem:HurIC} and Theorem \ref{thm:Hurformula} and we will show inductively that they hold. To obtain a recursion of Hurwitz numbers, we will rely on the boundary description of residueless strata, which appears implicitly in \cite{BaiCheGenGruMol}, and on the genus $0$ interpretation of the psi classes in terms of boundary divisors. Both the Hurwitz numbers and the combinatorial formula satisfy the same recursion in $m$, $g$ and $n$. The initial condition $m=0$ looks as follows.

\begin{lemma}[Initial condition]\label{lem:HurIC} \label{lemma: formula no1}
Let $\underline{g}$ be a length $n$ partition of $g$, with no entries equal to $1$. Then we have \begin{align*}
Hur(\underline{g})
= \frac{(2g-3)!!2^{g-1}(g-1)!}{(2g-n)!} 
= (2g-2)\dots(2g-n + 1).
\end{align*}

\end{lemma}
\begin{proof} We will prove this by induction on the length $n$ of the partition $\underline{g}$,. We start with the case $n = 2$. Then the value
	\begin{equation*}
		h(k-1,-2\underline{g}) \coloneqq \int_{\oM_{0,3}}\psi_{1}^{N_{k,g}-1}p_{*}[\PP\oSQ^{\rm res}_{0}(k,\ -2g_1+1, -2g_2 +1)]
	\end{equation*}
can be non-zero only if $N_{k,g} \coloneqq g-1-\frac{k-1}{2} = 1$, since it is an integral on a $0$-dimensional space. When $k = 2g-3$, the space $\PP\oSQ^{\rm res}_{0}(k,\ -2g_1+1, -2g_2 +1) $ is $0$-dimensional and hence we have an equality between the intersection number and cardinality. Using this and Remark \ref{remark: simplified h^1} we obtain
$$h(2g-4, -2g_1, - 2g_2) = h^{\rm full}(2g-4, 2, -2g_1, -2g_2) = 2.$$

We obtain \[Hur(g_1,g_2) = \frac{h(2g-4,-2g_1,-2g_2)}{2} = 1\]
Now on the right-hand side we get \[\frac{(2g-3)!!2^{g-1}(g-1)!}{(2g-2)!} = 1.\]

\noindent Assume now that our lemma holds for all partitions of length less than $n>2$. Then, define the partitions $\underline{g}_{12} := ( \underline{g}\setminus (g_1,g_2),\ g_1+g_2-1)$ and $\underline{g}_{12i} := ( \underline{g}\setminus (g_1,g_2,g_i),\ g_1+g_2+ g_i-1)$. Furthermore, assume that $k$ is not $2g-3$.

In order to compute $h(k-1, -2\underline{g})$, we will use the equality $\psi_1 = \delta_{1|23}$, i.e. the sum of all boundary divisors in $\oM_{0,n+1}$ for which the marking corresponding to the zero of order $k$ is on one component, and the markings corresponding to the poles of orders $2g_1, 2g_2$ are on the other. Moreover, we will work with the multi-scale compactification $\PP\Xi\oM_{0,n+1+N_{k,g}}^{\rm res}(k, 3^{\times N_{k,g}}, -2\underline{g}+1)$, and we will use the boundary description of residueless strata, see \cite{BaiCheGenGruMol} and \cite{BurRosZvo}, in order to understand the pullback of $\Delta_{1|23}$. By abuse of notation, we will denote again 
\[ p \colon \PP\Xi\oM_{0,n+1+N_{k,g}}^{\rm res}(k, 3^{\times N_{k,g}}, -2\underline{g}+1) \rightarrow \oM_{0,n+1}  \]
the map forgetting the differential, as well as the $N_{k,g}$ zeroes of order $2$.

Because we want to compute the integral 
\[ \int_{\oM_{0,n+1}}\psi_1^{N_{k,g}-2}p_{*}[\PP\oSQ^{\rm res}_{0}(k,\ -2\underline{g}+1)]  \delta_{1|23} = \frac{1}{N_{k,g}!}\int_{\PP\Xi\oM_{0,n+1+N_{k,g}}^{\rm res}(k, 3^{\times N_{k,g}}, -2\underline{g}+1)}\psi_1^{N_{k,g}-2}  p^*\delta_{1|23}\]
we are only interested in understanding the boundary divisors in  $p^*\Delta_{1|23}$
where the component containing the zero of order $k-1$ has dimension $N_{k,g}-2$ and all other components are $0$-dimensional (all other divisors in $p^*\Delta_{1|23}$ will vanish when multiplied with $\psi_1^{N_{k,g}-2}$). Because the dimension of residueless strata is equal to the number of zeroes minus 2, we want to understand divisors in $p^*\Delta_{1|23}$ where the component containing the zero of order $k-1$ has an additional $N_{k,g}-1$ marked zeroes of order $2$. This implies that \begin{itemize}
    \item There is only one other component, and that component has a unique order $2$ marked zero, as well as another zero at the node. 
    \item This component contains the poles of order $2g_1$ and $2g_2$ by definition of $\Delta_{1|23}$, and contains at most one other pole (cf. Corollary \ref{cor:pruning-v2}). 
\end{itemize}

As such, these components are uniquely determined by what other pole is on the rational curve containing the first two poles.

For $j = 2$ and $3$, we denote by
\[ \zeta_j\colon \oM_{0, 2 + n- j} \times \oM_{0,1+j} \rightarrow \oM_{0,n+1} \]
the morphism glueing together the last marking of the first component to the first marking of the second one. Then we can compute the integral $\int_{\oM_{0,n+1}}\psi_1^{N_{k,g}-2}p_{*}[\PP\oSQ^{\rm res}_{0}(k,\ -2\underline{g}+1))]  \delta_{1|23}$  as 
\begin{align*}
\label{splittingproperty}
&\int_{\oM_{0,n+1}}\zeta_{2*}\Big(
\psi_1^{N_{k,g}-2}\frac{1}{(N_{k,g}-1)!}p_{*}[\PP\Xi\oM_{0,n+N_{k,g}-1}^{\rm res}
(k,3^{\times N_{k,g}-1},-2\underline{g}_{12}+1,\,-2g_1-2g_2-1)] \notag
\\
&\qquad\qquad\qquad\qquad
\otimes\ p_{*}[\PP\Xi\oM_{0,4}^{\rm res}
(2g_1+2g_2+1, 3 ,-2g_1+1,\,-2g_2+1]
\Big) \notag
\\
&
+\sum_{i=3}^{n}
\int_{\oM_{0,n+1}}\zeta_{3*}\Big(
\psi_1^{N_{k,g}-2}\frac{1}{(N_{k,g}-1)!}p_*[\PP\Xi\oM_{0,n+N_{k,g}-2}^{\rm res}
(k,3^{\times N_{k,g}-1},-2\underline{g}_{12i}+1,\,-2g_1-2g_2-2g_i-1)] \notag
\\
&\qquad\qquad\qquad\qquad
\otimes\ p_{*}[\PP\Xi\oM_{0,5}^{\rm res}
(2g_1+2g_2+2g_i+1,3,-2g_1+1,\,-2g_2+1,\,-2g_i+1)]
\Big).
\end{align*}

In the above formula, the factorial appear simply to account for the marked zeroes of order $2$. Moreover, we used that all contributing divisors in $p^*\Delta_{1|23}$ appear with multiplicity one. This is a residueless analogue of \cite[Lemma 6.4]{CosChenMol24}. Indeed, the same multi-scale compactification construction works for strata with residue conditions and the analogues of \cite[Equation (12.6) and Equation (12.8)]{BaiCheGenGruMol} imply the required multiplicities in the pullback of $\Delta_{1|23}$.

Going back to the recurrence of $h(k-1, -2\underline{g})$, the only terms that contribute in the computation are the ones for which there are either two or three poles on the component not containing the marked zero. Moreover, because of stability, any component of $\Delta_{1|23}$ must contain another marked pole on the component containing the marked zero. As such, we can write the formula above as the recursive formula 

\begin{align*} h(k-1,-2\underline{g}) &= h^{\rm full}(2g_1+2g_2-4,2, -2g_1, -2g_2)\cdot h(k-1,-2\underline{g}_{12}) \\&\qquad+ \delta_{n\geq 4} \cdot \sum_{i=3}^n h^{\rm full}(2g_1+2g_2 + 2g_i-4,2, -2g_1, -2g_2, -2g_i) \cdot h(k-1, -2\underline{g}_{12i}).
\end{align*}

By Remark \ref{remark: simplified h^1}, we obtain that 
\[ h^{\rm full}(2g_1+2g_2-4,2, -2g_1, -2g_2) =h^{\rm full}(2g_1+2g_2 + 2g_i-4,2, -2g_1, -2g_2, -2g_i) = 2. \] 

Moreover, the partitions $\underline{g}_{12}$ and $\underline{g}_{12i}$ of $g-1$ do not have entries equal to $1$, and have lengths $n-1$ and $n-2$ respectively. Thus, we have the formula
\begin{align*}
    Hur(\underline{g}) &\coloneqq \sum_{k=1,\ {\rm odd}}^{2g-3} \frac{(2g-3)!!}{k!!\,\,2^{N_{k,g}}}\,\,h(k-1,-2\underline{g}) \\
    & = \frac{h(2g-4, -2\underline{g})}{2} + (2g-3)\sum_{k=1,\ {\rm odd}}^{2g-5} \frac{(2g-5)!!}{k!!\,\,2^{N_{k,g-1}}} \left(h(k-1, \underline{g}_{12}) + \sum_{i=3}^nh(k-1, \underline{g}_{12})\right) \\
    &= \delta_{n \leq 3} + (2g-3)Hur(\underline{g}_{12}) + (2g-3) \delta_{n\geq 4}\sum_{i=3}^nHur(\underline{g}_{12i}).
\end{align*}

Using the induction hypothesis, together with the equality 
\[ Hur(\underline{g}_{12i}) = Hur(\underline{g}_{12j})  \]
for any $3\leq i, j \leq n$, we obtain
\[ Hur(\underline{g}) = \delta_{n \leq  3} + \frac{(2g-3)!!2^{g-1}(g-1)!}{(2g-n)!} \Bigg(\frac{2g-n}{2g-2} + \delta_{n\geq 4}\frac{n-2}{2g-2}\Bigg) \]
and the conclusion follows. 

\end{proof}
Similar ideas can be used for the case $m > 0$. We obtain: 
\begin{theorem}[$Hur$ formula]\label{thm:Hurformula}\label{thm:hur-b(a)} Let $\underline{g}$ be a length $n$ partition of $g$ with exactly $m$ of its entries equal to $1$. Then we have 
 \[Hur(\underline{g}) = (2g-3)!!\sum_{i=0}^{g-1}(-1)^i\begin{pmatrix}m\\i\end{pmatrix}\ 2^{g-2i-1}\frac{(g-i-1)!}{(2g-i-n)!},\] 
where $\binom{m}{i} \coloneqq 0 $ if $i > m$. 
\end{theorem}

\begin{proof}
If $m=0$ we are done by Lemma \ref{lem:HurIC}. We will assume the theorem to be correct when the number of entries equal to $1$ is less than $m$, and prove it for $m$. We then split the problem into two cases that we will treat separately: \begin{itemize}
    \item Case I: we have $1\leq m \leq n-1$ and 
    \item Case II: we have $m = n = g$. 
\end{itemize}

In the first case, the same method of computing $h(k-1, -2\underline{g})$ as in Lemma \ref{lem:HurIC} can be applied here. Since this integral is independent of the ordering in the partition $\underline{g}$, we assume $g_1 > 1$ and $g_2 = 1$. Then the same method as in Lemma \ref{lem:HurIC} yields 

\begin{align*} h(k-1,-2\underline{g}) &= h(2g_1+2g_2-4, -2g_1, -2g_2)\cdot h(k-1,-2\underline{g}_{12}) \\&\qquad+ \delta_{n\geq 4} \cdot \sum_{i=3}^n h(2g_1+2g_2 + 2g_i-4, -2g_1, -2g_2, -2g_i) \cdot h(k-1, -2\underline{g}_{12i}).
\end{align*}

In this situation, Remark \ref{remark: simplified h^1} implies
\[ h(2g_1-2, -2g_1, -2)  = h^{\rm full}(2g_1-2,2, -2g_1, -2) = 1 \] 
and similarly
\[ h(2g_1 + 2g_i-2, -2g_1, -2, -2g_i) = h^{\rm full}(2g_1 + 2g_i-2,2, -2g_1, -2, -2g_i) = 2.\] 
After summing up (with suitable coefficients) over all $k$, we obtain the formula 
\[ Hur(\underline{g}) =  \frac{1}{2}\delta_{n=2} + \delta_{n=3} + \frac{(2g-3)}{2}\delta_{n\geq 3}Hur(\underline{g}_{12}) + (2g-3) \delta_{n\geq 4}\sum_{i=3}^nHur(\underline{g}_{12i}).\]
In this case, we know that the partition $\underline{g}_{12}$ of $g-1$ has $n-1$ entries, out of which $m-1$ are equal to $1$. Similarly, the partition $\underline{g}_{12i}$ of $g-1$ has $n-1$ entries, and it has \begin{itemize}
    \item $m-1$ entries equal to $1$ if $g_i > 1$, 
    \item $m-2$ entries equal to $1$ if $g_i = 1$. 
\end{itemize}
Furthermore, we used that 
\[
h(2g-4,-2\underline{g})
=
\begin{cases}
1, & \text{if } n = 2,\\
2, & \text{if } n = 3.
\end{cases}
\]

We denote $H(m,g,n) \coloneqq (2g-3)!!\sum_{i=0}^{g-1}(-1)^i\begin{pmatrix}m\\i\end{pmatrix}\ 2^{g-2i-1}\frac{(g-i-1)!}{(2g-i-n)!}$. Since we have $m-1$ indices $3\leq i\leq n$ for which $g_i = 1$, we obtain the formula 
\begin{align*}
    Hur(\underline{g}) =  \frac{1}{2}\delta_{n=2} &+ \delta_{n=3} + \frac{(2g-3)}{2}\delta_{n\geq 3} H(m-1,g-1,n-1) \\ &+(2g-3) (m-1) \delta_{n\geq 4} H(m-2,g-1,n-2)  \\
    &+ (2g-3) (n-m-1)\delta_{n\geq 4} H(m-1,g-1,n-2) .
\end{align*} 

It is then an immediate consequence of Proposition \ref{prop:CombiRec} below that $Hur(\underline{g}) = H(m,g,n)$, which is what we wanted to prove.

We are left to treat the case $\underline{g} = (\underbrace{1,1,\ldots,1}_{g \ \textrm{times}})$. We will assume that $g\geq 4$, since for low genera a simple check gives the required equalities, i.e. $Hur(1,1) = 0$ and $Hur(1,1,1) = \frac{7}{4}$. We will assume the theorem holds up for all partitions of length $g-1$ or smaller, and prove it for length $g$. As before, in order to compute $h(k-1, -2\underline{g})$, we will use the equality $\psi_1 = \Delta_{1|23}$ and reason inductively. As such, we have to describe the components of the intersection $p_{*}\PP\oSQ^{\rm res}_{0}(k,\ -2\underline{g}+1)) \cap \Delta_{1|23}$. These depend on the indices $H \subseteq \{1,2,\ldots g\}$ corresponding to the poles that are not on the same component as the marked zero, and on the order $k'$ at the node. 

As in the proof of Lemma \ref{lem:HurIC}, we can only get non-zero contribution from the integral when 

\[ \dim \PP\oSQ^{\rm res}_{0}(k',\ -2\underline{g}_H + 1)  ) = 0 \ \textrm{and}\]
\[ \dim \PP\oSQ^{\rm res}_{0}(k,\ -2\underline{g}_{H^c} +1, -k') = N_{k,g} - 2. \]

We distinguish two different cases.
\begin{itemize}
    \item If $k' \geq 0$, then the first dimension condition implies $k' = 2\sum_{i\in H}g_i -3 = 2|H|-3$. In this case, Lemma \ref{lem:pruning} and Corollary \ref{cor:pruning-v2} imply $|H| = 3$ and hence $k' = 3$. Thus for any $|H| = 3$ with $\{1,2\}\subseteq H$ we obtain a contribution to the integral coming from the boundary loci 
   \[ \zeta_{H*}\left(p_{*}\PP\oSQ^{\rm res}_{0}(k,\ \underbrace{-1,\ldots,-1}_{g-3 \ \textrm{times}}, -3) \times p_{*}\PP\oSQ^{\rm res}_{0}(3,\ -1, - 1, -1)  \right). \]
   
    \item If $k' < 0$, the first dimension condition implies that $|H| = 2$ and there are two unmarked zeroes corresponding to the differential on the component. Thus $k' = -1$ and we obtain a contribution to the integral coming from the boundary locus
    \[ \zeta_{H*}\left(p_{*}\PP\oSQ^{\rm res}_{0}(k,\ \underbrace{-1,\ldots,-1}_{g-2 \ \textrm{times}}, 1) \times p_{*}\PP\oSQ^{\rm res}_{0}(-1, - 1, -1)\right). \]
\end{itemize}

Since there are $g-2$ choices of $H$ in the first case, and only one in the second, we multiply $p_{*}\PP\oSQ^{\rm res}_{0}(k,\ -2\underline{g}+1)) \cap \Delta_{1|23}$ with $\psi_1^{N_{k,g}-2}$ and obtain the recursion formula 
\[ h(k-1, -2\underline{g}) = (g-2)h(2, -2, -2, -2) h(k-1, \underbrace{-2,\ldots, -2}_{g-3 \ \textrm{times}}, -4) + h(-2, -2, -2) h(k-1, \underbrace{-2,\ldots, -2}_{g-2 \ \textrm{times}}). \]

We have the equalities 
\[ h(2,-2, -2, -2) = h^{\rm full}(2, 2, -2, -2, -2) = 2  \ \textrm{and}\]
\[ h(-2,-2,-2) = \frac{h^{\rm full}(2, 2, -2, -2, -2)}{2} = 1.\]
For the second cardinality we divided by $2$ since the order of the two unmarked zeros is irrelevant. We thus have
\[ h(k-1, -2\underline{g}) = 2(g-2)h(k-1, \underbrace{-2,\ldots, -2}_{g-3 \ \textrm{times}}, -4) + h(k-1, \underbrace{-2,\ldots, -2}_{g-2 \ \textrm{times}}). \]
Substituting this formula, along with $h(2g-4,\underbrace{-2,\ldots, -2}_{g \ \textrm{times}}) = 0$ (cf. Lemma \ref{lem:pruning}), in the definition of $Hur$ we obtain 
\[ Hur(\underbrace{1,\ldots, 1}_{g \ \textrm{times}}) = (2g-3)(g-2)Hur(\underbrace{1,\ldots, 1}_{g-3 \ \textrm{times}},2) + \frac{(2g-3)(2g-5)}{4}Hur(\underbrace{1,\ldots, 1}_{g-2 \ \textrm{times}}).\]
The formula for $ Hur(\underbrace{1,\ldots, 1}_{g \ \textrm{times}})$ is now a consequence of the induction hypothesis, the formula for $Hur(\underbrace{1,\ldots, 1}_{g-3 \ \textrm{times}},2)$ and of Proposition \ref{prop:CombiRec-1g} below. This concludes the proof.
\end{proof}

\subsubsection{Combinatorics}

In order to complete Theorem \ref{thm:hur-b(a)}, we are left to find recurrence formulas for the values 
$$H(m,g,n) \coloneqq (2g-3)!!\sum_{i=0}^{g-1}(-1)^i\begin{pmatrix}m\\i\end{pmatrix}\ 2^{g-2i-1}\frac{(g-i-1)!}{(2g-i-n)!}.$$ We already saw in Lemma \ref{lem:HurIC} and in the proof of Theorem \ref{thm:hur-b(a)} that the function $H$ has the same initial values as the Hurwitz integrals. We are left to show they satisfy the same recurrence formulas. 

In order to simplify the combinatorial arguments, we consider 
\[H'(m,g,n) = \frac{2^{g+1}}{(2g-3)!!}H(m,g,n) = \sum_{i=0}^{g-1}(-1)^i\binom{m}{i}\frac{2^{2g-2i}(g-i-1)!}{(2g-n-i)!}.\]

The recurrence for the right-hand side of Theorem~\ref{thm:Hurformula} takes the following shape for $m \neq g$.
\begin{proposition}[Combinatorial recurrence]\label{prop:CombiRec}
    Let $g\geq n > m \geq 1$ and $ n\geq 3$. Then the function above satisfies the recurrence 
    \begin{align*}
        H(m,g,n) &= \frac{2g-3}{2}\cdot H(m-1,g-1,n-1) + (n-m-1)(2g-3)H(m-1,g-1,n-2) \\ 
        &+ (m-1)(2g-3)H(m-2,g-1,n-2)
    \end{align*}
\end{proposition}
\begin{proof}
    In terms of the function $H'$ we want to prove
      \[H'(m,g,n) = H'(m-1,g-1,n-1) + 2(n-m-1)H'(m-1,g-1,n-2) + 2(m-1)H'(m-2,g-1,n-2). \]
    Using the identity $\binom{m}{i} = \binom{m-1}{i} + \binom{m-1}{i-1}$, we obtain 
    \[ H'(m,g,n) = -H'(m-1,g-1,n-1) + \sum_{i=0}^{m-1}(-1)^i\binom{m-1}{i}\cdot \frac{2^{2g-2i}(g-i-1)!}{(2g-n-i)!}.\]
    Next we compute 
    \begin{align*}
        2(m-1)H'(m-2,g-1,n-2) &= 2(m-1)\sum_{i=0}^{m-2}(-1)^i\binom{m-2}{i} \cdot \frac{2^{2g-2-2i}\cdot(g-1-i-1)!}{(2g-2-n+2-i)!} \\
        & = \sum_{i=0}^m (-1)^i \binom{m-1}{i} \cdot \frac{2^{2g-2i}(g-i-1)!}{(2g-n-i)!} \cdot \frac{m-1-i}{2(g-i-1)}. 
    \end{align*}
    Similarly
$$        2(n-m-1)H'(m-1,g-1,n-2) = \sum_{i=0}^m (-1)^i  \binom{m-1}{i} \cdot \frac{2^{2g-2i}(g-i-1)!}{(2g-n-i)!} \cdot \frac{n-m-1}{2(g-i-1)}.$$
We now look at 
$$\sum_{i=0}^{m-1}(-1)^i\binom{m-1}{i}\cdot \frac{2^{2g-2i}(g-i-1)!}{(2g-n-i)!} - 2(n-m-1)H'(m-1,g-1,n-2) - 2(m-1)H'(m-2,g-1,n-2) $$ 
which is equal to 
$$\sum_{i=0}^m(-1)^i\binom{m-1}{i}\cdot \frac{2^{2g-2i}(g-i-1)!}{(2g-n-i)!} \cdot (1- \frac{n-m-1}{2(g-i-1)}- \frac{m-1-i}{2(g-1-i)})$$
\[ = \sum_{i=0}^m(-1)^i\binom{m-1}{i}\cdot \frac{2^{2g-2i}(g-i-1)!}{(2g-n-i)!} \cdot \frac{2g-2-2i-n+m+1-m+i+1}{2(g-i-1)}   \]
\[= \sum_{i=0}^m(-1)^i\binom{m-1}{i}\cdot \frac{2^{2g-2i}(g-i-1)!}{(2g-n-i)!} \cdot \frac{2g-n-i}{2(g-i-1)}. \]
After simplifying the fraction, we obtain that the sum is equal to $2H'(m-1,g-1,n-1)$, which is what we wanted to prove. 
\end{proof}

While the combinatorial identity of Proposition \ref{prop:CombiRec} holds for $n = 3$, this is not the form that we needed in Theorem \ref{thm:Hurformula}. Using the equality $H(0,g-1,1) = \frac{1}{2g-3}$ we obtain that 
\[ (n-m-1)(2g-3)H(m-1,g-1,n-2) + (m-1)(2g-3)H(m-2,g-1,n-2) =  1\]
for $n=3$ and $1\leq m \leq 2$. We can thus rewrite the combinatorial identity in the form required in Theorem \ref{thm:Hurformula}: 
\[ H(m,g,3) = 1  + \frac{2g-3}{2}H(m-1,g-1,2)\]
for $m \in \{1,2\}$. For the same reason, we obtain $ H(1,g,2) = \frac{1}{2}$. 

The recurrence in Theorem~\ref{thm:Hurformula} takes the following shape when $m = g$.
\begin{proposition}[Combinatorial recurrence 2]\label{prop:CombiRec-1g}   
    For $g \geq 4$ we have the recursion formula
    \[H(g,g,g) = \frac{(2g-5)(2g-3)}{4}H(g-2,g-2,g-2) + (g-2)(2g-3)H(g-3,g-1,g-2). \]
\end{proposition}
\begin{proof}
After scaling and rewritting the identity in terms of $H'$, we need to prove
\[H'(g,g,g) = H'(g-2,g-2,g-2) + 2(g-2)H'(g-3,g-1,g-2). \]

    Using the identity $\binom{g}{i} = \binom{g-2}{i} + 2\binom{g-2}{i-1} + \binom{g-2}{i-2}$ we obtain 
    \[ H'(g,g,g) = \sum_{i=0}^{g-1}(-1)^i\cdot \text{\Huge[}\binom{g-2}{i} + 2\binom{g-2}{i-1} + \binom{g-2}{i-2}\text{\Huge ]} \cdot \frac{2^{2g-2i}(g-i-1)!}{(g-i)!}. \]
    We also have 
    \begin{align*}
        H'(g-2,g-2,g-2) &= \sum_{i=0}^{g-2}(-1)^i\binom{g-2}{i}\cdot \frac{2^{2g-4-2i}(g-3-i)!}{(g-2-i)!} \\
        & = \sum_{i=2}^{g}(-1)^i\binom{g-2}{i-2}\cdot \frac{2^{2g-2i}(g-1-i)!}{(g-i)!}
    \end{align*}
    In particular, we have 
    \[ H'(g,g,g)-H'(g-2,g-2,g-2) = \sum_{i=0}^{g-1}(-1)^i\cdot \text{\Huge[}\binom{g-2}{i} + 2\binom{g-2}{i-1}\text{\Huge ]} \cdot \frac{2^{2g-2i}(g-i-1)!}{(g-i)!},  \]
    and want this value to be equal to $2(g-2)H'(g-3,g-1,g-2)$.

    We have
    \begin{align*}
        2(g-2)H'(g-3,g-1,g-2) &= 2(g-2)\sum_{i=0}^{g-3} (-1)^i \frac{(g-3)!}{i!(g-3-i)!} \frac{2^{2g-2-2i}\cdot (g-2-i)!}{(g-i)!}  \\
        & = \sum_{i=0}^{g-3} (-1)^i \binom{g-2}{i}\cdot \frac{2^{2g-2i}(g-i-1)!}{(g-i)!} \cdot \frac{g-2-i}{2(g-i-1)} 
    \end{align*}
    Finally, we look at the difference 
    \begin{align*}
        &\sum_{i=0}^{g-2} (-1)^i \binom{g-2}{i} \cdot \frac{2^{2g-2i}(g-i-1)!}{(g-i)!} - 2(g-2)H'(g-3,g-1,g-2)  \\
        &= \sum_{i=0}^{g-2} (-1)^i \binom{g-2}{i} \cdot \frac{2^{2g-2i}(g-i-1)!}{(g-i)!} \text{\Large [}1-\frac{g-2-i}{2(g-i-1)}\text{\Large ]} \\
         &= \sum_{i=0}^{g-2} (-1)^i \binom{g-2}{i} \cdot \frac{2^{2g-2i}(g-i-1)!}{(g-i)!} \cdot \frac{g-i}{2(g-i-1)} \\
         &= \sum_{i=1}^{g-1} (-1)^{i-1} \binom{g-2}{i-1} \frac{2^{2g-2i+2}(g-i-1)!}{(g-i)!} \cdot \frac{1}{2}\\
         &= 2\sum_{i=0}^{g-1} (-1)^{i-1} \binom{g-2}{i-1} \frac{2^{2g-2i}(g-i-1)!}{(g-i)!}. \\
    \end{align*}  
    The conclusion follows. 
\end{proof}

For completeness, we remark that the splitting property used to provide the geometric recurrence relations above holds more generally. We shall not use this result in this article.
\begin{remark}[pCohFT property]
\noindent The notation $S_g(\alpha_1,\dots,\alpha_n)$ indicates either of \begin{align*}
&p_*[\PP\oH_{g}^{\res}(\alpha_1\!+\!1,\dots, \alpha_n\!+\!1)] \in  A^{-g+1 + \#\{\alpha_i <0\} + |\alpha|}(\overline{\mathcal{M}}_{g,n}),\\
&p_*[\PP\oSQ_{g}^{\res}(\alpha_1\!+\!1,\dots, \alpha_n\!+\!1)], \quad  p_*[\PP\oSQ_{g}^{\res}(\alpha_1\!+\!1,\dots, \alpha_n\!+\!1)]^\pm \in  A^{\#\{\alpha_i <0\} + |\alpha|/2}(\overline{\mathcal{M}}_{g,n}),
\end{align*} with $\alpha_i$ in $\ZZ^\pm = 2\ZZ$, or $\ZZ^\emptyset = \ZZ\setminus\{-1\}$ as appropriate. We write $S_{g}^{\rm full}(\alpha_1,\dots,\alpha_n)$ to emphasize when the order condition is full, i.e. the sum of the specified orders is maximal. In the full case these classes form a partial cohomological field theory of infinite rank related to the KP hierarchy \cite{BurRosZvo} (or BKP hierarchy \cite{KloVel} for spin-parity). We extend this pCohFT result to the non-full case. 
\begin{proposition}\label{prop:pcohft}
The classes $S_g(\alpha_1,\dots,\alpha_n)$ in cohomology define a partial cohomological field theory of infinite rank\footnote{As in \cite{BR21b}.} over $\left(V:= \bigoplus_{\alpha\in \ZZ^\bullet}\QQ\cdot e_\alpha\,,\quad \eta_{\alpha\beta}:=\delta_{\alpha+\beta,-2}\,,\quad \mathbf{1}:= e_0\right),$ with $\bullet = \emptyset, \pm$ as appropriate.
\end{proposition}
\begin{proof} 
We proof it without loss of generality for the classes of cones of squares.
The symmetry is clear.
The case of full order conditions $\sum_i \alpha_i = 2g-2$ is \cite[Proposition 1.8]{BurRosZvo} (or similarly \cite[Proposition 1.7]{KloVel} in the spin-parity case). By Remark \ref{rem:N!-factorial}, we have $$S_g(\alpha_1,\dots,\alpha_n) = \frac{(\pi_N)_*}{N!}S_g^{\rm full}(\alpha_1,\dots,\alpha_n,2,\dots,2)$$ 
for $N$ such that $\sum_i \alpha_i + 2N = 2g-2$.

Let $\Gamma$ be the stable graph separating a $g_1$ vertex named $1$ with markings $I$ from a $g_2$ vertex with markings $J$ for partitions $g_1+g_2 = g$ and $I\sqcup J = \{1,\dots,n\}$. Let $\Gamma_{N_1,N_2}$ for some natural numbers $N_i$ be the same graph, appended by $N_i$ additional markings on the vertex named $i$.
Then a standard push-pull argument gives \begin{align*}
(\zeta_\Gamma)^*S_g(\alpha) &= \frac{1}{N!}(\zeta_\Gamma)^*\ (\pi_N)_*\ S_g^{\rm full}(\alpha,2^N)\\
&=  \frac{1}{N!}\sum_{\substack{N_1,N_2 \geq 0\\ N_1+N_2 = N}}\binom{N}{N_1,N_2} \ (\pi_{N_1} \times \pi_{N_2})_* \ (\zeta_{\Gamma_{N_1,N_2}})^*\ S_g^{\rm full}(\alpha,2^N)\\
&=  \frac{1}{N!}\sum_{\substack{N_1,N_2 \geq 0\\ N_1+N_2 = N}}\binom{N}{N_1,N_2} \ (\pi_{N_1})_* \otimes (\pi_{N_2})_* \ \sum_{\mu\in 2\ZZ} \delta_{\mu,N_1,N_2}\bigg(S_{g_2}^{\rm full}(\alpha_I,2^{N_1},\mu)\otimes S_{g_1}^{\rm full}(\alpha_J,2^{N_2},-\mu-2)\bigg)\\
&=\sum_{\substack{N_1,N_2 \geq 0\\ N_1+N_2 = N}} \sum_{\mu\in 2\ZZ}
\delta_{\mu,N_1,N_2}\bigg(S_{g_2}(\alpha_I,\mu)\otimes S_{g_1}(\alpha_J,-\mu-2)\bigg)\\
&= \sum_{\mu\in 2\ZZ} S_{g_1}(\alpha_I,\mu)\otimes S_{g_2}(\alpha_J,-\mu-2).
\end{align*}
Here the symbol $\delta_{\mu,N_1,N_2}$ enforces the full-ness condition by vanishing unless $2g_1 -2 = \sum \alpha_I + \mu + 2N_1$ and $2g_2 -2 = \sum\alpha_J - \mu - 2 + 2N_2$, in which case it has value $1$.
Note that this is a finite sum, since the sums of the orders satisfy $2g_1 -2 = \sum \alpha_I + \mu + 2N_1$ and $2g_2 -2 = \sum\alpha_J - \mu - 2 + 2N_2$ for some non-negative $N_1,N_2$. Indeed, if for example $\mu$ runs to negative infinity, $N_1$ can grow to compensate, but then $-\mu-2$ grows larger, which $N_2$ is unable to compensate.

For the unit axiom consider (c.f. \cite[Lemma 2.11]{PanPixZvo}),
\[\pi^*S_g(\alpha) = \pi^*\frac{(\pi_N)_*}{N!}S^{\rm full}_g(\alpha,2^N) = \frac{(\pi_N)_*}{N!}\pi^*S^{\rm full}_g(\alpha,2^N) = \frac{(\pi_N)_*}{N!}S^{\rm full}_g(\alpha,0,2^N) = S_g(\alpha,0).\]
Furthermore, if $\alpha+\beta$ exceeds $-2$, $S_0(\alpha,\beta,0)$ is zero. If their sum equals $-2$, we are in the full case for which the unit axiom is satisfied. In the remaining case, the class vanishes for dimension reasons or $N = 1$. Then, $S^{\rm full}_0(\alpha,\beta,2)$ vanishes by inspection of Lemma \ref{lem:pruning} and so 
\[S_0(\alpha,\beta,0) = (\pi_4)_*S^{\rm full}_0(\alpha,\beta,0,2) = (\pi_4)_*\pi_3^*S^{\rm full}_0(\alpha,\beta,2) = 0.\]
Collecting these results we see that $S_0(\alpha,\beta,0) = \delta_{\alpha+\beta,-2}$.
\end{proof}

\end{remark}

\section{Solving the recursion} \label{sec: solve recursion}

Now that all the terms in the volume recursion are understood explicitly, it is solved formally via coordinate transformations and the following observation for the base case.

\begin{example}(genus 1)
Note that the canonical spin structure is trivial on curves $E$ in $\cH_{1,1}(1)$ and of parity $h^0(E, \mathcal{O}_E) = 1 \text{ mod } 2$. Thus, $\cH_{1,1}(1)^-=\cH_{1,1}(1) =\SQ_1(1)=\cH_{1,1}$. Since the $\CC^\times$-action on this cone of squares scales the 1-dimensional family of holomorphic 1-forms, $\PP\mathcal{SQ}_1$ is the moduli space of elliptic curves itself. The tautological line bundle $\mathcal{O}(-1)$ then coincides with the Hodge bundle. After taking the closures we get \[a_1 = \int_{[\PP\oSQ_1]} \xi = - \int_{\oM_{1,1}}\lambda_1 = -b_1,\] and equivalently  $a_1^\pm = -b_1^\pm$. \qed
\end{example}

\noindent Recall the notation $$D(g) = (-1)^{g}\ 2^{2g-2}\ (g-1)!.$$

\noindent Define three families of polynomials 
with $g>0$ \begin{align*}
M_g &:= \frac{1}{(2g-1)}\sum_{\substack{n\geq1,\, g_1,\dots,g_n\geq 1\\ g_1+\dots+ g_n = g}}\frac{(1-2g)^{n-1}}{n!}\prod_{i=1}^n(2g_i-1)!(-1)^{g_i}x_{g_i} &&\in \QQ[x_1,\dots,x_g]\subset \QQ[x_1,x_2,\dots]\\ 
S_g &:= \sum_{m=0}^{g-1}\frac{(-1)^m}{m!}\ \left(z_1\right)^m\ z_{g-m}  &&\in\QQ[z_1,\dots,z_g]\subset \QQ[z_1,z_2,\dots]\\
R_g &:= \frac{2^{g-1}}{(2g-3)!!}\left(y_g + \sum_{\substack{m \geq 2,\\ \underline{g} \in {\rm Part}(g)_{m}}}\prod_i (2g_i-1)y_{g_i} \, \cdot \frac{Hur(\underline{g})}{m!}\right)  &&\in \QQ[y_1,\dots,y_g]\subset\QQ[y_1,y_2,\dots]\\
\end{align*}
The following Proposition shows that the combinatorics of the recursion in Lemma \ref{lem:agpm-recursion} and its (to be proven) solution in Theorem \ref{thm:maintheorem} are equivalent.
\begin{proposition}\label{prop:rec=main}
The family of coordinate transformations \[y_g = M_g(x_1,\dots, x_g)\]
coincides with the family of coordinate transformations defined by \[ R_g(y_1,\dots,y_g)= S_g\bigg(D(1)\cdot x_1,\dots, D(g)\cdot x_g\bigg).\]
\end{proposition}
\begin{proof}
Note firstly by induction that these families of coordinate transformations are well-defined and invertible, e.g. $M_g$ reads
\begin{align*}
y_1 &= -x_1\\
y_2 &= \frac{1}{3} \left( 3!\ x_2 + \frac{-3}{2!}x_1^2\right)\\
y_3 &= \frac{1}{5}\left( 5!\ (-x_3) + \frac{-5}{2!}3!\ x_2 (-x_1) + \frac{5^2}{3!} (-x_1)^3\right)\\
&\vdots
\end{align*}
As in \cite{CosSauSch}, we use the combinatorial result of \cite[Corollary 2.(ii)]{BirGilWei} to write the above coordinate inversion as 
\begin{equation}\label{eq:Bellinverse}x_g = M_g^{\rm inv}(y_1,\dots,y_g):= \sum_{n\geq 1}\frac{1}{n!(2g-n)!}\sum_{g_1+\dots + g_n = g}\prod_{j=1}^n(2g_j-1)(-1)^{g_j}y_{g_j}.\end{equation}
In order to then prove the asserted coincidence, we write all in terms of the $y$-variables and compare coefficients. That is, we check that the following family of equations holds:
\[ R_g(y_1,\dots,y_g) = S_g\bigg(D(1)\cdot M_1^{\rm inv}(y_1),\dots, D(g)\cdot M_g^{\rm inv}(y_1,\dots,y_g)\bigg).\]

\noindent For convenience, write 

\begin{align*}
d_g(\underline{r})&:= \left((-1)^{g-1}\frac{2^{2(g-m)-2}}{m!}\frac{2^{2(g-m)-1}\,(g-m-1)!}{2^{g-m}\,({2^{2(g-m)-1}-1})}\right)\times \\ 
&\left( \frac{{2^{2(g-m)-1}-1}}{{-2^{(g-m)-1}}\ n!(2(g-m)-n)!}(-1)^{g-m}\prod_{j=1}^n(2r_j-1)\right)   
\end{align*}
such that 
\begin{align*}
    Q_g(y_1,\dots,y_g)&:=S_g\bigg(D(1)\cdot M_1^{\rm inv}(y_1),\dots, D(g)\cdot M_g^{\rm inv}(y_1,\dots,y_g)\bigg) \\
    &=\sum_{m=0}^{g-1}(y_1)^m\ \left(\sum_{\substack{r_1+\dots + r_n = g-m  \\ n\geq 1}}d_g(\underline{r})\prod_{j=1}^ny_{r_j}\right).
\end{align*}

\noindent 
Note that the coefficient of $y_g$ in $R_g, Q_g$ can be seen to coincide using the identity $$ (2k-1)!! = (2k)!/(2^k k!).$$

\noindent It remains to prove \begin{align*}
\frac{2^{g-1}}{(2g-3)!!}\left(\sum_{\substack{m \geq 2,\\ \underline{g} \in {\rm Part}(g)_{m}}}\prod_i (2g_i-1)y_{g_i} \, \cdot \frac{Hur(\underline{g})}{m!}\right) = \sum_{m=1}^{g-1}(y_1)^m\ \left(\sum_{\substack{r_1+\dots + r_n = g-m  \\ n\geq 1}}d_g(\underline{r})\prod_{j=1}^ny_{r_j}\right).
\end{align*}
This follows from inspection of the formula for $Hur$ in Theorem \ref{thm:Hurformula}.
\end{proof}

\begin{proof}[Proof of Main Theorem~\ref{thm:maintheorem}]\label{proofofmain}
For $g=1$ this follows from the above example.
By Lemma \ref{lem:agpm-recursion}, \[R_g(a_1^\pm,\dots,a_g^\pm) = \beta_g^\pm = S_g\bigg(D(1)\cdot b_1^\pm, \dots, D(g)\cdot b_g^\pm\bigg) \in \QQ.\] By Proposition \ref{prop:rec=main}, this is equivalent to \[a_g^\pm = M_g(b_1^\pm,\dots,b_g^\pm) \in \QQ,\] which is Theorem \ref{thm:maintheorem}.
\noindent Now that we solved the recursion, one can write the result in terms of powerseries coefficients as follows.
Since $S^\pm=\frac{\cosh(t/2 )\ t/2}{\sinh(t/2)} =  \sum_n \frac{B_{2n}(1) \ t^{2n}}{(2n)!}$ and by \cite{FabPan}
\[ (-)^gb_g^\pm = \frac{-2^{g-1}(-)^g b_g}{2^{2g-1}-1} =  \frac{-2^{g-1}(-)^g }{2^{2g-1}-1}\frac{2^{2g-1}-1}{2^{2g-1}}\frac{|B_{2g}|}{(2g)!} = 2^{-g}\ {(-)^{g-1}}\frac{|B_{2g}|}{(2g)!},\] one sees $2^{-g}[t^{2g}]S^\pm = (-)^gb_g^\pm$. Then, since $a_g^\pm = M_g(b_1^\pm,\dots,b_g^\pm)$
is inverse to $(-)^gb_g^\pm = \frac{1}{(2g)!}[t^{2g}]\mathcal{F}^\pm(t)^{2g}$, we obtain the first powerseries expression.
To encode $a_g^\pm = M_g(b_1^\pm,\dots,b_g^\pm)$ directly, note that it is equivalent to \begin{equation}[t^{2g}]\mathcal{F}^\pm(t) = \frac{1}{1-2g}[t^{2g}]e^{(1-2g)\bigg(\sum_{h\geq 1} (2h-1)!(-)^hb_h^\pm t^{2h}\bigg)}.\end{equation}\label{eq:closedpowerseries}
One can rewrite the exponentiated formal series by $\sum_{h\geq 1} (2h-1)! \ 2^{-h}\ \frac{B_{2h}}{(2h)!}\ t^{2h} = \sum_{h\geq 1} \ \frac{B_{2h}}{(2h)}\ \left(\frac{t}{\sqrt2}\right)^{2h}.$
\end{proof}

\section{Spin virtual volumes} \label{sec: areaSV}
In this section, we will adapt the methods used to compute spin volumes in order to obtain recursion formulas for $d_g^{\pm}$ (and hence for area Siegel-Veech constants) and for the generalization $d_{g,p}^{\pm}$ of these integrals. Our goal will be to provide analogues of Proposition \ref{prop:boundary-contribution}, as well as use the Segre-Chiodo correspondence to compute the principal part contribution. 

\subsection{Spin Siegel-Veech constants} \label{subsection:dg_computation}
\subsubsection{The recursion statement}
In order to compute $d_g^\pm$, we will integrate the formula in Proposition \ref{prop:2g-1-recursion} against $\xi^{2g-2}\delta_0$. This computation shares many similarities with the computation of $a_g^{\pm}$. The principal part is obtained from the same Chiodo - Segre correspondence in Theorem \ref{thm:Segre-Chiodo}. Moreover, because of the vanishings  $\xi^{2g} = \xi^{2g-1}\delta_0 = 0$, the boundary contributions will come from the same graphs.

We have an analogue of Proposition \ref{prop:boundary-contribution}, accounting for the boundary contribution. In the same notation as \textit{loc. cit.} we have:

\begin{proposition}\label{prop:boundary-cont-d_g} Let $\oGamma$ be an odd bi-colored graph, and let $k\in\{1,\dots,2g-3\}$ be an odd number.
Then if $\oGamma $ satisfies $(*)$, we have 
    $$\int_{\oM_{g,1}}p_{*}\left(\psi^{N_{k,g}-1}_{1}\xi^{2g-2}\delta_0\pi_{N_{k,g}*}\zeta_{\oGamma*}[\PP\oSQ(\oGamma)]^{\pm}\right) =   N_{k,g}!h\big(k-1,(-2g_v)_{v\in V_0}\big)\sum_{v\in V_0}\left( d_v^{\pm}\cdot \prod_{w\in V_{0}, w\neq v}a_{g_w}^{\pm}\right).$$

Otherwise, if $\oGamma$ does not satisfy $(*)$, we obtain 
$$\int_{\oM_{g,1}}p_{*}\left(\psi^{N_{k,g}-1}_{1}\xi^{2g-2}\delta_0\pi_{N_{k,g}*}\zeta_{\oGamma*}[\PP\oSQ(\oGamma)]^{\pm}\right) =  0.$$
\end{proposition}

\noindent Using this proposition, we can evaluate the boundary contribution, and we have: 

\begin{corollary}\label{cor:boundarytermd_g} In the notation of Lemma \ref{lem:agpm-recursion} and  Proposition \ref{prop:2g-1-recursion}, we have 
\begin{equation*}
    p_{*}\left(\xi^{2g-2}\delta_0\sum_{\substack{k=1, \\ k\ {\rm odd}}}^{2g-3}\frac{1}{N_{k,g}!}\frac{\Delta_{k}}{2}\prod_{\substack{\ell>k, \\ \ell\,{\rm odd}}}^{2g-3}\left(\frac{\xi+\ell\psi_{1}}{2}\right)\right)=\sum_{g_1 > 0} \sum_{\underline{h}\vdash g-g_1} \frac{1}{|Aut(\underline{h})|} \cdot  Hur(\underline{g}) \cdot (2g_1-1)d_{g_1}^\pm \prod_{i\geq 2} (2g_i-1)a_{g_i}^\pm.
\end{equation*}
\end{corollary}
Here we encounter the same function $Hur$ defined just before Lemma \ref{lem:agpm-recursion}.
Notice that, if we denote $\underline{g} = (g_1,\underline{h})$, we have 
\[\frac{|Aut(\underline{g})|}{|Aut(\underline{h})|} = \ \textrm{number of entries in} \ \underline{g} \ \textrm{equal to} \ g_1.\]
Our choice of summation in Corollary \ref{cor:boundarytermd_g} is just a way of keeping track of equal terms in Proposition \ref{prop:boundary-cont-d_g}.
\subsubsection{Contribution from principal part}

Using again the vanishing $\xi^{2g-1}\delta_0 = 0$, we see that the  contribution from the principal part is equal to 
\[ \frac{(2g-3)!!}{2^{g-1}}[\PP\oSQ_{g}(1)] \psi_1^{g-1}\xi^{2g-2}\delta_0 = \frac{(2g-3)!!}{2^{g-1}}\int_{\overline{\mathcal{M}}_{g,1}}
s_{g,2g-2}^{\pm} \, \delta_0 \, \psi_1^{g-1}.   \]

\begin{lemma} We have the equality
    \[
\int_{\overline{\mathcal{M}}_{g,1}}
s_{g,2g-2}^{\pm} \, \delta_0 \, \psi_1^{g-1}
= \frac{1}{2}\left(\beta_{g-1}^{\pm} + \frac{(-1)^g}{(g-1)!}(b_1)^{g-1}\right). 
\]
\end{lemma}

\begin{proof}
  Using again the formula in Theorem \ref{thm:Segre-Chiodo} we obtain:
\[
s_{g,2g-2}^{\pm}
=
\sum_{\Gamma \in \mathrm{Tree}_{g,1}}
\frac{(-1)^{|E(\Gamma)|}}{|\mathrm{Aut}(\Gamma)|}
\cdot \zeta_{\Gamma*}
\left(
\sum_{v \in V(\Gamma)}
\left[
\Omega^{\pm, 1^{n(v)}}_{g(v),\,n(v)} - 2^{2g(v)-1}
\right]_{2g(v)-2}
\otimes \bigotimes_{w\neq v} \left[
\Omega^{\pm, 1^{n(w)}}_{g(w),\,n(w)}
\right]_{2g(w)-1}
\right).
\]

Looking at Lemma \ref{eq:spinchiodo=2Lambda} in degree \(2g-2\), we obtain: \\

$\Big[
\Omega^{\pm, 1^n}_{g,\,n}
- 2^{2g-1}
\Big]_{2g-2} = 
\begin{cases}
    2^{2g-2}
\cdot (-1)^{g-1}
\cdot
\left[
\lambda_{g-1}^{2}
- \frac{5}{2} \, \lambda_g \lambda_{g-2}
\right],
\quad \text{if } g\geq 2, \\ 
-1 \quad \text{if } g = 1.
\end{cases}$

Using that \(\delta_0 \lambda_g = 0\) for any \(g \geq 0\), and
\(\delta_0 \lambda_{g-1}^2 = 0\) for all \(g \neq 1\),
we obtain the vanishing of many terms appearing in the summation above.

In particular,
\begin{align*}
    \int_{\overline{\mathcal{M}}_{g,1}}
s^{\pm}_{g,2g-2} \, \delta_0 \, \psi_1^{g-1}
&= \sum_{\Gamma \in \mathrm{Tree}_{g,1}}
\frac{(-1)^{|E(\Gamma)|}}{|\mathrm{Aut}(\Gamma)|}
\Bigg[
\sum_{v \neq v_1} \int_{\overline{\mathcal{M}}_{g(v_1),\,n(v_1)}}
\psi^{g-1}_1 \,
\big[ \Omega^{\pm, 1^{n(v_1)}}_{g(v_1), n(v_1)} \big]_{2g(v_1)-1} \times 
 \\ 
& 
\times \int_{\overline{\mathcal{M}}_{g(v),\,n(v)}}
\delta_0\cdot \big[ \Omega^{\pm, 1^{n(v)}}_{g(v), n(v)} \big]_{2g(v)-2}   \times \prod_{\substack{w \neq v_1 \\ w \neq v}}
\int_{\overline{\mathcal{M}}_{g(w),\,n(w)}}
\big[ \Omega^{\pm, 1^{n(w)}}_{g(w), n(w)} \big]_{2g(w)-1}  +  \\
& + \int_{\overline{\mathcal{M}}_{g(v_1),\,n(v_1)}}
\delta_0 \, \psi_1^{g-1} \big[ \Omega^{\pm, 1^{n(v_1)}}_{g(v_1), n(v_1)} \big]_{2g(v_1)-2} \times  \prod_{w \neq v_1}
\int_{\overline{\mathcal{M}}_{g(w),\,n(w)}}
\big[ \Omega^{\pm, 1^{n(w)}}_{g(w), n(w)} \big]_{2g(w)-1} \Bigg]
\end{align*}

By degree considerations,
\[
\int_{\overline{\mathcal{M}}_{g(v),\,n(v)}}
\big[ \Omega^{\pm, 1^{n(v)}}_{g(v), n(v)} \big]_{2g(v)-1}
\neq 0
\quad \text{only if } g(v)=1,\; n(v)=1
\]
and also 
\[
\int_{\overline{\mathcal{M}}_{g(v),\,n(v)}}
\delta_0 \cdot
\big[ \Omega^{\pm, 1^{n(v)}}_{g(v), n(v)} \big]_{2g(v)-2}
\neq 0
\quad \text{only if } g(v)=1,\; n(v)=1.
\]

As a consequence, the only contributing graphs are those satisfying $g(v)=1, n(v)=1$ for any $v\neq v_1$ (i.e. different from the vertex containing the marking). Moreover, we have

\[
\int_{\overline{\mathcal{M}}_{g(v),\,n(v)}}
\psi_1^{g-1} \, \delta_0 \cdot
\big[ \Omega^{\pm, 1^{n(v)}}_{g(v),\,n(v)} \big]_{2g(v)-2}
\neq 0
\quad \text{if and only if} \ g(v)=1 \ \textrm{and} \ n(v) = g.
\]

Therefore, the equality above becomes
\begin{align*}
    \int_{\overline{\mathcal{M}}_{g,1}}
\psi_1^{g-2} \, \delta_0 \, s^\pm_{g,2g-2}
&= \sum_{m = 1}^{g-1}
\frac{(-1)^m}{m!}
\cdot m \cdot
\int_{\overline{\mathcal{M}}_{g-m,\,1}}
\psi_1^{g-m-1} \,
\lambda_{g-m} \lambda_{g-m-1}
\cdot (-1)^{g-m-1}
\cdot 2^{2(g-m-1)} \times  \\
& \times \left(
\int_{\overline{\mathcal{M}}_{1,1}} \lambda_1
\right)^{m-1} \times (-1) \cdot \int_{\overline{\mathcal{M}}_{1,1}} \delta_0 \\
&+ \frac{(-1)^{g-1}}{(g-1)!}
\cdot
\int_{\overline{\mathcal{M}}_{1,\,g}}
\psi_1^{g-1} \, \delta_0 \cdot (-1)
\times
\left(
\int_{\overline{\mathcal{M}}_{1,1}} \lambda_1
\right)^{g-1}.
\end{align*}

Using the formula for $\beta_g^\pm$ in Lemma \ref{lem:bg-pm-computation} as well as $\int_{\overline{\mathcal{M}}_{1,1}} \delta_0 = \frac{1}{2}$ and $\int_{\overline{\mathcal{M}}_{1,1}} \lambda_1 =  b_1$, the conclusion follows.
\end{proof}

We have computed both the principal and boundary contribution for the formula in Proposition \ref{prop:2g-1-recursion} integrated against $\xi^{2g-2}\delta_0$. Putting everything together in a formula, we obtain an analogue of Lemma \ref{lem:agpm-recursion}. 

\begin{lemma} \label{lem:dgpm-recursion}
The values $d_g^\pm$ satisfy the recursion formula 
    \begin{align*}
    	d_g^\pm &= \frac{(2g-3)!!}{2^{g}}\left(\beta_{g-1}^\pm + \frac{(-1)^g}{(g-1)!}\cdot \frac{1}{24^{g-1}}\right) \\
    	&-\sum_{g_1 > 0} \sum_{\underline{h}\vdash g-g_1} \frac{1}{|Aut(\underline{h})|} \cdot  Hur(\underline{g}) \cdot (2g_1-1)d_{g_1}^\pm \prod_{i\geq 2} (2g_i-1)a_{g_i}^\pm.
    \end{align*}
\end{lemma}

\subsection{Computing spin flower integrals}

In fact, the methods of Subsection \ref{subsection:dg_computation} can be applied in more generality. To make this precise, we first need some definitions. 

\begin{definition}
    Let $R = \{ r_1 + r_2 = 0, r_3 + r_4 = 0, \ldots, r_{2p-1}+r_{2p} = 0\}$ be the space of residue conditions for some $p\geq 0$ and consider the sublocus of the stratum $\mathbb{P}\overline{\mathcal{H}}_{g-p, 1+ 2p}(2g-1, \underbrace{0, \dots, 0}_{\text{2p times}})$ defined by $R$: 
    \[ \mathbb{P}\overline{\mathcal{H}}_{g-p, 1+ 2p}^R(2g-1, \underbrace{0, \dots, 0}_{\text{2p times}}) \coloneqq \{[C,x_0,x_1,\ldots, x_{2p}, \varphi]\ \mid \ \textrm{Res}_{x_{2i-1}}(\varphi) + \textrm{Res}_{x_{2i}}(\varphi) = 0 \ \textrm{for all} \ 1\leq i \leq p\} . \]
    We consider the map
    $$\zeta\colon \mathbb{P}\overline{\mathcal{H}}_{g-p, 1+ 2p}^R(2g-1, \underbrace{0, \dots, 0}_{\text{2p times}}) \rightarrow \mathbb{P}\oH_{g, 1}(2g-1)$$
    obtained by glueing together points $x_{2i-1}$ and $x_{2i}$ for all $1 \leq i \leq p$. We define $[D_p]$ to be the push-forward 
    \[ [D_p] \coloneqq  \zeta_*[\mathbb{P}\overline{\mathcal{H}}_{g-p, 1+ 2p}^R(2g-1, \underbrace{0, \dots, 0}_{\text{2p times}})] \]
    Furthermore, we define the integral 
    \[ d_{g,p}^\pm \coloneqq \int_{\mathbb{P}\overline{\mathcal{H}}_g(2g-1)^\pm} \xi^{2g-1-p}[D_p]. \]
\end{definition}

The same methods already outlined twice in this paper will work for this case, and obtain a recursion formula for $d_{g,p}^\pm $. First, we extend the definition of $[D_p]$ to $\PP\oH_{g,1}$ as 
 \[ [D_p] \coloneqq  \zeta_*[\mathbb{P}\overline{\mathcal{H}}_{g-p, 1+ 2p}^R(\underbrace{2,2,\ldots,2}_{2g-2 \ \text{times}}, \underbrace{0, \dots, 0}_{\text{2p times}})] \]
 with  $R = \{ r_1 + r_2 = 0, r_3 + r_4 = 0, \ldots, r_{2p-1}+r_{2p} = 0\}$ as before. The restriction of this class to $\PP\overline{\mathcal{H}}_g(2g-1)$ coincides with the previously defined class, and we therefore denote both by $[D_p]$.

We have the following vanishing result, cf. \cite[Corollary 2.7]{Sauvaget-virtual}:
\begin{lemma} \label{lemma:flower-vanishing}
    For any $p \geq 0$, we have the vanishing 
    \[\xi^{2g-p} \cdot [D_p] = 0 \in A^\bullet(\PP\oH_{g,n}). \]
\end{lemma}

In fact, another way to view this vanishing, is to look at the class $\xi^{2g-p}$ in the Chow ring of $$\mathbb{P}\overline{\mathcal{H}}_{g-p, 2g-2+ 2p}(\underbrace{2,2,\ldots,2}_{2g-2 \ \textrm{times}}, \underbrace{0, \dots, 0}_{2p \ \text{times}}).$$ Each multiplication with a $\xi$ on this space can be interpreted as adding a residue vanishing condition, so we get a linear equivalence with boundary loci where the number of poles on level $0$ decreases with at least two, cf. \cite[Proposition 7.6]{Sau2}. After multiplying $p$ times with $\xi$, we get that $\xi^p$ is linearly equivalent to boundary loci for which the level $0$ component is holomorphic. Since we land in the holomorphic case in the boundary, we know that $\xi^{2(g-p)}$ always vanish on these loci.

In fact, this result can be generalized to other loci in the boundary of $\PP\oH_{g,n}$. For $\overline{\Gamma}$ a bi-colored graph in ${\rm Bic}_{g,1}(2,2,\ldots,2)$ (i.e. compatible with the partition $a= (\underbrace{2,2,\ldots,2}_{2g-2 \ \text{times}})$) we look at the corresponding morphism

\begin{equation*}
    \zeta_{\oGamma}\colon \oH(\oGamma)\rightarrow \oH_{g,1}.
\end{equation*}
This admits a projectivisation  
\begin{equation*}
    \zeta_{\oGamma}\colon \mathbb{P}\oH(\oGamma)\to\mathbb{P}\oH_{g,1} ,
\end{equation*}
where $\mathbb{C}^*$ acts on the left-hand-side by scaling simultaneously all differentials corresponding to vertices in level $0$. Then, if we denote by $[D(\overline{\Gamma})] \coloneqq \zeta_{\oGamma*}[\mathbb{P}\oH(\oGamma)] $, we have the following vanishing property.

\begin{lemma}  \label{lemma:nonflower-vanishing} 
    For any $p \geq 1$, let $\overline{\Gamma}$ be a bi-colored graph in ${\rm Bic}_{g,1}(2,2,\ldots,2)$ satisfying $h^1(\overline{\Gamma}) \geq p$, and assume that $D(\overline{\Gamma})$ is not contained in $D_p$. Then we have the vanishing
\[ \xi^{2g-1-p}[D(\overline{\Gamma})] = 0.\]
\end{lemma}
\begin{proof}
     Let $\overline{\Gamma}[0]$ be the level $0$ part of $\overline{\Gamma}$, with $V[0]$ its set of vertices and $E[0]$ its set of edges. We will denote by $s$ the number of connected components of this graph.
       
       We know that each horizontal edge imposes a residue condition (the sum of residues of the corresponding two simple poles is $0$) and each vertex imposes one more residue condition (the global residue condition). After imposing all conditions, we have exactly $h^1(\overline{\Gamma}[0])$ independent residues. 
       
       Indeed, we have $2|E[0]|$ (simple) poles, and the residue conditions at edges already give us $|E[0]|$ independent constraints. Furthermore, we have $|V[0]|$ global residue conditions, but for each connected component of $\overline{\Gamma}[0]$, one such condition is determined by the edge conditions and the other global residue conditions of that component. In total, we count $|E[0]| + |V[0]| - s$ independent linear conditions on the residues. This means that the number of independent residues is 
       $$2|E[0]| -(|E[0]| + |V[0]| - s)  = |E[0]| - |V[0]| + s =  h^1(\overline{\Gamma}[0]).$$  
       
       In particular, reasoning as in \cite[Section 2]{Sauvaget-virtual}, the maximal non-vanishing power of $\xi$ is at most  
       \[ \sum_{v\in V[0]} (2g_v-1)  + h^1(\overline{\Gamma}[0]) + (s-1). \]
       But we have  
       \[ \sum_{v\in V[0]} g_v \leq g - h^1(\overline{\Gamma})\]
       and hence we get 
       \[ \sum_{v\in V[0]} (2g_v-1)  + h^1(\overline{\Gamma}[0]) + (s-1) \leq 2g - 2h^1(\overline{\Gamma}) - |V[0]| + h^1(\overline{\Gamma}[0]) + (s-1). \]
       Using $h^1(\overline{\Gamma}) \geq p$, $h^1(\overline{\Gamma}) \geq h^1(\overline{\Gamma}[0])$ and $|V[0]| \geq s$, we obtain the required inequality 
       \[ \sum_{v\in V[0]} (2g_v-1)  + h^1(\overline{\Gamma}[0]) + (s-1) \leq 2g - p-1. \]
       In particular we have 
\[ \xi^{2g-1-p}[D(\overline{\Gamma})] = 0\]
unless $h^1(\overline{\Gamma}) = h^1(\overline{\Gamma}[0]) = p$, $|V[0]| =  s$ and $\sum_{v\in V[0]} g_v = g - h^1(\overline{\Gamma})$. We will see that all graphs respecting these numerical conditions produce loci contained in $D_p$. 

Indeed, when we collapse all lower levels, all nodes on top level get identified into a single node possesing $p$ loops, and hence $D(\overline{\Gamma}) \subseteq D_p$.  
\end{proof}

We consider the maps
\[ \zeta_p\colon \oM_{g-p,n+2p} \rightarrow \oM_{g,n} \]
obtained by gluing together the points $x_{2i-1}$ and $x_{2i}$ for all $n\leq i \leq p$ (where the numbering starts from $0$). We define the push-forward 
\[ [F_p] \coloneqq \zeta_{p,*}[\oM_{g-p,n+2p}],\]
as the boundary locus with generic dual graph 
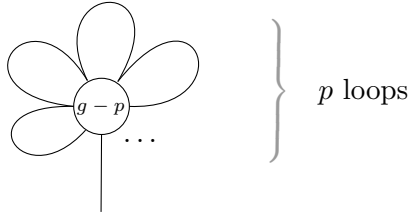
\[\begin{tikzpicture}[x=0.75pt,y=0.75pt,yscale=-1,xscale=1]

\draw   (275.83,140.7) .. controls (275.83,132.9) and (282.08,126.59) .. (289.8,126.59) .. controls (297.51,126.59) and (303.76,132.9) .. (303.76,140.7) .. controls (303.76,148.49) and (297.51,154.8) .. (289.8,154.8) .. controls (282.08,154.8) and (275.83,148.49) .. (275.83,140.7) -- cycle ;
\draw    (289.8,154.8) -- (289.56,193.65) ;
\draw    (297.97,128.22) .. controls (339.98,75.34) and (360.45,142.8) .. (303.76,140.7) ;
\draw    (275.83,140.7) .. controls (249.95,140.16) and (242.47,149.92) .. (244.98,157.31) .. controls (247.94,166.04) and (264.87,171.46) .. (281.81,152.59) ;
\draw    (282.7,127.71) .. controls (253.8,79.69) and (212.86,134.09) .. (275.83,140.7) ;
\draw    (297.97,128.22) .. controls (330.28,73.16) and (258.11,77.51) .. (282.7,127.71) ;

\draw (299.38,155.63) node [anchor=north west][inner sep=0.75pt]   [align=left] {$\displaystyle \dotsc $};
\draw (275.96,136.01) node [anchor=north west][inner sep=0.75pt]  [font=\tiny] [align=left] {$\displaystyle g-p$};
\draw (385.05,172.55) node [anchor=north west][inner sep=0.75pt]  [font=\LARGE,color={rgb, 255:red, 155; green, 155; blue, 155 }  ,opacity=1 ,rotate=-180.86] [align=left] {$\displaystyle \begin{cases}
 & \\
 & 
\end{cases}$};
\draw (397.35,125.75) node [anchor=north west][inner sep=0.75pt]   [align=left] {$\displaystyle p$ loops};

\end{tikzpicture}\]

 We can then compute the principal part contribution to $d^\pm_{g,p}$.  We have: 
\begin{lemma} \label{lemma: D_p = F_p - correspondence}Let $p\colon \PP\oH_{g,1} \rightarrow \oM_{g,1}$ be the forgetful map. Then we have the equality 
\begin{align*}
     \frac{1}{2^{g-1}}[\PP\oSQ_{g}(1)]^{\pm}\prod_{\substack{k=1, \\ k\ {\rm odd}}}^{2g-3}(\xi+k\psi_{1}) \cdot \xi^{2g-1-p} \cdot [D_p] &= \frac{(2g-3)!!}{2^{g-1}} \cdot [\PP\oSQ_{g}(1)]^{\pm} \psi_1^{g-1}  \cdot \xi^{2g-1-p} \cdot [D_p]\\
     &= \frac{(2g-3)!!}{2^{g-1}} \cdot [\PP\oSQ_{g}(1)]^{\pm} \psi_1^{g-1}  \cdot \xi^{2g-1-p} \cdot p^*[F_p] \\
     & = \frac{(2g-3)!!}{2^{g-1}} \cdot s_{g,2g-1-p}^\pm \cdot \psi_1^{g-1} \cdot [F_p].
\end{align*}
\end{lemma}

\begin{proof} The first equality is a consequence of Lemma \ref{lemma:flower-vanishing}. For the second equality, we use that $p^*[F_p]$ consists of classes supported on $p^{-1}(F_p)$ and the only one that does not vanish when multiplied with $\xi^{2g-1-p}$ is $[D_p]$, cf. Lemma \ref{lemma:nonflower-vanishing}. The last equality is simply the push-pull formula for the morphism $p$. 
\end{proof}

In order to evaluate the product in the previous lemma, we will use the spin Segre-Chiodo correspondence in Theorem \ref{thm:Segre-Chiodo}. A lot of the graphs in this formula of the Segre class will not contribute to the product, and these vanishings will be made explicit in the following two lemmata. 
\begin{lemma} \label{lemma:vanishing-no-psi}
	Let $n, p \geq 1$ and $0 \leq d_1, d_2 \leq g$. Then we have the vanishing 
	\[ \int_{\overline{\mathcal{M}}_{g,n}} [F_p] \lambda_{d_1}\lambda_{d_2} = 0, \]
	except for the case $g = n = p = 1, d_1 = d_2= 0$, in which case the integral is $1$.
\end{lemma} 
\begin{proof}
	The class $[F_p]$ is defined as the (scaled) push-forward image of a glueing morphism from the moduli space $\overline{\mathcal{M}}_{g-p, n + 2p}$. In particular, we have the vanishings 
	\begin{itemize}
		\item $[F_p] \cdot \lambda_d = 0 \ \forall d \geq g-p+1$ and 
		\item $[F_p] \cdot \lambda_{g-p}^2 = 0$, except when $g=p$. 
	\end{itemize} 
Assuming that the integral is not zero, and $g \neq p$ we obtain the degree contradiction 
\[ 3g-2 \leq 3g-3 +n = p +d_1 +d_2 \leq p + 2g -2p-1 \leq 2g-2.\]
We are left with the case $ g= p$ when the only possibility for the integral to be non-zero is $d_1 = d_2 = 0$. By degree considerations, the equality $3g-3+n = g$ will imply $g = n =p =1$, and it is well-known that $\int_{\overline{\mathcal{M}}_{1,1}}[F_1] = \int_{\overline{\mathcal{M}}_{1,1}}2\delta_0 = 1$. 
\end{proof}

On the component containing the marking, we will have to account for an extra contribution of $\psi_1^{g-1}$. We have the following vanishing lemma: 
\begin{lemma} \label{lemma:vanishing-with-psi}Let $p \geq 1$, $ 1\leq g_1 \leq g$ and $0 \leq d_1, d_2 \leq g_1$. Then we have the vanishing 
	\[ \int_{\overline{\mathcal{M}}_{g_1,g-g_1+1}} \psi_1^{g-1}[F_p]\lambda_{d_1} \lambda_{d_2} = 0, \]
	except for the case $p = g_1 = 1$ and $d_1 = d_2 = 0$, in which case the integral is $1$. 
\end{lemma}

\begin{proof}
	As in the previous lemma, we assume the non-vanishing of the integral. In particular, we have $[F_p]\lambda_{d_1} \lambda_{d_2} \neq 0$, which by the same argument implies 
	\begin{itemize}
		\item $d_1 + d_2 \leq 2g_1 - 2p -1$ if $ g_1 \neq p$ or 
		\item $d_1 = d_2 = 0$, if $ g_1 = p$. 
	\end{itemize}
When $g_1 \neq p$, we obtain again a contradiction in terms of degrees: 
\[ 3g_1 - 3 + 1 + g-g_1 = g-1+p +d_1+d_2 \leq g-1+p + 2g_1 -2p-1 \]
which after simplification becomes 
\[ 2g_1 + g -2 \leq 2g_1 + g -p-2.\]
We are left to treat the case $g_1 = p$ and $d_1 = d_2 = 0$. In this case, we have 
\[ \int_{\overline{\mathcal{M}}_{g_1,g-g_1+1}} \psi_1^{g-1}[F_p] \neq 0 \]
if and only if $g_1 = 1$. 
\end{proof}

These two Lemmata will imply the vanishing of many terms appearing in the spin Segre-Chiodo correspondence, and will be sufficient to compute the principal part contribution to $d^\pm_{g,p}$.
\begin{proposition} \label{prop: F_p-Chiodo}
    The integral appearing in the principal part contribution satisfies
	\[ [\mathbb{P}\overline{\mathcal{SQ}}_g(1)] \psi_1^{g-1} \xi^{2g-1-p} p^* [F_p] = \beta_{g-p}^{\pm} + (-1)^{g-1-p}\frac{p}{(g-p)!}\cdot b_1^{g-p}. \]
\end{proposition}

\begin{proof}
	Using the projection formula, we have 
	\begin{align*} [\mathbb{P}\overline{\mathcal{SQ}}_g(1)]\cdot \psi_1^{g-1} \xi^{2g-1-p} \cdot p^* [F_p] &=   \int_{\overline{\mathcal{M}}_{g,1}} \psi_1^{g-1} s^{\pm}_{g, 2g-1-p} [F_p]. \end{align*}
    Using the spin Segre-Chiodo correspondence of Theorem \ref{thm:Segre-Chiodo}, we have an explicit formula for $s^{\pm}_{g, 2g-1-p}$ in terms of boundary loci and $\lambda$-classes, thus we consider the integrals \[\int_{\overline{\mathcal{M}}_{\Gamma}} \zeta_{\Gamma}^*[F_p]\cdot\left(\bigotimes_{v\in V(\Gamma)}\Omega_{g(v),n(v)}^{\pm,1^{n(v)}}(2)-2^{2g(v)-1}\right) \cdot  \psi_1^{g-1}\otimes 1.\] Since flowers and trees share no edges, the push-pull formula of gluing morphisms (cf. \cite[Appendix A]{GraPan} and \cite[Proposition 2.14]{Schvan}) brings no excess contributions and 
    \[\zeta_{\Gamma}^*[F_p] = \sum_{(p_v)_{v}} \bigotimes_{v\in V(\Gamma)}[F_{p_v}],\] where the sum ranges over the distributions $(p_v)_{v\in V(\Gamma)}$ of the $p$ ordered loops over the vertices $v$ of the tree $\Gamma$.\\
	\noindent From the previous Lemmata, we know how the graphs $\Gamma$ with non-vanishing integrals look like: 
	\begin{itemize}
		\item The vertices not containing the marking are all of genus $1$ (cf. Lemma \ref{lemma:vanishing-no-psi}). Such a vertex has a unique adjacent edge, connecting it with the vertex containing the marking.  
	\end{itemize}
\noindent \textbf{Case 1:} When the genus of the component containing the marking is strictly greater than $1$, we get that multiplication of the classes glued along a tree $\zeta_{\Gamma*}\left(\bigotimes_{v\in V(\Gamma)}\Omega_{g(v),n(v)}^{\pm,1^{n(v)}}(2)-2^{2g(v)-1}\right)$ with $[F_p]$ amounts to a multiplication of $[\Omega_{1,1}^{\pm,1}(2)]_0-2 = -1$ with $[F_1]=2\delta_0$ on $p$ of the non-marked components. Id est, for the graph

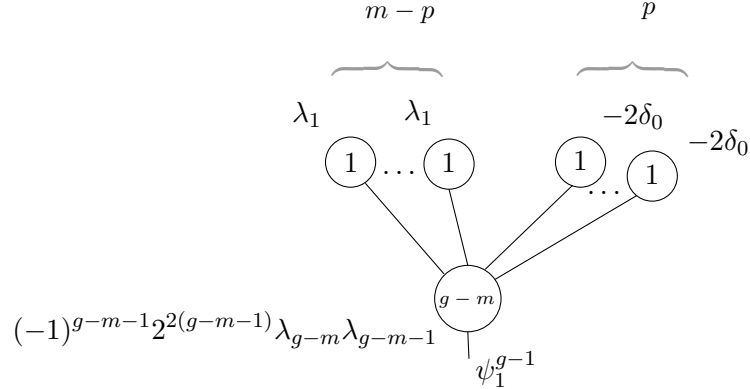
\[\begin{tikzpicture}[x=0.75pt,y=0.75pt,yscale=-1,xscale=1]

\draw   (268,95.42) .. controls (268,88.47) and (273.63,82.83) .. (280.58,82.83) .. controls (287.53,82.83) and (293.17,88.47) .. (293.17,95.42) .. controls (293.17,102.37) and (287.53,108) .. (280.58,108) .. controls (273.63,108) and (268,102.37) .. (268,95.42) -- cycle ;
\draw    (339.43,180.55) -- (340,194.12) ;
\draw    (288,106) -- (327.5,151.55) ;
\draw    (330,109) -- (339.43,147.23) ;
\draw   (317.42,96.42) .. controls (317.42,89.47) and (323.05,83.83) .. (330,83.83) .. controls (336.95,83.83) and (342.58,89.47) .. (342.58,96.42) .. controls (342.58,103.37) and (336.95,109) .. (330,109) .. controls (323.05,109) and (317.42,103.37) .. (317.42,96.42) -- cycle ;
\draw   (322.78,163.89) .. controls (322.78,154.69) and (330.23,147.23) .. (339.43,147.23) .. controls (348.63,147.23) and (356.09,154.69) .. (356.09,163.89) .. controls (356.09,173.09) and (348.63,180.55) .. (339.43,180.55) .. controls (330.23,180.55) and (322.78,173.09) .. (322.78,163.89) -- cycle ;
\draw    (391.98,107.08) -- (348.08,149.28) ;
\draw   (383.42,95.42) .. controls (383.42,88.47) and (389.05,82.83) .. (396,82.83) .. controls (402.95,82.83) and (408.58,88.47) .. (408.58,95.42) .. controls (408.58,102.37) and (402.95,108) .. (396,108) .. controls (389.05,108) and (383.42,102.37) .. (383.42,95.42) -- cycle ;
\draw    (424.08,112.28) -- (353.08,154.28) ;
\draw   (419.42,102.42) .. controls (419.42,95.47) and (425.05,89.83) .. (432,89.83) .. controls (438.95,89.83) and (444.58,95.47) .. (444.58,102.42) .. controls (444.58,109.37) and (438.95,115) .. (432,115) .. controls (425.05,115) and (419.42,109.37) .. (419.42,102.42) -- cycle ;

\draw (276,90) node [anchor=north west][inner sep=0.75pt]   [align=left] {$\displaystyle 1$};
\draw (342,188.97) node [anchor=north west][inner sep=0.75pt]   [align=left] {$\displaystyle \psi _{1}^{g-1}$};
\draw (325,90) node [anchor=north west][inner sep=0.75pt]   [align=left] {$\displaystyle 1$};
\draw (323.57,159.57) node [anchor=north west][inner sep=0.75pt]  [font=\tiny] [align=left] {$\displaystyle {\textstyle g-m}$};
\draw (306,63) node [anchor=north west][inner sep=0.75pt]   [align=left] {$\displaystyle \lambda _{1}$};
\draw (250,64) node [anchor=north west][inner sep=0.75pt]   [align=left] {$\displaystyle \lambda _{1}$};
\draw (95,169) node [anchor=north west][inner sep=0.75pt]   [align=left] {$\displaystyle \ \ \ (-1)^{g-m-1}2^{2( g-m-1)} \lambda _{g-m}$$\displaystyle \lambda _{g-m-1}$};
\draw (295.17,98.42) node [anchor=north west][inner sep=0.75pt]   [align=left] {$\displaystyle \dotsc $};
\draw (328.42,41.7) node [anchor=north west][inner sep=0.75pt]  [font=\large,color={rgb, 255:red, 155; green, 155; blue, 155 }  ,opacity=1 ,rotate=-90.59] [align=left] {$\displaystyle \begin{cases}
 & \\
 & 
\end{cases}$};
\draw (286.7,14) node [anchor=north west][inner sep=0.75pt]  [font=\small] [align=left] {$\displaystyle m-p$};
\draw (451.42,41.7) node [anchor=north west][inner sep=0.75pt]  [font=\large,color={rgb, 255:red, 155; green, 155; blue, 155 }  ,opacity=1 ,rotate=-90.59] [align=left] {$\displaystyle \begin{cases}
 & \\
 & 
\end{cases}$};
\draw (425.7,14) node [anchor=north west][inner sep=0.75pt]  [font=\small] [align=left] {$\displaystyle p$};
\draw (391,89) node [anchor=north west][inner sep=0.75pt]   [align=left] {$\displaystyle 1$};
\draw (427,96) node [anchor=north west][inner sep=0.75pt]   [align=left] {$\displaystyle 1$};
\draw (405.18,65) node [anchor=north west][inner sep=0.75pt]   [align=left] {$\displaystyle -2\delta _{0}$};
\draw (448.18,78) node [anchor=north west][inner sep=0.75pt]   [align=left] {$\displaystyle -2\delta _{0}$};
\draw (398,108.42) node [anchor=north west][inner sep=0.75pt]   [align=left] {$\displaystyle \dotsc $};

\end{tikzpicture}
\]
with the marked component of genus $g-m$ and $m$ components of genus $1$, we get a contribution of 
\[ \frac{(-1)^m}{m!} \prod_{i=0}^{p-1} (m-i)\int_{\overline{\mathcal{M}}_{g-m,1}}\psi_1^{g-m-1}\lambda_{g-m}\lambda_{g-m-1} \cdot (-1)^{g-m-1}2^{2(g-m-1)} \cdot \left(\int_{\overline{\mathcal{M}}_{1,1}}\lambda_1\right)^{m-p}\cdot \left(-\int_{\overline{\mathcal{M}}_{1,1}}2\delta_0 \right)^p.\]
After simplification, this is 
\[ (-1)^{g-p-1}\frac{2^{2(g-m-1)}}{(m-p)!} \cdot \int_{\overline{\mathcal{M}}_{g-m,1}}\psi_1^{g-m-1}\lambda_{g-m}\lambda_{g-m-1} \cdot b_1^{m-p}. \]

\noindent \textbf{Case 2:} When the genus of the component containing the marking is $1$, we get two types of non-zero contributions. 
\begin{enumerate}
	\item As before, we have the case when multiplication with $[F_p]$ does not add a loop to the marked component.

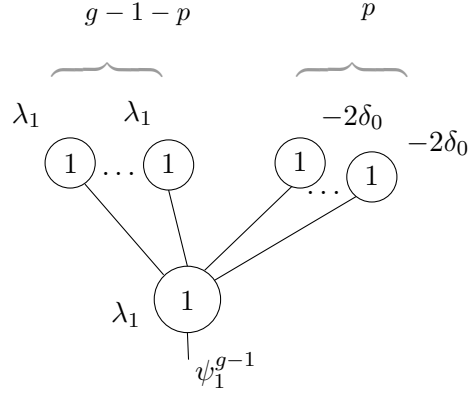
\[\begin{tikzpicture}[x=0.75pt,y=0.75pt,yscale=-1,xscale=1]

\draw   (268,95.42) .. controls (268,88.47) and (273.63,82.83) .. (280.58,82.83) .. controls (287.53,82.83) and (293.17,88.47) .. (293.17,95.42) .. controls (293.17,102.37) and (287.53,108) .. (280.58,108) .. controls (273.63,108) and (268,102.37) .. (268,95.42) -- cycle ;
\draw    (339.43,180.55) -- (340,194.12) ;
\draw    (288,106) -- (327.5,151.55) ;
\draw    (330,109) -- (339.43,147.23) ;
\draw   (317.42,96.42) .. controls (317.42,89.47) and (323.05,83.83) .. (330,83.83) .. controls (336.95,83.83) and (342.58,89.47) .. (342.58,96.42) .. controls (342.58,103.37) and (336.95,109) .. (330,109) .. controls (323.05,109) and (317.42,103.37) .. (317.42,96.42) -- cycle ;
\draw   (322.78,163.89) .. controls (322.78,154.69) and (330.23,147.23) .. (339.43,147.23) .. controls (348.63,147.23) and (356.09,154.69) .. (356.09,163.89) .. controls (356.09,173.09) and (348.63,180.55) .. (339.43,180.55) .. controls (330.23,180.55) and (322.78,173.09) .. (322.78,163.89) -- cycle ;
\draw    (391.98,107.08) -- (348.08,149.28) ;
\draw   (383.42,95.42) .. controls (383.42,88.47) and (389.05,82.83) .. (396,82.83) .. controls (402.95,82.83) and (408.58,88.47) .. (408.58,95.42) .. controls (408.58,102.37) and (402.95,108) .. (396,108) .. controls (389.05,108) and (383.42,102.37) .. (383.42,95.42) -- cycle ;
\draw    (424.08,112.28) -- (353.08,154.28) ;
\draw   (419.42,102.42) .. controls (419.42,95.47) and (425.05,89.83) .. (432,89.83) .. controls (438.95,89.83) and (444.58,95.47) .. (444.58,102.42) .. controls (444.58,109.37) and (438.95,115) .. (432,115) .. controls (425.05,115) and (419.42,109.37) .. (419.42,102.42) -- cycle ;

\draw (276,90) node [anchor=north west][inner sep=0.75pt]   [align=left] {$\displaystyle 1$};
\draw (325,90) node [anchor=north west][inner sep=0.75pt]   [align=left] {$\displaystyle 1$};
\draw (342,188.97) node [anchor=north west][inner sep=0.75pt]   [align=left] {$\displaystyle \psi _{1}^{g-1}$};
\draw (334.57,158.57) node [anchor=north west][inner sep=0.75pt]  [font=\normalsize] [align=left] {$\displaystyle 1$};
\draw (306,63) node [anchor=north west][inner sep=0.75pt]   [align=left] {$\displaystyle \lambda _{1}$};
\draw (250,64) node [anchor=north west][inner sep=0.75pt]   [align=left] {$\displaystyle \lambda _{1}$};
\draw (295,164) node [anchor=north west][inner sep=0.75pt]   [align=left] {$\displaystyle \ \lambda _{1}$$ $};
\draw (295.17,98.42) node [anchor=north west][inner sep=0.75pt]   [align=left] {$\displaystyle \dotsc $};
\draw (328.42,41.7) node [anchor=north west][inner sep=0.75pt]  [font=\large,color={rgb, 255:red, 155; green, 155; blue, 155 }  ,opacity=1 ,rotate=-90.59] [align=left] {$\displaystyle \begin{cases}
 & \\
 & 
\end{cases}$};
\draw (286.7,14) node [anchor=north west][inner sep=0.75pt]  [font=\small] [align=left] {$\displaystyle g-1-p$};
\draw (451.42,41.7) node [anchor=north west][inner sep=0.75pt]  [font=\large,color={rgb, 255:red, 155; green, 155; blue, 155 }  ,opacity=1 ,rotate=-90.59] [align=left] {$\displaystyle \begin{cases}
 & \\
 & 
\end{cases}$};
\draw (425.7,14) node [anchor=north west][inner sep=0.75pt]  [font=\small] [align=left] {$\displaystyle p$};
\draw (391,89) node [anchor=north west][inner sep=0.75pt]   [align=left] {$\displaystyle 1$};
\draw (427,96) node [anchor=north west][inner sep=0.75pt]   [align=left] {$\displaystyle 1$};
\draw (405.18,65) node [anchor=north west][inner sep=0.75pt]   [align=left] {$\displaystyle -2\delta _{0}$};
\draw (448.18,78) node [anchor=north west][inner sep=0.75pt]   [align=left] {$\displaystyle -2\delta _{0}$};
\draw (398,108.42) node [anchor=north west][inner sep=0.75pt]   [align=left] {$\displaystyle \dotsc $};

\end{tikzpicture}\]

 In this case we get a contribution of
	\[ \frac{(-1)^{g-1}}{(g-1-p)!} \int_{\overline{\mathcal{M}}_{1,1}}\lambda_{1} \cdot \left(\int_{\overline{\mathcal{M}}_{1,1}}\lambda_1\right)^{g-1-p}\cdot (-1)^p.\]
	\item Because of the vanishing exception in Lemma \ref{lemma:vanishing-with-psi}, we also have the case when multiplication with $[F_p]$ adds a loop to the marked component, and on another $p-1$ of the components without the marking. 

    
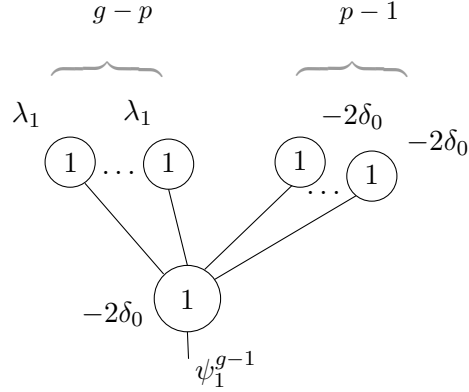
\[
\begin{tikzpicture}[x=0.75pt,y=0.75pt,yscale=-1,xscale=1]
\draw   (268,95.42) .. controls (268,88.47) and (273.63,82.83) .. (280.58,82.83) .. controls (287.53,82.83) and (293.17,88.47) .. (293.17,95.42) .. controls (293.17,102.37) and (287.53,108) .. (280.58,108) .. controls (273.63,108) and (268,102.37) .. (268,95.42) -- cycle ;
\draw    (339.43,180.55) -- (340,194.12) ;
\draw    (288,106) -- (327.5,151.55) ;
\draw    (330,109) -- (339.43,147.23) ;
\draw   (317.42,96.42) .. controls (317.42,89.47) and (323.05,83.83) .. (330,83.83) .. controls (336.95,83.83) and (342.58,89.47) .. (342.58,96.42) .. controls (342.58,103.37) and (336.95,109) .. (330,109) .. controls (323.05,109) and (317.42,103.37) .. (317.42,96.42) -- cycle ;
\draw   (322.78,163.89) .. controls (322.78,154.69) and (330.23,147.23) .. (339.43,147.23) .. controls (348.63,147.23) and (356.09,154.69) .. (356.09,163.89) .. controls (356.09,173.09) and (348.63,180.55) .. (339.43,180.55) .. controls (330.23,180.55) and (322.78,173.09) .. (322.78,163.89) -- cycle ;
\draw    (391.98,107.08) -- (348.08,149.28) ;
\draw   (383.42,95.42) .. controls (383.42,88.47) and (389.05,82.83) .. (396,82.83) .. controls (402.95,82.83) and (408.58,88.47) .. (408.58,95.42) .. controls (408.58,102.37) and (402.95,108) .. (396,108) .. controls (389.05,108) and (383.42,102.37) .. (383.42,95.42) -- cycle ;
\draw    (424.08,112.28) -- (353.08,154.28) ;
\draw   (419.42,102.42) .. controls (419.42,95.47) and (425.05,89.83) .. (432,89.83) .. controls (438.95,89.83) and (444.58,95.47) .. (444.58,102.42) .. controls (444.58,109.37) and (438.95,115) .. (432,115) .. controls (425.05,115) and (419.42,109.37) .. (419.42,102.42) -- cycle ;

\draw (276,90) node [anchor=north west][inner sep=0.75pt]   [align=left] {$\displaystyle 1$};
\draw (325,90) node [anchor=north west][inner sep=0.75pt]   [align=left] {$\displaystyle 1$};
\draw (334.57,158.57) node [anchor=north west][inner sep=0.75pt]  [font=\normalsize] [align=left] {$\displaystyle 1$};
\draw (306,63) node [anchor=north west][inner sep=0.75pt]   [align=left] {$\displaystyle \lambda _{1}$};
\draw (250,64) node [anchor=north west][inner sep=0.75pt]   [align=left] {$\displaystyle \lambda _{1}$};
\draw (285,164) node [anchor=north west][inner sep=0.75pt]   [align=left] {$\displaystyle -2\delta _{0} $$ $};
\draw (342,188.97) node [anchor=north west][inner sep=0.75pt]   [align=left] {$\displaystyle \psi _{1}^{g-1}$};
\draw (295.17,98.42) node [anchor=north west][inner sep=0.75pt]   [align=left] {$\displaystyle \dotsc $};
\draw (328.42,41.7) node [anchor=north west][inner sep=0.75pt]  [font=\large,color={rgb, 255:red, 155; green, 155; blue, 155 }  ,opacity=1 ,rotate=-90.59] [align=left] {$\displaystyle \begin{cases}
 & \\
 & 
\end{cases}$};
\draw (290.7,14) node [anchor=north west][inner sep=0.75pt]  [font=\small] [align=left] {$\displaystyle g-p$};
\draw (451.42,41.7) node [anchor=north west][inner sep=0.75pt]  [font=\large,color={rgb, 255:red, 155; green, 155; blue, 155 }  ,opacity=1 ,rotate=-90.59] [align=left] {$\displaystyle \begin{cases}
 & \\
 & 
\end{cases}$};
\draw (414.7,14) node [anchor=north west][inner sep=0.75pt]  [font=\small] [align=left] {$\displaystyle p-1$};
\draw (391,89) node [anchor=north west][inner sep=0.75pt]   [align=left] {$\displaystyle 1$};
\draw (427,96) node [anchor=north west][inner sep=0.75pt]   [align=left] {$\displaystyle 1$};
\draw (405.18,65) node [anchor=north west][inner sep=0.75pt]   [align=left] {$\displaystyle -2\delta _{0}$};
\draw (448.18,78) node [anchor=north west][inner sep=0.75pt]   [align=left] {$\displaystyle -2\delta _{0}$};
\draw (398,108.42) node [anchor=north west][inner sep=0.75pt]   [align=left] {$\displaystyle \dotsc $};

\end{tikzpicture}
\]

    In this case, we get a contribution of 
	\[ \frac{(-1)^{g-1}}{(g-1)!}\cdot p\cdot \prod_{i=0}^{p-2} (g-1-i)\cdot \left(-\int_{\overline{\mathcal{M}}_{1,g}} \psi_1^{g-1}2\delta_0\right) \cdot \left(-\int_{\overline{\mathcal{M}}_{1,1}}2\delta_0\right)^{p-1}\cdot \left(\int_{\overline{\mathcal{M}}_{1,1}} \lambda_1\right)^{g-p}.\]
	This is furthermore equal to 
	\[ (-1)^{g-1-p}\frac{p}{(g-p)!}\cdot b_1^{g-p}.\]
	Adding together all contributions, and using the formula for $\beta_{g-p}^{\pm}$ derived from Lemma \ref{lem:btilde-computation}, we obtain the conclusion.
\end{enumerate}
\end{proof}

In order to compute $\int_{\mathbb{P}\overline{\mathcal{H}}_g(2g-1)^\pm} \xi^{2g-1-p}[D_p]$ using Proposition \ref{prop:2g-1-recursion}, we are left to evaluate
\[ \left[\sum_{\substack{k=1, \\ k\ {\rm odd}}}^{2g-3}\frac{1}{N_{k,g}!}\frac{\Delta_{k}}{2}\prod_{\substack{\ell>k, \\ \ell\,{\rm odd}}}^{2g-3}\left(\frac{\xi+\ell\psi_{1}}{2}\right)\right] \cdot \xi^{2g-1-p}[D_p]. \]

Because of Lemma \ref{lemma:flower-vanishing}, we have the equality 
\[ \left[\sum_{\substack{k=1, \\ k\ {\rm odd}}}^{2g-3}\frac{1}{N_{k,g}!}\frac{\Delta_{k}}{2}\prod_{\substack{\ell>k, \\ \ell\,{\rm odd}}}^{2g-3}\left(\frac{\xi+\ell\psi_{1}}{2}\right)\right] \cdot \xi^{2g-1-p}[D_p] =  \left[\sum_{\substack{k=1, \\ k\ {\rm odd}}}^{2g-3}\frac{1}{N_{k,g}!}\Delta_{k} \frac{(2g-3)!!}{k!!\,\,2^{N_{k,g}}}\psi_1^{N_{k,g}-1} \right]\cdot \xi^{2g-1-p}[D_p] \] 

But this can be evaluated analogously to Proposition \ref{prop:boundary-contribution} and Proposition \ref{prop:boundary-cont-d_g}.

\begin{proposition}\label{prop:boundary-contribution-dg,k} Let $\oGamma$ be an odd bi-colored graph, and let $k\in\{1,\dots,2g-3\}$ be an odd number.
Then if $\oGamma $ satisfies $(*)$, see Proposition \ref{prop:boundary-contribution}, we have 
    \begin{align*}
\int_{\oM_{g,1}}
p_{*}\Big(
    \psi^{N_{k,g}-1}_{1}
    \xi^{2g-1-p}
    [D_p]
    \pi_{N_{k,g}*}
    \zeta_{\oGamma*}
    [\PP\oSQ(\oGamma)]^{\pm}
\Big)
=&
N_{k,g}!\,
h\big(k-1,(-2g_v)_{v\in V_0}\big) 
\\
& \times
\sum_{\substack{(p_v)_{v\in V_0} \\
\sum_{v\in V_0} p_v = p}}
\frac{p!}{\prod_{v\in V_0}p_v!}
\prod_{v\in V_0}
d_{g_v,p_v}^{\pm}.
\end{align*}

\noindent Otherwise, if $\oGamma$ does not satisfy $(*)$, we obtain 
$$\int_{\oM_{g,1}}p_{*}\left(\psi^{N_{k,g}-1}_{1}\xi^{2g-1-p}[D_p]\pi_{N_{k,g}*}\zeta_{\oGamma*}[\PP\oSQ(\oGamma)]^{\pm}\right) =  0.$$
\end{proposition}
\begin{proof} 
 The proof is similar to the proof of Proposition \ref{prop:boundary-contribution}. First we rewrite
\begin{align*}
&\int_{\PP\oH_{g,1}} [D_p]\ \psi^{N_{k,g}-1}_{1}\xi^{2g-1-p}\pi_{N_{k,g}*}\zeta_{\oGamma*}[\PP\oSQ(\oGamma)]^{\pm} =\int_{[\PP\oSQ(\oGamma)]^{\pm}}\zeta_{\oGamma}^* \pi_{N_{k,g}}^*\bigg( [D_p]\ \psi^{N_{k,g}-1}_{1}\xi^{2g-1-p}\bigg).\\
\end{align*}

\noindent 
Note that $\pi_{N_{k,g}}^*[D_p]$ still contains only horizontal loops. Assume otherwise for contradiction. Since the differentials are visible in the support of $[D_p]$, the non-contracted component in its inverse image is still visible. Since the differentials in the inverse image have still simple poles at any nodes that will contract onto the loops in $[D_p]$, we know that from the main vertex extend edges horizontally. These extend to a contracted network of genus $0$ vertices. In this network the vertices have at least one additional marking with order $2$, and at most two edges. There exists such a vertex with a horizontal edge. If it has no vertical edges this is obstructed by the degree sum being $-2$ on this vertex. If it does have a vertical edge, the node there supposedly has a zero. Again, this is obstructed by the degree sum. \\
Thence, by transversality of horizontal and vertical edges, we obtain (right-hand side interpreted with the additional markings):
\[\zeta_{\oGamma}^*\pi_{N_{k,g}}^*[D_p] = \sum_{(p_v)_{v}} \bigotimes_{v\in V(\Gamma)}[D_{p_v}].\] 

\noindent Thus, we can write the integral as
\begin{align*}
 \sum_{(p_v)_{v\in V{(\Gamma')}}}&\int_{\PP\oH_{\textbf{g}[0], \textbf{n}[0]}}  \bigotimes_{v\in V(\Gamma')_{0}}[D_{p_v}]\ \xi^{2g-1-p}\pi_{N_0*}[\PP\oSQ_{\textbf{g}[0]}(\textbf{a}[0],R[0])]^{\pm} \\
&\times \int_{\PP\oH_{\textbf{g}[-1], \textbf{n}[-1]}[\textbf{P}]} \bigotimes_{v\in V(\Gamma')_{-1}}[D_{p_v}]\ \psi^{N_{k,g}-1}_{1}\pi_{N_{-1}*}[\PP\oSQ_{\textbf{g}[-1]}(\textbf{a}[-1],R[-1])]^{\pm}.
\end{align*} Here again $\Gamma'$ is the stable graph obtained after forgetting the additional markings. The vector $\textbf{P}$ is a suitable twist to accommodate the poles.\\

\noindent The exact arguments in Lemma \ref{lem:non-trivial-terms} and Lemma \ref{lem:xi-int-gluing-vanishing} can be generalised using the vanishing result of Lemma~\ref{lemma:flower-vanishing}. In particular, we obtain that the former integral vanishes unless for all $v \in V_0$ we have $n(v)=1$, $\textbf{a}[0](v) = 2g(v)-1$, and $\sum_{v \in V_0}g(v) = g$, and equals \[\prod_{v\in V_0}\int_{[\PP\oSQ_{g(v)}(2g(v)-1)]^{\pm}} [D_{p_v}]\ \xi^{2g(v)-1-p_v},\]
which also shows that level $-1$ carries all extra markings $N_{k,g}$ and is of genus $0$ and thus, attains no loops. For each way to distribute $p_v$ of the loops to each vertex $v\in V_0$, we get a contribution of $\prod_{v\in V_0} d_{g_v,p_v}^{\pm}$ and there are $$\binom{p}{(p_v)_{v\in V_0}} = \frac{p!}{\prod_{v\in V_0}p_v!}$$ ways to distribute the $p$ loops. Moreover, the component on bottom level containing the marking must have dimension at least $N_{k,g}-1$ so that the integral containing $\psi_1^{N_{k,g}-1}$ does not vanish. This is possible only if the bottom level consists of a single vertex, with contribution as before. 
\end{proof}

Now that we have computed both the boundary contribution and the contribution coming from the principal part, we can state the formula for $d_{g,p}^\pm$, which is just a consequence of Proposition \ref{prop:2g-1-recursion}, Lemma \ref{lemma: D_p = F_p - correspondence}, Proposition \ref{prop: F_p-Chiodo} and Proposition \ref{prop:boundary-contribution-dg,k} put together. 

\begin{lemma}\label{lem:bouquetrecursion} We have 
\begin{align*}
    d_{g,p}^{\pm} =& \frac{(2g-3)!!}{2^{g-1}} \left( \beta_{g-p}^{\pm} + (-1)^{g-p-1} \frac{p}{(g-p)!}b_1^{g-p}\right) \\
    &- \sum_{\substack{2\leq m\leq g,\\\underline{g}\in{\rm Part}(g)_{m}}}\frac{Hur(\underline{g})}{m!} \left(\sum_{\substack{(p_i)_{1\leq i \leq m} \\ \sum_{1\leq i \leq m} p_i = p}} \frac{p!}{\prod_{v\in V_0}p_v!} \prod_{i = 1}^m(2g_{i}-1) d_{g_i, p_i}^{\pm}\right).
\end{align*}
\[     \]
\end{lemma}
\subsection{Derivation of closed formulas}\label{sec:derivation}
Define the derivation $\partial^\pm$ on the algebra $\QQ[x_1,x_2,\dots]$ via \[\partial^\pm: x_{1} \mapsto 1\] and $0$ on other generators. This defines a derivation $\partial^\pm$ on $\QQ[y_1,y_2,\dots]$ under the coordinate transformations $y_g=M_g(x_1,\dots,x_g)$ defined above \ref{prop:rec=main}.

\begin{lemma}\label{lem:virtualpm-derivation}
Let $ev_{a^\pm}$ denote the evaluation \[\QQ[y_1,y_2,\dots,y_g] \to \QQ,\]\[f \mapsto f(a_1^\pm,a_2^\pm,\dots,a_g^\pm)\] for a given $g>0$, as well as their inductive limit over all $g$, then
\[ev_{a^\pm}\left((\partial^\pm)^p\  y_g\right) =  d_{g,p}^\pm.\]
\end{lemma}
\begin{example}[genus 1]
By direct computation, for $p\geq 1$,
\[d_{1,p}^\pm  = -\int_{\overline{\mathcal{M}}_{1,1}} [F_p] = -\delta_{p,1}= -ev_{b^\pm}(\partial^\pm)^p x_1=  ev_{a^\pm}(\partial^\pm)^py_1.\]
\end{example}

\begin{proof}[Proof of \ref{lem:virtualpm-derivation}]
We first prove $ev_{a^\pm}(\partial^\pm y_g) = d_g^\pm$ by induction on $g$.
The initial case is $g=1$ above. 
Note that under \ref{prop:rec=main}, 

\begin{align*}
 \quad &(\partial^\pm)^p(R_g) =  \frac{2^{g-1}}{(2g-3)!!}\left((\partial^\pm)^p y_g + \sum_{\substack{m \geq 2,\\ \underline{g} \in {\rm Part}(g)_{m}}} \sum_{\substack{p_1+\dots+p_m = p \\ p_i \geq 0}}p! \prod_{i} (2g_i-1)\frac{(\partial^\pm)^{p_i}}{p_i!}y_{g_i} \, \cdot \frac{Hur(\underline{g})}{m!}\right)\\
\end{align*}
should be equal to
\begin{align*}
 (\partial^\pm)^p \Bigg(S_g\bigg(D(1)\cdot x_1,\dots, D(g)\cdot x_g\bigg)\Bigg) &= \sum_{m=0}^{g-1}\frac{1}{m!} (\partial^\pm)^p\bigg(x_1^m\cdot D(g-m)x_{g-m}\bigg)  \\
 &= \sum_{m=0}^{g-2}(\partial^\pm)^p\frac{x_1^m}{m!} D(g-m)x_{g-m} \  - (\partial^\pm)^p\frac{x_1^{g}}{({g-1})!} \\
 &= S_{g-p}\bigg(D(1)\cdot x_1,\dots, D(g-p)\cdot x_{g-p}\bigg) - p\ \frac{x_1^{g-p}}{({g-p})!}.
\end{align*}
\noindent Now we evaluate these series.
By induction, 
\begin{align*}
ev_{a^\pm}((\partial^\pm)^p R_g) &= \frac{2^{g-1}}{(2g-3)!!}\left(ev_{a^\pm}((\partial^\pm)^p y_g) +  \sum_{\substack{p_1+\dots+p_m = p \\ p_i \geq 0}} p! \prod_{i=1}^m (2g_i-1)\frac{d^\pm_{g_i,p_i}}{p_i!} \, \cdot \frac{Hur(\underline{g})}{m!}\right).
\end{align*}

\noindent By Lemma \ref{lem:bouquetrecursion} above (and Lemma \ref{lem:btilde-computation}), 
\begin{align*}
&S_{g-p}\bigg(D(1)\cdot b_1^\pm,\dots, D(g-p)\cdot b^\pm_{g-p}\bigg) \ - p\frac{(b_1^\pm)^{g-p}}{({g-p})!}  = \beta^\pm_{g-p} - p\frac{(b_1^\pm)^{g-p}}{({g-p})!} \\
&= \frac{2^{g-1}}{(2g-3)!!}\left(d^\pm_{g,p} +  \sum_{\substack{p_1+\dots+p_m = p \\ p_i \geq 0}} p! \prod_{i=1}^m (2g_i-1)\frac{d^\pm_{g_i,p_i}}{p_i!} \, \cdot \frac{Hur(\underline{g})}{m!}\right).
\end{align*}
Comparing the first terms concludes the proof.
\end{proof}

\begin{proof}[Proof of Theorem \ref{thm:bouquetformula}]
Let be $F_p(t,y) = \sum_{g\geq0}(2g-1)\ y_{g,p}\ t^{2g} \in \mathbb{Q}\llbracket y_{g,p}, t\rrbracket_{g\geq p\geq 0}$ such that $-y_{0,0} = 1$.
The coordinate transformation from Proposition \ref{prop:rec=main}, by the combinatorics of equation \ref{eq:closedpowerseries}, is equivalent to 
\[[t^{2g}]\ {F}_0(t,y) = \frac{1}{1-2g}[t^{2g}]\bigg(\exp{ \sum_{h\geq 1} (2h-1)!(-)^hx_h \ t^{2h} }\bigg)^{(1-2g)}.\]
\noindent Note that the above Lemma states that $F_p(t,d_{g,p}^\pm) = (\partial^\pm)^pF_0(t,y)_{y=a^\pm}$.
Thence we compute \[\partial^\pm \bigg( \sum_{h\geq 1} (2h-1)!(-)^hx_h \ t^{2h} \bigg) = -t^{2}.\]
The closed formula now follows from applying this derivation $p$-times to the closed expression\footnote{The same tactic would provide closed formulas for the non-parity generating function, but since the derivation then gives an eigenfunction instead of just a shift, this method provides a less clean result for non-parity.} of $F_0(t,y)$ in Theorem \ref{thm:maintheorem}, obtained by expressing the Bernoulli numbers in terms of $b_g^\pm$ and substituting $x_g = b_g^\pm$, and evaluating $x_g = b_g^\pm$ (and equivalently $y_{g,0} = y_g = a_g^\pm$). For the implicit formula one applies, under the above substitutions, the operator $\exp\epsilon\partial^\pm$ to the generating series for the $(-)^gb_g^\pm$. The conclusion then follows from evaluating $\sum_p \frac{\epsilon^p}{p!} F_p = e^{\epsilon\partial^\pm} F_0$ and the Main Theorem.
\end{proof}

\begin{remark}
Extracting coefficients in $\epsilon$, Theorem \ref{thm:bouquetformula} reads
\begin{align*}
 2^{-g} [t^{2g}]\ \mathcal{S}^\pm(t) &= \frac{1}{(2g)!}[t^{2g}]\Bigl( \mathcal{F}_0^{\pm}(t)\Bigr)^{2g},\\
  -\delta_{g,1} &= \frac{1}{(2g-1)!}[t^{2g}]\Bigl( \mathcal{F}_0^{\pm}(t)\Bigr)^{2g-1}\mathcal{F}^\pm_1(t),\\
  0 &= \frac{1}{(2g-1)!}[t^{2g}]\ (2g-1)\Bigl(\mathcal{F}_0^{\pm}(t)\Bigr)^{2g-2}(\mathcal{F}_1^\pm(t))^2 + \Bigl(\mathcal{F}_0^{\pm}(t)\Bigr)^{2g-1}\mathcal{F}^\pm_2(t), \\
  &\vdots
\end{align*}
\end{remark}

\section{Remarks on the non-parity case} \label{sec: non-weighted}
In this section we discuss an integral evaluation and vanishing result one encounters by studying the non-parity analogue of our proof of the Main Theorem. We start by building up the non-parity analogues of 
statements in previous sections.\\

\noindent Denote by $s_g := \sum_{c\geq0}p_{*}(\xi^{c}[\PP\oSQ_{g}(1^n)])$ the total Segre class of the cone $\oSQ_{g,n}$ over $\oM_{g,n}$ and let \[\beta_g := \int_{\oM_{g,1}}\psi_1^{g-1}s_g.\] 
\begin{lemma}[Segre and volumes]\label{lem:ag-recursion}
Let $g\in \ZZ_{\geq2}$ and let ${\rm Part}(g)_{m}$ be the set of length $m$ unordered partitions $\underline{g}=(g_1,\dots,g_m)$ of $g$, with entries $g_i \geq 1$. Then the following holds:
    \begin{equation*}
        a_{g}=\frac{(2g-3)!!}{2^{g-1}}\,\beta_{g}-\sum_{\substack{2\leq m\leq g,\\\underline{g}\in{\rm Part}(g)_{m}}}\frac{1}{m!} \left(\prod_{i=1}^{m}(2g_{i}-1)\,a_{g_{i}}\right)\,Hur(\underline{g}).
    \end{equation*}
\end{lemma}
\begin{proof}
As remarked, the computation of the spin-parity classes in Proposition \ref{prop:strata-recursion} holds for the non-parity classes as well. The same holds for the derivation from Lemma \ref{lem:xi-vanishing} of the non-trivial principal term and Corollary \ref{cor:boundaryterm} for the non-trivial boundary terms. Furthermore, the Hurwitz integrals in said boundary terms coincide in the spin-parity or classical case, since there we consider only genus $0$ curves. 
\end{proof}

In our computation of spin volumes we computed $\beta_g^\pm$ via spin Chiodo classes to solve the recursion on spin volumes. In this section we do the opposite: we can compare the volume recursion above with the classical result \ref{eq:ag-formula} of \cite{Sau} to compute the Chiodo integrals in $\beta_g$.

\subsection{Segre and Chiodo classes}\label{sec:non-parity}
Let $a = (1-2a_1\ , \dots , 1-2a_n)\in\ZZ^{n}$ be a vector of odd integers, $\cL\to \oC_{g,n}^{1/2}\to \oM_{g,n}^{1/2}$ the universal spin bundle 
and $\epsilon\colon \oM_{g,n}^{1/2}\to \oM_{g,n}$ the forgetful morphism.

Recall the definition of the Chiodo classes \[\Omega_{g,n}^{a}(t):=\epsilon_{*}\bigg(c_t\big(-R^{\bullet}\pi_{*}\cL( \sum_{i} a_i x_i)\big)\bigg).\]

\begin{theorem}\label{thm:Chiodo=Segre} For all stable $(g,n)$ the following identities hold
\begin{align}
     \Omega_{g,n}^{1^n}(2t)-2^{g-1}&=\sum_{\Gamma\in{{\rm Tree}_{g,n}}}\frac{t^{E(\Gamma)}}{|\Aut(\Gamma)|}\zeta_{\Gamma*}\left(\bigotimes_{v\in V(\Gamma)}s_{g(v)}(t)\right), \ \ {\rm and} \label{eq:Chiodo=Segre} \\
     s_{g}(t)&=\sum_{\Gamma\in{{\rm Tree}_{g,n}}}\frac{(-t)^{E(\Gamma)}}{|\Aut(\Gamma)|}\zeta_{\Gamma*}\left(\bigotimes_{v\in V(\Gamma)}\Omega_{g(v),n(v)}^{1^{n(v)}}(2t)-2^{g(v)-1}\right)
\end{align}
in $A^\bullet(\oM_{g,n})[t]$.
\end{theorem}

The aforementioned theorem 
was already proven in~\cite{HolPolSau} for the spin-parity classes. We stress that below we only re-illustrate their proof and identify the numerical modification to the non-parity case.
We refer the reader to~\cite[Section 5.5 and 6]{HolPolSau} for the original treatment of the arguments that will appear in the proof of Theorem \ref{thm:Chiodo=Segre}.

\begin{corollary}\label{cor:Chiodo-vanishing} 
Let $a=(1-2a_{i})\in\ZZ^{n}_{\geq0}$
be a vector of odd integers, 
and $[\cdot]_{k}$ denote the coefficient of $t^k$, then
\begin{equation*}
    [\Omega_{g,n}^{a}(t)]_{k}=0\ \ {\rm if}\ \ k>2g-1+\sum_{i=1}^{n}|a_{i}|.
\end{equation*}
\end{corollary}

\begin{proof}Using Lemma ~\ref{lem:xi-vanishing} we obtain that 
        $s_{g,k}=p_{*}\left(\xi^{k}[\PP\oSQ_{g}(1^{n})]\right)=0$
if $k>2g-1$. Therefore, plugging the top Segre class of $\oSQ_{g}$ in the formula~\ref{eq:Chiodo=Segre}
we can identify the top degree of the class $\Omega_{g,n}^{1^{n}}(2t)$ to be $2g-1$ (c.f. Lemma \ref{lem:btilde-computation}). Furthermore, for $a\in \ZZ^{n}$ as in the statement, we obtain the general result by a direct application of the insertion-shift property~\cite[Theorem 4.1ii)]{GiaLewNor}. 
\end{proof}

The vanishing property shown in the corollary above was already known for the spin Chiodo class
from the formula \ref{eq:spinchiodo=2Lambda}. To our knowledge no such degree bounds were known for 
\textit{non-parity} Chiodo classes. 

\begin{example}\label{eq:Chiodo-g=1}[genus $1$] 
Note that one can compute directly via the formula of~\cite{Chi} that \[\Omega_{1,1}^{1}(2t)=2-\lambda_{1}t.\]
\end{example}
\begin{theorem}\label{thm:chiodo-bg}
\[\int_{\oM_{g,1}}\psi_1^{g-1}\ \Omega_{g,1}^{1} = 2^{-1}\ (g-1)!\ (-1)^gb_g\]

\end{theorem}
\begin{proof}
By Corollary \ref{cor:Chiodo-vanishing}, the same arguments as in the proof of Lemma ~\ref{lem:bg-pm-computation} hold, so
    \begin{align*}
        \int_{\oM_{g,1}}\psi^{g-1}_{1}s_{g,2g-1}=
    \sum_{m=0}^{g-1}\frac{(-1)^{m}}{m!}
        \int_{\oM_{g,1}}&\psi_{1}^{g-m-1}[\Omega_{g-m,1}^{1}(2t)]_{2(g-m)-1}\ \times \label{eq:top-segre-integral-w/Chiodo} \\
        &\bigg(\int_{\oM_{1,1}}
               [\Omega_{1,1}^{1}(2t)]_{1}\bigg)^m.
    \end{align*}
Write $\gamma_g^\bullet:= \int_{\oM_{g,1}}\psi_1^{g-1}\ \Omega_{g,1}^\bullet(2),$ for $\bullet \in \{\emptyset, \pm\}$. Then \[\beta_g = S_g(\gamma_1,\dots, \gamma_g ).\]
By Lemma \ref{lem:ag-recursion}, then \[R_g(a_1\dots,a_g) = S_g(\gamma_1,\dots, \gamma_g ).\]
By Proposition \ref{prop:rec=main}, this is equivalent to \[a_g = M_g\bigg(\frac{\gamma_1}{D(1)}, \dots,\frac{\gamma_g}{D(g)}\bigg).\]
However, Equation \ref{eq:ag-formula} of Sauvaget states \[a_g = M_g\bigg(b_1, \dots,b_g\bigg),\] so by invertibility \[b_g = \frac{\gamma_g}{D(g)}.\]
To conclude, note that by definition $b_g^\pm = \frac{\gamma_g^\pm}{D(g)}$ and 
therefore \[\frac{b_g}{b_g^\pm} = \frac{\gamma_g}{\gamma_g^\pm}.\] Substitute the computation of $\gamma_g^\pm$ in Equation \ref{eq:chiodopm-bgpm} to conclude. 
\end{proof}

\noindent \textbf{{Proof of Theorem ~\ref{thm:Chiodo=Segre}}}.
Following~\cite{HolPolSau}, the strategy to prove this theorem is the introduction of an auxiliary 
moduli space. Let $P = (p_1,\dots,p_n)\in\ZZ^{n}_{\geq0}$ be a vector of non-negative integers and consider the universal spin bundle $\cL\to \oC_{g,n}^{1/2}\xrightarrow{\pi} \oM_{g,n}^{1/2}$ to define
\begin{equation*}
    \cH^0_{\cL({\Sigma_i} p_ix_{i})} 
    := \text{Spec}\ Sym^{\bullet}R^{1}\pi_{*}\cL(-\sum_{i=1}^{n}p_{i}x_{i}).
\end{equation*}
As proven in~\cite[Section 6]{HolPolSau}, 
the space $\cH^0_{\cL({\Sigma_i} p_ix_{i})}$
is a cone over $\oM_{g,n}^{1/2}$ parametrizing global sections of $\cL( \text{\tiny$\sum_i$} p_ix_{i})$.
We describe below how, in the case $P=(1,0^{n-1})$, the push forward of the Segre class
of $\cH^0_{\cL(x_{1})}$
to $\oM_{g,n}$ can be expressed both
in terms of Chiodo classes $\Omega_{g,n}^{1^{n}}$ and
in terms of Segre classes of the rank $1$ cones $\oSQ_{g}(1^{n})$ and $\oSQ_{g}(-1,1^{n-1})$, the latter two of which are in turn related. \\

On one hand we have the following
\begin{lemma}\label{lem:segre=Chiodo} Assume that $P$ contains at least one strictly-positive $p_i$, then following equality holds in $A^\bullet(\oM_{g,n})$:
\begin{equation*}
    \epsilon_{*}s(\cH^0_{\cL({\Sigma_i} p_ix_{i})})(t)=\frac{\Omega_{g,n}^{1^n}(t)}{\prod_{i=1}^{n}\prod_{j=0}^{p_{i}-1}\left(1+\frac{1-2p_{i}+2j}{2}t\psi_{i}\right)}.
\end{equation*}

\end{lemma}
\begin{proof} On generic points $(C,L) \in \oC_{g,n}^{1/2}$ we have $h^{0}(C,L) \leq 1$ and so $h^{0}(C,L(- \text{\tiny$\sum_i$} p_ix_{i}))=0$ and by Serre duality also $h^{1}(C,L( \text{\tiny$\sum_i$} p_ix_{i}))=0$, thus $R^1\pi_{*}\cL( \text{\tiny$\sum_i$} p_ix_{i})=0$. Then we can compute the following on $\oM_{g,n}^{1/2}$.
    \begin{align*}
   s\left(\cH^0_{\cL({\Sigma_i} p_ix_{i})}\right)(t)&= 
    s\left((R^1\pi_{*}\cL(-\sum_i\  p_ix_{i}))^{\vee}\right)(t)\\
   &=c\left((R^{1}\pi_{*}\cL(-\sum_i\  p_ix_{i}))^{\vee}\right)(t)^{-1} \\
    &=c\left(R^{0}\pi_{*}\cL(\sum_i\  p_ix_{i})\right)(t)^{-1}\\
    &=c\left(-R^\bullet\pi_{*}\cL(\sum_i\  p_ix_{i})\right)(t).
\end{align*}
Here, the first equality holds by definition and the third equality follows from Grothendieck-Serre duality.
Therefore, after pushing forward via $\epsilon$ we obtain
\begin{equation*}
    \epsilon_{*}s\left(\cH^0_{\cL({\Sigma_i} p_ix_{i})}\right)(t)
    =\Omega_{g,n}^{(1-2P)}(t).
\end{equation*}
Finally, using the shifting property for Chiodo classes~\cite[Theorem 4.1 ii)]{GiaLewNor} we obtain the stated formula.
\end{proof}

On the other hand, in~\cite[Section 6.4]{HolPolSau}, 
the authors give an explicit description of
the irreducible components of $\cH^0_{\cL(p_1x_{1})}$ and deduced the following. 
Let be $s_{g}^{1}(t):= s(\oSQ_{g}(-1,1^{n-1}))(t)$ and\footnote{These classes contain essentially the Segre classes of the components.} \begin{align*}
    \widetilde{s}^{0}_{g}(t)&:=\sum_{{{\rm Tree}_{g,n}}}\frac{t^{|E(\Gamma)|}}{|\Aut(\Gamma)|}\zeta_{\Gamma*}\left(\bigotimes_{v\in V(\Gamma)}s_{g(v)}(t)\right), \\
    \widetilde{s}^{1}_{g}(t)&:=\sum_{{{\rm Tree}_{g,n}}}\frac{t^{|E(\Gamma)|}}{|\Aut(\Gamma)|}\zeta_{\Gamma*}\left(s^{1}_{g(v_{0})}(t)\bigotimes_{v\neq v_{0}}s_{g(v)}(t)\right).
\end{align*} In the specific case $P=(1,0^{n-1})$, we have
\begin{lemma}~\cite[Proposition 6.6]{HolPolSau}
The following holds:
\begin{equation}\label{eq:segre=tilde-segre}
    \epsilon_{*}s\left(\cH^0_{\cL(x_{1})}(2t)\right)=\frac{1}{2}\left(\widetilde{s}^{0}_{g}(t)+\widetilde{s}^{1}_{g}(t)\right).
\end{equation}
\qed

\end{lemma}
\vspace{7pt}

 To deduce~\cite[Theorem 1.10]{HolPolSau} the authors
prove an independent relation between the spin-parity versions of the Segre classes on the right-hand side of Equation \ref{eq:segre=tilde-segre}. We will follow the same path by adapting 
these relations in the case of non-weighted classes. So, on a third hand we have
\begin{lemma}\label{cor:Segre-tilde-relation}
The following identity holds:
\begin{equation*}
     (1-t\psi_{1})\widetilde{s}^{1}_{g}(t)=(1+t\psi_{1}) \widetilde{s}^{0}_{g}(t) +2^{g}.
\end{equation*}
\end{lemma}
\begin{proof}
Analogous to \textit{loc.cit.}, Proposition \ref{prop:strata-recursion} gives for $g\geq 0$,
        \begin{equation}
            (1-t\psi_1)s_g^{1}(t)= (1+t\psi_1) s_{g}(t)+  \sum_{\substack{\Gamma\in {\rm Tree}_{g,1+n} \\ \Gamma\  {\rm is\ back\text{-}bone}}} \frac{(-t)^{E(\Gamma)}2^{g(v_0)}}{|{\rm Aut}(\Gamma)|}\zeta_{\Gamma*}\left([\oM_{g(v_0),n(v_0)}]\bigotimes_{v\in V_{\rm out}} s_{g(v)}(t)\right).
        \end{equation} 
By substitution one obtains the desired equality up to the degree $0$ term.
This is computed below.
\end{proof}

\begin{lemma}[Degree 0]\label{lem:s0-s1-deg-0} Following the notation of~\cite{HolPolSau},
    we define for all $(g,n,m)\in\NN^{3}$ such that $2g-2+n+m>0$ the function $d'(g,m)$ to be the degree of
    $
       p\colon \PP\oSQ_{g}((-1)^{m},1^{n},\{0\})\to \oM_{g,n+m}. 
    $
    Then we have 
    \begin{equation*}
        d'(g,m)=2^{2g-1}+(-1)^{m+1}2^{g-1}.
    \end{equation*}
\end{lemma}
\begin{proof} The arguments in the proof of~\cite[Proposition 5.6]{HolPolSau} hold verbatim to show
    \begin{align*}
        d'(g,m+2)&=d'(0,3)d'(g,m)+d'(0,2)d'(g,m+1) \\
        d'(g+1,m)&=d'(1,1)d'(g,m)+d'(1,0)d'(g,m+1).
    \end{align*}
    In genus $0$ we have no odd spin structures and so the computations in \emph{loc.cit.} show that 
    \begin{equation*}
        d'(0,m)=m\bmod2.
    \end{equation*}
    Furthermore, in genus $1$, the space $\PP\oSQ_{1}(1^{n},\{0\})$ (that is, $m=0$) generically 
    parametrizes holomorphic sections 
    of the canonical bundle in $g=1$ whose associated spin structure is the unique odd one, and so $d'(1,0)=1$. Finally, points in $\PP\oSQ_{1}(-1,3)$ generically parametrize non-trivial $2$-torsion divisors (which in particular are even) on elliptic curves. Therefore, we obtain $d'(1,1)=3$. In total we have 
    \begin{align*}
        d'(g,m+2)&=d'(g,m) \\
        d'(g+1,m)&=3d'(g,m)+d'(g,m+1).
    \end{align*}
    Together with the initial conditions in genus $0$ we have the desired result.
\end{proof}

\begin{proof}[Proof of Theorem ~\ref{thm:Chiodo=Segre}]\label{proof-nonparity} We use Lemma ~\ref{lem:segre=Chiodo} and Equation ~\ref{eq:segre=tilde-segre}
to show 
\begin{equation*}
    \frac{\Omega_{g,n}^{1^{n}}(2t)}{1-t\psi_{1}}= \frac
    {1}{2}\left(\widetilde{s}^{0}_{g}(t)+\widetilde{s}^{1}_{g}(t)\right).
\end{equation*}
Then, clearing the denominators and using Corollary ~\ref{cor:Segre-tilde-relation}, we obtain 
\begin{align*}
    \Omega_{g,n}^{1^{n}}(2t)&=\frac{1}{2}((1-t\psi_{1})\widetilde{s}_{g}^{1}(t)-(1+t\psi_{1})\widetilde{s}_{g}^{0}(t) +2\widetilde{s}_{g}^{0}(t)) \\
    &=\frac{1}{2}(2^{g}+2\widetilde{s}_{g}^{0}(t)) \\
    &=2^{g-1}+\sum_{\Gamma\in {\rm Tree}_{g,n}}\frac{t^{E(\Gamma)}}{|\Aut(\Gamma)|}\zeta_{\Gamma*}\left(\bigotimes_{v\in V(\Gamma)}s_{g(v)}(t)\right).
\end{align*} One Moebius inverts these equations to obtain the second line in the theorem.
\end{proof}

\begin{remark} As we saw above, some of the statements used throughout the proof of Theorem ~\ref{thm:Chiodo=Segre} hold in
greater generality than the vector $P=(1,0^{n-1})$. 
It would be interesting to follow a similar 
approach also for Segre classes of $\cH^0_{\cL({\Sigma_i} p_ix_{i})}$ of other vectors $P$ and $r-$spin structures to deduce more general relations.
We intend to investigate this approach in future work.

\end{remark}

\bibliographystyle{alpha}
\bibliography{bibliography.bib}

\end{document}